\documentclass{amsart}
\usepackage[applemac]{inputenc}
\usepackage[french]{babel}
\usepackage[T1]{fontenc}
\usepackage{hyperref}
\usepackage{amssymb}
\usepackage[matrix,arrow]{xy}
\usepackage{mathrsfs}
\newtheorem{thm}{Theorem}[section]
\newtheorem{defi}[thm]{Définition}
\newtheorem{prop}[thm]{Proposition}

\newtheorem{lem}[thm]{Lemme}
\newtheorem{coro}[thm]{Corollaire}

\theoremstyle{remark}

\newtheorem{rem}[thm]{Remarque}
\newtheorem{ex}[thm]{Exemple}

\newcommand{\colim}{\mathrm{colim}}

\newcommand{\C}{\mathbb{C}}

\newcommand{\qq}{\mathbb{Q}}

\newcommand{\ZZ}{\mathbb{Z}}
\newcommand{\Z}{\mathbb{Z}}

\newcommand{\Spec}{\mathrm{Spec}}

\newcommand{\Spa}{\mathrm{Spa}}

\newcommand{\HH}{\mathrm{H}}

\newcommand{\ocal}{\mathcal{O}}
\newcommand{\oscr}{\mathscr{O}}

\title{Faisceaux equivariants sur $\mathbb{P}^1$ et faisceaux automorphes}
\author{Vincent Pilloni}
\date{}
\email{}
\begin{document}

\selectlanguage{french}
\begin{abstract} Inspiré par  \cite{pan2021locally},  on définit et étudie un foncteur qui associe à des faisceaux équivariants sur $\mathbb{P}^1$ des faisceaux automorphes sur les courbes modulaires. 
\end{abstract}
\maketitle

\tableofcontents

\section{Introduction}
Cette note est inspirée par le travail \cite{pan2021locally}. Dans ce travail, Lue Pan utilise la méthode de Sen (et sa généralisation par Berger-Colmez \cite{MR2493221}, \cite{MR3552018}) pour  décrire le sous-faisceau des vecteurs localement analytiques du faisceau structural des courbes modulaires perfectoides.  

Dans cet article, on introduit   un foncteur qui associe à des faisceaux équivariants pour l'action infinitésimale de $\mathrm{GL}_2$ sur des ouverts  de  $\mathbb{P}^1$  (qui est l'espace des périodes de Hodge-Tate pour les courbes modulaires) des faisceaux  sur la tour de courbes modulaires. Ce foncteur envoie les  faisceaux équivariants cohérents de rang $1$  sur $\mathbb{P}^1$   sur les  faisceaux  des formes modulaires  classiques. Il envoie  les  faisceaux équivariants  cohérents    de rang $1$  sur les strates de Bruhat sur les faisceaux de formes modulaires surconvergentes (\cite{MR3265287}, \cite{MR3097946}, \cite{MR3609197}, \cite{Boxer-Pilloni}).  En dérivant le foncteur, on obtient  "naturellement"  la décomposition de Hodge-Tate de la cohomologie \'etale des courbes modulaires. 
En appliquant finalement ce foncteur  au faisceau dual du faisceau des opérateurs différentiels étendus sur $\mathbb{P}^1$,  on obtient   la   description des vecteurs localement analytiques dans la cohomologie complétée donnée dans la section 4 de \cite{pan2021locally}. 

Introduisons quelques notations. Soit $G = \mathrm{GL}_2$, $T \subseteq  B$ un tore maximal inclus dans un Borel. On pose $U$ le radical unipotent de $B$. On note enfin $\mathcal{FL} = B\backslash G = \mathbb{P}^1$.   On fixe un sous-groupe compact ouvert $K^p \subseteq G(\mathbf{A}_f^p)$. Pour tout sous-groupe compact ouvert $K^p \subseteq G(\qq_p)$, on note $X_{K_pK^p} \rightarrow \Spa(\C_p, \ocal_{\C_p})$ la courbe modulaire compactifiée  de niveau $K_pK^p$, vue comme espace adique sur $\C_p$. On note $X_{K^p} = \lim_{K_p} X_{K_pK^p}$ la courbe perfectoide de niveau modéré $K^p$ construite dans \cite{scholze-torsion}. 
On considère  alors  le diagramme (où l'application $\pi_{K_p}$ est la projection vers un niveau fini, et l'application $\pi_{HT}$ est l'application des périodes de Hodge-Tate):
\begin{eqnarray*}
\xymatrix{ &X_{K^p} \ar[ld]_{\pi_{K_p}} \ar[rd]^{\pi_{HT}}& \\
X_{K_pK^p}&&\mathcal{FL}}
\end{eqnarray*}

Pour tout ouvert $U$ de $\mathcal{FL}$, on note  $\mathbf{Coh}_{\mathfrak{g}}(U)$ la catégorie des faisceaux cohérents $\mathfrak{g}$-équivariants sur $U$, où $\mathfrak{g}$ est l'algèbre de Lie de $G$.  

Soit $\mathfrak{n}^0 \subseteq \oscr_{\mathcal{FL}} \otimes_{\C_p} \mathfrak{g}$ la sous-algèbre de Lie nilpotente universelle. 
Pour tout objet $\mathscr{F} \in \mathbf{Coh}_{\mathfrak{g}}(U)$, on possède par définition un morphisme $\mathfrak{g} \rightarrow \mathrm{End}_{\C_p}(\mathscr{F})$ et celui-ci induit un morphisme $\mathfrak{n}^0 \rightarrow \underline{\mathrm{End}}_{\oscr_{U}}(\mathscr{F})$.  La cohomologie de $\mathfrak{n}^0$, notée $\HH^i(\mathfrak{n}^0, \mathscr{F})$, est  la cohomologie du complexe placé en dégré $0$ et $1$ :
$$\mathscr{F} \rightarrow \mathscr{F} \otimes_{\oscr_U} \mathfrak{n^0}^\vee.$$
On note $\mathbf{Coh}_{\mathfrak{g}}(U)^{\mathfrak{n}^0} \subseteq \mathbf{Coh}_{\mathfrak{g}}(U)$ la sous-catégorie pleine des objets vérifiant $\HH^0(\mathfrak{n}^0, \mathscr{F}) = \mathscr{F}$.   Les objets de $\mathbf{Coh}_{\mathfrak{g}}(U)^{\mathfrak{n}^0}$ sont d'ailleurs exactement les objets de $\mathbf{Coh}_{\mathfrak{g}}(U)$  tels que l'action infinitésimale de $G $ s'étend en  une action de l'algèbre des opérateurs différentiels  étendus (voir \cite{MR733805}):  $$\mathcal{D} = \oscr_{\mathcal{FL}} \otimes_{\C_p} \mathcal{U}(\mathfrak{g})/ \mathfrak{n}^0 \oscr_{\mathcal{FL}} \otimes_{\C_p} \mathcal{U}(\mathfrak{g}).$$ 
On note $\hat{\oscr} = (\pi_{HT})_\star \oscr_{X_{K^p}}$. C'est un faisceau $G(\qq_p)$-équivariant.  On note $\oscr^{sm} \subseteq \hat{\oscr}$ le sous-faisceau des vecteurs lisses pour l'action de $G(\qq_p)$, il admet la description suivante : 
$$\oscr^{sm} = \colim_{K_p} \oscr_{K_p}, ~\textrm{où} ~~ \oscr_{K_p} = (\pi_{HT})_\star (\pi_{K_p})^{-1} \oscr_{X_{K_pK^p}}$$
On introduit alors  le  foncteur:
\begin{eqnarray*} 
VB:   \mathbf{Coh}_{\mathfrak{g}}(U) &\rightarrow& \mathbf{Mod}(\oscr^{sm}\vert_U) \\
\mathscr{F} & \mapsto & \colim_{K_p}(\mathscr{F} {\otimes}_{\oscr_{U}} \hat{\oscr}\vert_U)^{K_p}
\end{eqnarray*}
Explicitons la définition. La colimite porte sur les sous-groupes ouverts compacts de $G(\qq_p)$.  Pour tout ouvert $V$ affinoide de $U$, on a $$VB(\mathscr{F})(V) =  \colim_{K_p} \HH^0(K_p, \mathscr{F}(V) {\otimes}_{\oscr_{\mathcal{FL}(V)}} \hat{\oscr}(V)).$$ On montre que l'action de $K_p$ diagonale est bien défini pour tout $K_p$ suffisamment petit.  
\begin{thm}\label{thm-intro}
\begin{enumerate} 
\item Pour tout $\mathscr{F} \in \mathbf{Coh}_{\mathfrak{g}}(U)$, $VB(\mathscr{F})$ est un faisceau localement libre de rang fini de $\oscr^{sm}\vert_{U}$-modules. 


\item Pour tout $\mathscr{F} \in \mathbf{Coh}_{\mathfrak{g}}(U)^{\mathfrak{n^0}}$, on possède un isomorphisme canonique  $$VB(\mathscr{F}) {\otimes}_{\oscr^{sm}\vert_U} \hat{\oscr}\vert_U \rightarrow \mathscr{F} {\otimes}_{\oscr_{U}} \hat{\oscr}\vert_U$$   et le foncteur $VB$ restreint à la catégorie $\mathbf{Coh}_{\mathfrak{g}}(U)^{\mathfrak{n^0}}$ est exact.

\item Le foncteur $VB$ est dérivable, et pour tout  $\mathscr{F} \in \mathbf{Coh}_{\mathfrak{g}}(U)$,  on a $\mathrm{R}^iVB(\mathscr{F}) = VB(\HH^i(\mathfrak{n}^0, \mathscr{F}))$.

\end{enumerate}
\end{thm}

\begin{rem}  Vu le point $(2)$ du théorème, on pense donc à la condition d'être annulé par $\mathfrak{n}^0$ comme à une condition d'admissibilité.  
\end{rem}

\begin{ex}[\ref{subsec-ex1}] Tout caractère algébrique  $\kappa$ du tore $T$ de $G$ fournit un faisceau $G$-équivariant sur $\mathcal{FL}$ que nous notons $\oscr^\kappa_{\mathcal{FL}}$. On peut le voir comme un objet  de la catégorie $\mathbf{Coh}_{\mathfrak{g}}(\mathcal{FL})^{\mathfrak{n}^0}$ et on note $\omega^{\kappa,sm}_{\mathcal{FL} }= VB(\oscr^\kappa_{\mathcal{FL}})$. C'est un faisceau $G(\qq_p)$-équivariant sur $\mathcal{FL}$. On vérifie qu'on possède un isomorphisme de $G(\qq_p)$-représentations:
$$\HH^i(\mathcal{FL}, \omega^{\kappa,sm}_{\mathcal{FL}}) = \colim_{K_p} \HH^i(X_{K_pK^p}, \omega^{\kappa}_{K_p})$$ où ce dernier groupe est la limite inductive des  cohomologies des faisceaux des formes modulaires de poids $\kappa$ sur les courbes modulaires $X_{K_pK^p}$. 
\end{ex}

\begin{ex}[\ref{subsec-ex3}] Considérons la décomposition de Bruhat $\mathcal{FL} = U_{w_0} \cup \{ \infty\}$, où $U_{w_0} = B\backslash Bw_0 B$ pour $w_0$ un générateur du groupe de Weyl,  et $\infty = B\backslash B$. A tout caractère $\kappa \in X^\star(T)_{\C_p}$, on peut associer un faisceau inversible  $\oscr^\kappa_{U_{w_0}}$ dans la catégorie $\mathbf{Coh}_{\mathfrak{g}}(U_{w_0})$, qui est $B(\qq_p)$-équivariant. On note  $\omega^{\kappa,sm}_{U_{w_0}}= VB(\oscr^\kappa_{U_{w_0}})$. Le groupe de cohomologie $$\HH^1_c(U_{w_0}, \omega^{\kappa,sm}_{U_{w_0}})$$ est une représentation  de $B(\qq_p)$. En prenant les invariants par $U(\ZZ_p)$ et la partie de pente finie, on retrouve l'espace de cohomologie locale qui est considéré dans la théorie de Coleman supérieure pour le $\HH^1$ (voir \cite{Boxer-Pilloni}).

De même,  on peut construire un faisceau $\oscr^\kappa_{\infty}$ inversible dans la catégorie $\mathbf{Coh}_{\mathfrak{g}}(\oscr_{\mathcal{FL},\infty}))$ (la limite des catégories $\mathbf{Coh}_{\mathfrak{g}}(U)$ pour $U$ parcourant les voisinages de $\infty$) qui est $B(\qq_p)$-équivariant.  On note  $\omega^{\kappa,sm}_{\infty}= VB(\oscr^\kappa_{\infty})$. Le groupe de cohomologie $$\HH^0(\{\infty\}, \omega^{\kappa,sm}_{\infty})$$ est une représentation de $B(\qq_p)$. Cet espace se relie à l'espace des formes surconvergentes de poids $\kappa$ (et donc à la théorie de Coleman pour le $\HH^0$). Si $\kappa \in X^\star(T)$, on possède une suite longue reliant cohomologie locale et cohomologie classique:
$$ 0 \rightarrow \HH^0(\mathcal{FL}, \omega^{\kappa,sm}_{\mathcal{FL}})  \rightarrow \HH^0(\{\infty\}, \omega^{\kappa,sm}_{\infty}) \rightarrow \HH^1_c(U_{w_0}, \omega^{\kappa,sm}_{U_{w_0}}) \rightarrow \HH^1(\mathcal{FL}, \omega^{\kappa,sm}_{\mathcal{FL}}) \rightarrow 0$$
\end{ex}

\begin{ex}[\ref{subsec-ex2}] Le point $3)$ du théorème couplé au théorème de comparaison primitif  de \cite{MR3090230} implique la décomposition de Hodge-Tate de la cohomologie \'etale des courbes modulaires avec coefficients. Soit $V$ une représentation algébrique du groupe $G$ et $\mathcal{V}_{K_p,et}$ le faisceau associé sur le site pro-Kummer-étale de $X_{K_pK^p}$. On a une suite exacte  de Hodge-Tate :
$$ 0 \rightarrow \HH^1(\mathcal{FL}, VB(\oscr_{\mathcal{FL}} \otimes V)) \rightarrow \colim_{K_p} (\HH^1_{proket}(X_{K_pK^p}, \mathcal{V}_{K_p,et})\otimes_{\qq_p} \C_p) $$ $$ \rightarrow  \HH^0(\mathcal{FL}, \mathrm{R}^1VB(\oscr_{\mathcal{FL}} \otimes V)) \rightarrow 0.$$
 \end{ex}

\bigskip

Passons à présent à la description des vecteurs localement analytiques $\oscr^{la} \subseteq \hat{\oscr}$. 
Pour tout $n \geq 0$, on note $G_n$ le sous-groupe affinoide des éléments  de $G$ qui se réduisent sur $1$ modulo $p^n$ et $\oscr_{G_n}$ son anneau de fonctions.  C'est un $G_n$-bi-module. Pour tout $f \in \oscr_{G_n}$ et $g \in G_n$, on pose 
\begin{eqnarray*}
g \star_1 f(-) & =& f(g^{-1}-),\\
g\star_2 f(-) &=& f(-g).
\end{eqnarray*}
On considère le faisceau $\oscr_{G_n} \hat{\otimes} \oscr_{\mathcal{FL}}$. On possède une action $\star_3$ de $G$ sur $\oscr_{\mathcal{FL}}$. Munis de l'action $\star_{1,3}$ (composée des $\star_{1}$ et $\star_{3}$), le faisceau $\oscr_{G_n} \hat{\otimes} \oscr_{\mathcal{FL}}$ est un faisceau $G_n$-équivariant.  Il possède une seconde action $\star_2$ de $G_n$ qui est $\oscr_{\mathcal{FL}}$-linéaire. 
Pour tout $f \in \oscr_{G_n} \hat{\otimes} \oscr_{\mathcal{FL}}$ (vu comme une fonction $G_n \rightarrow \oscr_{\mathcal{FL}}$), on a
 \begin{eqnarray*}
g \star_{1,3} f(-) & =&  g f(g^{-1}-), ~\forall g \in G_n,\\
g\star_2 f(-) &=& f(-g),~\forall g \in G_n.\\
\end{eqnarray*}
L'action $\star_{1,3}$ se dérive et induit une action de $\mathfrak{g}$ et donc une action ($\oscr_{\mathcal{FL}}$-linéaire)
$$ \mathfrak{n}^0 \rightarrow \underline{\mathrm{End}}_{\oscr_{\mathcal{FL}}}(\oscr_{G_n} \hat{\otimes} \oscr_{\mathcal{FL}})$$
On définit   $\mathcal{C}^{n-an}  = \HH^0(\mathfrak{n}^0, \oscr_{G_n} \hat{\otimes} \oscr_{\mathcal{FL}})$.  La fibre du faisceau $\mathcal{C}^{n-an}$ en un point $x \in \mathcal{FL}$ est une certaine complétion (dépendant de $n$) du dual du
module de Verma universel  $\C_p \otimes_{\mathcal{U}(\mathfrak{n}_x)} \mathcal{U}(\mathfrak{g})$. Ce faisceau est stable sous les actions $\star_{1,3}$ et  $\star_{2}$ considérées précédemment.  L'action $\star_{1,3}$ induit  aussi une action  $\star_{hor}$ du   Cartan horizontal  $\mathfrak{h} \hookrightarrow \mathfrak{b}^0/\mathfrak{n}^0$.
En passant à la limite, on définit l'espace des germes de fonctions analytiques au voisinage de $1$ : $ \oscr_{G,1} = \colim_n \oscr_{G_n}$, et on définit le  faisceau $$\mathcal{C}^{la} = \colim_n \mathcal{C}^{n-an} = \HH^0(\mathfrak{n}_0, \oscr_{G,1}  \hat{\otimes}_{\C_p} \oscr_{\mathcal{FL}}).$$  
Le formule $h \star_{1,2,3} f(-)  = hf(h^{-1}- h)$ définit une action de $G(\qq_p)$ sur les faisceaux  $\oscr_{G,1}  \hat{\otimes}_{\C_p} \oscr_{\mathcal{FL}}$ et $\mathcal{C}^{la}$.
La formule définissant le foncteur $VB$ s'étend naturellement à tous les faisceaux $G_n$-équivariants,  et par passage à la limite à $\oscr_{G,1}  \hat{\otimes}_{\C_p} \oscr_{\mathcal{FL}}$ et $\mathcal{C}^{la}$.
Par définition, $\oscr^{la} =  \colim_{K_p} (\oscr_{G,1} \hat{\otimes}_{\C_p} \hat{\oscr})^{K_p}= VB ( \oscr_{G,1}  \hat{\otimes}_{\C_p} \oscr_{\mathcal{FL}})$.  Le théorème qui suit est une formulation agréable de certains résultats de la section 4 de \cite{pan2021locally}. Cette formulation est inspirée par la section 6 de \cite{MR3552018}. On peut le voir comme une généralisation du théorème \ref{thm-intro} pour un faisceau non-cohérent.

\begin{thm}[\cite{pan2021locally}]
On a $\oscr^{la} = VB( \mathcal{C}^{la})$ et on  possède un isomorphisme canonique, $\hat{\oscr}$-linéaire, $\mathfrak{g}$ et $G(\qq_p)$-équivariant:
$$ \Psi :  \oscr^{la} \hat{\otimes}_{\oscr^{sm}} \hat{\oscr} \rightarrow \mathcal{C}^{la} \hat{\otimes}_{\oscr_{\mathcal{FL}}} \hat{\oscr}$$
\end{thm}

\begin{rem} En formulant le théorème ainsi, on réalise que  les calculs de la cohomologie des sous-algèbres de Lie de $\mathfrak{g}$ sur $\oscr^{la}$ (par exemple ceux de la section 5 de \cite{pan2021locally}) peuvent se transporter en des calculs sur le faisceau $\mathcal{C}^{la}$ et se réduisent donc à des questions de théorie des représentations.  Ce point de vue est développé dans la dernière section \ref{sectiononcohomology}   de cet article. 
\end{rem}

 Expliquons comment est définie l'application $\Psi$. On possède une application orbite:
$\oscr^{la} \rightarrow \oscr_{G,1} \hat{\otimes}_{\C_p} \hat{\oscr}$ qui à une section localement analytique $f \in \hat{\oscr}$, associe un germe de fonction analytique $g \mapsto g f$. Le théorème affirme  donc que cette application orbite se factorise à travers le sous-espace $\mathcal{C}^{la} \hat{\otimes}_{\oscr_{\mathcal{FL}}} \hat{\oscr} $ et que le linéarisé de l'application orbite, pris au dessus des vecteurs lisses, induit un isomorphisme sur $\mathcal{C}^{la} \hat{\otimes}_{\oscr_{\mathcal{FL}}} \hat{\oscr} $.

 Expliquons un peu mieux ce résultat  et décrivons aussi quelles sont toutes les actions. Pour tout $n$, on montre qu'il existe un sous-groupe ouvert compact $K_{p}(n) \subseteq G_n(\qq_p)$, un faisceau de $\oscr_{K_{p}(n)}$-modules ${\oscr}^{n-an}_{K_p(n)}$ (qui est sur chaque ouvert affinoide standard de $\mathcal{FL}$ de la forme $B \hat{\otimes}_{\C_p} \oscr_{K_{p}(n)}$ pour un espace de Banach $B$ sur $\C_p$), et un isomorphisme canonique, $\hat{\oscr}$-linéaire: 

$$\Psi_n : {\oscr}^{n-an}_{K_p(n)} \hat{\otimes}_{\oscr_{K_{p}(n)}} \hat{\oscr} = \mathcal{C}^{n-an} \hat{\otimes}_{\oscr_{\mathcal{FL}}} \hat{\oscr}$$

Faisons agir $K_p(n)$ via son action  sur $\hat{\oscr}$ sur le membre de gauche, et faisons le agir   diagonalement sur $\mathcal{C}^{n-an} \otimes_{\oscr_{\mathcal{FL}}} \hat{\oscr}$ (via $\star_{1,3}$ sur $\mathcal{C}^{n-an}$ et l'action naturelle sur le faisceau $\hat{\oscr}$).  Le morphisme $\Psi_n$ est alors $K_p(n)$-équivariant. 

On obtient donc que $VB(\mathcal{C}^{n-an}) = {\oscr}^{n-an}_{K_p(n)} \hat{\otimes}_{\oscr_{K_{p}(n)}} \hat{\oscr}^{sm}$ puis que $VB(\mathcal{C}^{la}) = \colim_n {\oscr}^{n-an}_{K_p(n)} = \oscr^{la}$ et $\Psi = \colim_n \Psi_n$.

On possède également une action diagonale de $G(\qq_p)$ sur $\oscr^{la} \hat{\otimes}_{\oscr^{sm}} \hat{\oscr}$ et une action diagonale de $G(\qq_p)$ sur $\mathcal{C}^{la} \hat{\otimes}_{\oscr_{\mathcal{FL}}} \hat{\oscr}$ (via   $\star_{1,2,3}$  sur $\mathcal{C}^{la}$ et l'action naturelle sur $\hat{\oscr}$). L'isomorphisme $\Psi$ est $G(\qq_p)$-équivariant pour ces actions.

On possède  une action de $\mathfrak{g}$ sur $\mathcal{C}^{la}$ par l'action $\star_2$, ainsi qu'une action  de $\mathfrak{g}$ sur $\oscr^{la}$ (obtenue en dérivant l'action de $G(\qq_p)$). Le morphisme $\Psi$ est $\mathfrak{g}$-équivariant pour ces actions.

Mentionnons enfin que le Cartan horizontal $\mathfrak{h} = \HH^0(\mathcal{FL}, \mathfrak{b}^0/\mathfrak{n}^0)$ agit sur $\mathcal{C}^{la}$ (resp. sur $\mathcal{C}^{n-an}$) et que l'action de $\mathfrak{h}$  commute avec  l'action $\star_{1,3}$ de $\mathfrak{g}$, (resp. $G_n$). Il en résulte qu'on possède une action $\star_{hor}$ de $\mathfrak{h}$ sur $\ocal^{la}$.
Voici une description plus directe de cette action. Comme on l'a déjà vu, l'action de $G(\qq_p)$ se dérive et  induit une action de $\mathfrak{g}$ sur ${\oscr}^{la}$. On l'étend linéairement en une action de  $\mathfrak{g}^0$.  On voit alors à l'aide du théorème que l'action de $\mathfrak{n}^0$ est triviale. Il en résulte alors une action de $\mathfrak{b}^0/\mathfrak{n}^0$. Sa restriction à $\mathfrak{h}$ est alors $-\star_{hor}$.

\subsection{Remerciements}  Je remercie Juan-Esteban Rodriguez pour de nombreux échanges sur le sujet.  Dans sa thèse, il étudie le cas des variétés de Shimura générales.  Ce texte est une version  détaillée d'un exposé donné lors d'un groupe de travail à l'ENS de Lyon sur le travail de Lue Pan. Je remercie  les participants au groupe de travail, et particulièrement  A. Bode, G. Boxer, G. Dospinescu, Joaquin Rodrigues,  Juan-Esteban Rodriguez et Olivier Taibi.  Cette recherche a bénéficié du soutient de l'ERC-2018-COG-818856-HiCoShiVa.

\section{Faisceaux équivariants sur $\mathcal{FL}$}
\subsection{Notations} Soit $k$ un corps complet  pour une valuation de rang $1$ étendant la valuation $p$-adique. 
Soit $S$ un espace adique localement de type fini sur $\Spa(k, \ocal_k)$.
 On utilise la théorie des espaces adiques de Huber (\cite{MR1306024}). 
Soit $G = \mathrm{GL}_2$ le groupe analytique sur $\Spa(k, \ocal_k)$ définit par $G(\Spa(R, R^+)) = \mathrm{GL}_2(R)$. Soit $B$ son Borel supérieur, $T$ le tore diagonal, et $U$ le radical unipotent de $B$. Soit $w_0$ le générateur du groupe de Weyl. On note $\mathfrak{g} = \mathrm{Lie}(G)$, $\mathfrak{b} = \mathrm{Lie}(B)$, $\mathfrak{n} = \mathrm{Lie}(U)$, $\mathfrak{h} = \mathrm{Lie}(T)$. On note $X^\star(T)$ le groupe des caractères de $T$ et $X^\star(T)^+$ le cone des caractères dominants pour $B$. 
On pose $\mathcal{FL} = B \backslash G$.  On note $\pi : G \rightarrow \mathcal{FL}$ la projection et on appelle $\infty$ le point $\pi(1)$ et $U_{w_0} = \pi ( w_0 U)$.
Pour tout point $x \in \mathcal{FL}$, on note $B_x = x^{-1}Bx = \mathrm{Stab}(x)$, $U_x$ le radical unipotent de $B_x$, $T_x = B_x/U_x$.  On pose $\mathfrak{b}_x = \mathrm{Lie}(B_x)$, $\mathfrak{n}_x = \mathrm{Lie}(U_x)$, $\mathfrak{h}_x = \mathrm{Lie}(T_x)$. 

\subsection{Faisceaux équivariants}\label{sect-f-eq}  
\subsubsection{Faisceaux cohérents}[\cite{stacks-project}, def. 01BV]
Soit $(X, \oscr_X)$ un espace annelé. Un faisceau $\mathscr{F}$ de $\oscr_X$-modules est cohérent si:
\begin{enumerate}
\item il est de type fini,
\item pour tout ouvert $U \subseteq X$,  et tout morphisme $\phi :  \oscr_U^i \rightarrow \mathscr{F}\vert_U$, le faisceau $\mathrm{Ker}(\phi)$ est de type fini.
\end{enumerate}

Si $\oscr_X$ lui-même est cohérent, alors un faisceau de $\oscr_X$-modules est cohérent si et seulement si il est de présentation finie (\cite{stacks-project}, lem. 01BZ et lem. 01BW). 

On note $\mathbf{Coh}(\oscr_X)$ la catégorie abélienne des $\oscr_X$-modules cohérents. Si $X$ est un espace adique et $\oscr_X$ est le faisceau structural on note aussi cette catégorie $\mathbf{Coh}(X)$. 
\subsubsection{Faisceaux équivariants}
Soit $S$ un espace adique localement de type fini sur $\Spa(k, \ocal_k)$.  
Soit $H$ un $S$-espace adique en groupes localement de type fini,  agissant à droite sur un $S$-espace adique $X$. On note $m : H \times_S X \rightarrow X$ l'action et $p : H \times_S X \rightarrow X$ la seconde projection. Un faisceau cohérent $H$-équivariant sur $X$ est un faisceau cohérent $\mathscr{F}$ munit d'un isomorphisme $act : m^\star \mathscr{F} \rightarrow p^\star \mathscr{F}$ vérifiant les propriétés suivantes. 
\begin{enumerate}
\item (L'identité agit trivialement)  Soit $e : S \rightarrow H$ la section neutre. On a $(e\times Id_X)^\star act = Id_{\mathscr{F}}$. 
\item(L'action est associative)  Pour $i=1,2,3$, on définit des applications $m_i : H \times_S H \times_S X \rightarrow H \times_S X$ par $m_1((g_1,g_2,x)) = (g_1 g_2,x)$, $m_2((g_1,g_2,x)) = (g_2, x g_1)$ and $m_3( (g_1,g_2,x)) = (g_1,x)$. 
Le diagramme suivant est commutatif :
\begin{eqnarray*}
\xymatrix{ m_1^\star m^\star \mathscr{F} \ar[rrr]^{m_1^\star act} \ar@{=}[d]  && &  m_1^\star p^\star \mathscr{F} \ar@{=}[d]  \\
m_2^\star m^\star \mathscr{F} \ar[r]^{m_2^\star act} & m_2^\star p^\star \mathscr{F} \ar@{=}[r]& m_3^\star m^\star \mathscr{F} \ar[r]^{m_3^\star act} & 
m_3^\star p^\star \mathscr{F}}
\end{eqnarray*}
\end{enumerate}

On note $\mathbf{Coh}_{H}(S)$ la catégorie des faisceaux cohérents $H$-équivariants sur $S$. 
\subsubsection{Faisceaux $G$-équivariants}
Le groupe $G$ agit sur lui-même à droite par translation à gauche via $(g, h) \mapsto g \star_1 h = g^{-1} h$ et translation à droite par $(g, h) \mapsto g \star_2 h = hg$. Ces actions induisent  deux structures de faisceaux $G$-équivariant sur $\oscr_G$.

Notons  $\mathbf{Rep}(B)$ la catégorie des représentations de $B$ sur les $k$-espaces vectoriels de dimension fini. 
On définit un foncteur $\mathbf{Rep}(B) \rightarrow \mathbf{Coh}_G(\mathscr{FL})$ qui envoie une représentation $(V,\rho)$ sur le faisceau $\mathcal{V}$ donné par $\mathcal{V} =  (\pi_\star \oscr_G {\otimes} V)^{B, \star_1 \otimes \rho}$. L'action $\star_2$ sur $\oscr_G$ induit la structure $G$-équivariante sur $\mathcal{V}$. Notons $i_{\infty} : \{ \infty \} \rightarrow \mathcal{FL}$ l'inclusion. Comme $B$ est le stabilisateur du point $\infty$,  on définit  un  foncteur $\mathbf{Coh}_G(\mathscr{FL}) \rightarrow \mathbf{Rep}(B)$ qui envoie $\mathscr{F}$ sur $i_{\infty}^\star \mathscr{F}$. 

\begin{prop}\label{prop-RepB} Le foncteur  
\begin{eqnarray*}
i_{\infty}^\star : \mathbf{Coh}_G(\mathscr{FL}) & \rightarrow & \mathbf{Rep}(B) \\
\mathscr{F} & \mapsto & i_{\infty}^\star \mathscr{F}
\end{eqnarray*}
est une équivalence de catégorie. Un quasi-inverse est donné par le foncteur $V \mapsto \mathcal{V}$. 
\end{prop}
\begin{proof} La restriction de l'isomorphisme d'action  $m^\star \mathscr{F} = p^\star \mathscr{F}$ à $G \times \{\infty\}$ fournit un isomorphisme $\pi^\star \mathscr{F} = \oscr_{G} \otimes_{k} i_\infty^\star \mathscr{F}$ dont on vérifie qu'il est équivariant pour l'action de $B$ via $\star_1 \otimes \rho$ et l'action de $G$ via $\star_2$.
\end{proof}

\begin{lem} Si $V \in \mathbf{Rep}(G)$, $\mathcal{V} = V {\otimes}_{k} \oscr_{\mathcal{FL}}$. 
\end{lem}

\begin{proof} Considérons l'application $\pi_\star \oscr_G {\otimes} V \rightarrow \pi_\star \oscr_G {\otimes} V$ donnée par  $f(g) \mapsto \rho(g)^{-1} f(g)$. Cette application échange l'action $\star_1 \otimes \rho$ et l'action $\star_1\otimes \mathrm{Id}$ de $B$. 
\end{proof}

\begin{ex}\label{exemple-basique} \begin{enumerate}
\item Pour tout caractère algébrique $\kappa = (k_1,k_2) \in X^\star(T)$, on note $\oscr_{\mathcal{FL}}^\kappa$ le faisceau associé à la représentation  $w_0\kappa$.  C'est un faisceau de degré $k_1-k_2$. Si $\kappa \in X^\star(T)^+$, $$\HH^0(\mathcal{FL}, \oscr_{\mathcal{FL}}^\kappa)  = \{ f : G \rightarrow \mathbb{A}_1,~f(bg) = w_0\kappa(b) f(g)\}$$ est la représentation de plus haut poids $\kappa$. 
\item On note $\mathfrak{g}^0 = \oscr_{\mathcal{FL}} \otimes \mathfrak{g}$ le faisceau associé à $\mathfrak{g}$ munit de la représentation adjointe.  
\item On note  $\mathfrak{n}^0 \subseteq \mathfrak{b}^0\subseteq \mathfrak{g}^0$ les sous-faisceaux $G$-équivariant, donnés par la condition $f \in \mathfrak{g}^0(U)$ est une section de $\mathfrak{b}^0(U)$ (resp. $\mathfrak{n}^0(U)$) si et seulement si   pour tout $x \in U$, $f(x) \in \mathfrak{b}_x$ (resp. $\mathfrak{n}_x$). Ainsi $\mathfrak{b}^0$ (resp. $\mathfrak{n}^0$) est  le faisceau associé à $\mathfrak{b}$ (resp. $\mathfrak{n}$) muni de la représentation adjointe par l'inverse du foncteur de la proposition \ref{prop-RepB}. 
\item On a un isomorphisme $\mathfrak{h} = \HH^0(\mathscr{FL}, \mathfrak{b}^0/\mathfrak{n}^0)$. L'image de $\mathfrak{h} \rightarrow  \mathfrak{b}^0/\mathfrak{n}^0$ est appelé Cartan horizontal. 
\item    Le faisceau localement libre $\mathfrak{g}^0/\mathfrak{b}^0$ s'identifie canoniquement au faisceau tangent $\mathcal{T}_{\mathcal{FL}}$. En effet, on possède un morphisme $\mathfrak{g} \rightarrow \mathcal{T}_{\mathcal{FL}}$ qui s'obtient en dérivant l'action de $G$ sur $\oscr_{\mathcal{FL}}$ (voir la section \ref{sect-dérivation}). On peut l'étendre linéairement en un morphisme  $\mathfrak{g}^0 \rightarrow \mathcal{T}_{\mathcal{FL}}$.  Ce morphisme passe au quotient en un isomorphisme de faisceau $G$-équivariants  $\mathfrak{g}^0/\mathfrak{b}^0 \rightarrow \mathcal{T}_{\mathcal{FL}}$ comme on le voit en comparant les fibres en le point $\infty$ des deux faisceaux. 
\end{enumerate}
\end{ex}

\subsubsection{Faisceaux $H$-équivariants}
On considère maintenant un sous-groupe ouvert  $H$  de $G$.  
Si $U$ est un ouvert de $\mathscr{FL}$ stable sous $H$, on possède une catégorie $\mathbf{Coh}_H(U)$ dont les objets sont les  faisceaux cohérents, $H$-équivariants  sur $U$.
Supposons qu'on possède un point $x \in U(k)$ tel que le morphisme $x H \rightarrow U$ induise un isomorphisme $x.\mathrm{Stab}_{H}(x) \backslash H \rightarrow U$.  Dans cette situation on a:

\begin{prop}\label{prop-rep-torseurs} Le foncteur $\mathbf{Coh}_{H}(U) \rightarrow \mathbf{Rep}(\mathrm{Stab}_H(x))$, donné par $\mathscr{F} \mapsto i_x^\star \mathscr{F}$ est une équivalence de catégorie.
\end{prop}
\begin{proof} C'est identique à la preuve de la proposition \ref{prop-RepB}.
\end{proof}





\subsubsection{Faisceaux $\mathfrak{g}$-équivariants}

On note $\hat{G}$ la complétion formelle de $G$ en la section neutre. Si $I$ désigne l'idéal d'augmentation de $\oscr_G$, alors $\oscr_{\hat{G}} = \lim_n \oscr_G/I^n$.  Notons $\mathfrak{g} = \mathrm{Hom}_k(I/I^2, k)$ l'algèbre le lie de $G$ et $\mathcal{U}(\mathfrak{g})$ l'algèbre enveloppante de $\mathfrak{g}$. Elle s'identifie à l'algèbre des operateurs differentiels invariants à gauche sur $G$ (voir \cite{MR2015057}, I, 7.10).   On possède un isomorphisme $\oscr_{\hat{G}} = \mathrm{Hom}_{k}( \mathcal{U}(\mathfrak{g}), k)$. 

On peut considérer la catégorie $\mathbf{Coh}_{\mathfrak{g}}(U)$ des faisceaux cohérents qui sont $\mathfrak{g}$-équivariant. 
Un objet de cette catégorie est un faisceau coherent $\mathscr{F}$, muni d'un morphisme d'action (vérifiant les conditions déduites de celles du numéro \ref{sect-f-eq}):
$act : \mathscr{F} \rightarrow \mathscr{F}\otimes_{k} \oscr_{\hat{G}}$. Par dualité, on possède un morphisme $\mathfrak{U}(\mathfrak{g}) \otimes \mathscr{F} \rightarrow \mathscr{F}$ qui est déterminé par sa restriction à $\mathfrak{g}$.

Un objet de $\mathbf{Coh}_{\mathfrak{g}}(U)$   est donc simplement la donnée d'un faisceau cohérent  $\mathscr{F}$ équipé d'un morphisme $ \mathfrak{g} \rightarrow \mathrm{End}_{k}( \mathscr{F})$, vérifiant :

\begin{enumerate}
\item  Pour tout ouvert $V$ de $U$,  pour tout $x, y \in \mathfrak{g}$,  et toute section $f \in \mathscr{F}(V)$, on a: $xyf - yxf = [x,y]f$,
\item  Pour tout ouvert $V$ de $U$, pour tout $x \in \mathfrak{g}$,  toute section $f \in \mathscr{F}(V)$, et toute section $h \in \oscr_{\mathcal{FL}}(V)$, on a $x(h f) = x(h)f + hx(f)$. 
\end{enumerate}

\begin{lem}\label{lem-extension} Tout objet de $\mathbf{Coh}_{\mathfrak{g}}(U)$ est un faisceau localement libre de rang fini sur $U$.
\end{lem}
\begin{proof} Soit $\mathscr{F} \in \mathbf{Coh}_{\mathfrak{g}}(U)$. Soit $\mathscr{F}^{tors}$ le sous-faisceau de torsion. Il suffit de voir que $\mathscr{F}^{tors} =0$. Soit $V = \Spa(A,A^+)$ un ouvert affinoide de $U$. Le $A$-module $\mathscr{F}^{tors}(V)$  est de type fini et de torsion. Soit $I \subseteq A$ son idéal annulateur. Par hypothèse, $I \neq 0$. De plus pour tout $g \in \mathfrak{g}$, $f \in \mathscr{F}^{tors}(V)$ et $i \in I$ on a $g(if) = g(i) f + i g(f)$. Appliquant cette identité  à $i \in I^2$, on voit que $g(f)$ est de torsion. Appliquant ceci à $i \in I$, on déduit que $g(i).f=0$. On voit donc que l'idéal $I$ de $A$ est non nul, et stable sous les dérivations de $A$. C'est donc que $I=A$. 
\end{proof}

Soit $\mathscr{F} \in  \mathbf{Coh}_{\mathfrak{g}}(U)$. On peut étendre l'action de $\mathfrak{g}$ en une action de $\oscr_{U} \otimes_k \mathfrak{g}$ et le sous-faisceau $\mathfrak{b}^0$ agit linéairement, c'est à dire qu'on possède un morphisme $\mathfrak{g}$-équivariant  $\mathfrak{b}^0 \rightarrow \underline{\mathrm{End}}_{\oscr_{U}}(\mathscr{F})$.

\begin{defi} On dit qu'un objet $\mathscr{F}$ de $\mathbf{Coh}_{\mathfrak{g}}(U)$ est de poids $\chi \in X^\star(T)_{k}$ si le morphisme précédent se factorise en un morphisme $\mathfrak{b}^0/\mathfrak{n}^0 \rightarrow \underline{\mathrm{End}}_{\oscr_{U}}(\mathscr{F})$ et si $\mathfrak{h}$ agit par à travers  le caractère $\chi$.
\end{defi}

\begin{ex} Le faisceau $\oscr_{\mathcal{FL}}^\kappa$ est de poids $w_0\kappa$. 
\end{ex}

\begin{rem} Les objets de poids nul de $\mathbf{Coh}_{\mathfrak{g}}(U)$ sont les fibrés à connexion intégrable. En effet, dans ce cas le morphism $\oscr_{U} \otimes \mathfrak{g} \rightarrow \underline{\mathrm{End}}_{k}(\mathscr{F})$ se factorise en un morphisme :
$\mathcal{T}_{U} \rightarrow  \underline{\mathrm{End}}_{k}(\mathscr{F})$. Les conditions $(1)$ et $(2)$ assurent que ce morphisme est équivalent à la donnée d'une connexion intégrable $\nabla : \mathscr{F} \rightarrow \mathscr{F} \otimes_{\oscr_{U}} \Omega^1_{U/k}$. 
\end{rem}

\subsubsection{Cohomologie de  l'algèbre de Lie nilpotente}




On note $\mathbf{Coh}_{\mathfrak{g}}(U)^{\mathfrak{n}^0}$ la sous-catégorie  de $\mathbf{Coh}_{\mathfrak{g}}(U)$ des faisceaux tués par $\mathfrak{n}^0$. 
On dispose de foncteurs: 
\begin{eqnarray*}
\mathbf{Coh}_{\mathfrak{g}}(U) &\rightarrow& \mathbf{Coh}_{\mathfrak{g}}(U)^{\mathfrak{n}^0} \\
\mathscr{F} &\mapsto& \mathscr{F}^{\mathfrak{n}^0} = \HH^0(\mathfrak{n}^0, \mathscr{F})\\
\mathscr{F} &\mapsto& \mathscr{F}_{\mathfrak{n}^0} = \mathrm{coker}( \mathscr{F} \rightarrow \mathscr{F} \otimes (\mathfrak{n}^0)^\vee) = \HH^1(\mathfrak{n}^0, \mathscr{F})
\end{eqnarray*}

\subsubsection{Faisceaux $H$-équivariants et faisceaux $\mathfrak{g}$-équivariants}\label{sect-dérivation}
Pour tout sous-groupe ouvert $H$ de $G$ et tout ouvert $U$ qui est $H$-équivariant, on possède un foncteur :
$\mathbf{Coh}_{H}(U) \rightarrow \mathbf{Coh}_{\mathfrak{g}}(U)$ obtenu en "dérivant" l'action de $H$. A tout objet $\mathscr{F}$, munit de son action $H$-équivariante, on associe l'objet $\mathscr{F}$ muni du morphisme 
$\mathscr{F} \rightarrow \mathscr{F} \otimes \oscr_H \rightarrow \mathscr{F} \otimes \oscr_{\hat{G}}$ déduit de l'application $\oscr_H \rightarrow \oscr_{\hat{G}}$. On déduit que le foncteur $\mathbf{Coh}_{H}(U) \rightarrow \mathbf{Coh}_{\mathfrak{g}}(U)$ est pleinement fidèle dès que $H$ est connexe. 

\begin{rem} Voici une autre manière de décrire l'action de $\mathfrak{g}$ sur un faisceau $\mathscr{F}$ qui est $H$-équivariant.  
On note $H^{(1)}$ le fermé de $H$ défini par le carré de l'idéal d'augmentation $I_H^2$. On a $\oscr_{H^{(1)}} = k \oplus \epsilon \mathfrak{g}^\vee$.  
On restreint le morphisme $m_H$ en un morphisme $m_{H^{(1)}} : H^{(1)} \times U \rightarrow U$ et la projection $p_H$ en une projection $p_{H^{(1)}} : H^{(1)} \times  U \rightarrow U$
Considérons le morphisme  composé $ \mathscr{F} \rightarrow m_{H^{(1)}}^\star \mathscr{F} \rightarrow p_{H^{(1)}}^\star \mathscr{F} = \mathscr{F} + \epsilon \mathfrak{g}^\vee \otimes \mathscr{F}$.  La projection sur  $\mathfrak{g}^\vee \otimes \mathscr{F}$ fournit  une action de $\mathfrak{g}$ sur $\mathscr{F}$. 
\end{rem}


On va maintenant montrer qu'on peut toujours "intégrer" une action infinitesimale sur un ouvert quasi-compact. Commençons par quelques préliminaires. Soit $\mathfrak{g}^+ = \mathrm{M}_2(\ocal_k)$ l'algèbre de Lie du groupe $\mathrm{GL}_2/\ocal_k$. Pour tout $n \geq 0$, on  défini un sous-groupe $G_n \hookrightarrow G$ dont les points sont les matrices entières de reductions triviale modulo $p^n$. Le spectre formel de $\oscr_{G_n}^+$ est un schéma formel $p$-adique en groupes lisse. Soit $I_{G_n}^+$ l'idéal d'augmentation de $\oscr_{G_n}^+$. On note $\mathfrak{g}_n^+ = \mathrm{Hom} (I_{G_n}^+/(I_{G_n}^+)^2, \oscr_k)$. On possède un morphisme $\mathfrak{g}_n^+ \rightarrow \mathfrak{g}$ qui induit un isomorphisme $\mathfrak{g}_n^+ = p^n \mathfrak{g}^+$. Notons $\mathcal{U}(\mathfrak{g}_n^+) \hookrightarrow \mathcal{U}(\mathfrak{g})$ l'algèbre enveloppante sur $\ocal_k$ de $\mathfrak{g}_n^+$. C'est la sous $\ocal_k$-algèbre  engendrée par $\mathfrak{g}_n^+$. Notons  $Dist(\oscr_{G_n}^+) = \colim_m \mathrm{Hom}_{\ocal_k} ( \oscr_{G_n}^+/(I_{G_n}^+)^m, \ocal_k)$. Pour $n \geq 1$,  c'est la sous $\ocal_k$-algèbre de $\mathcal{U}(\mathfrak{g})$ engendrée par les éléments $$\big\{ \frac{x^m}{m !}, ~m \in \ZZ_{\geq 0}, x  \in \mathfrak{g}_n^+\}.$$
On possède alors la suite d'inclusions pour tout $n \geq 1$ :
$$ \mathcal{U}(\mathfrak{g}_n^+) \hookrightarrow Dist(\oscr_{G_n}^+)  \hookrightarrow  \mathcal{U}(\mathfrak{g}^+_{n-1}).$$

\begin{prop}\label{prop-integration-action} Soit $U$ un ouvert quasi-compact, et soit $\mathscr{F} \in  \mathbf{Coh}_{\mathfrak{g}}(U)$. Alors il existe un sous-groupe ouvert $H$ de $G$ tel que $\mathscr{F}$ soit dans l'image du foncteur $\mathbf{Coh}_{H}(U) \rightarrow \mathbf{Coh}_{\mathfrak{g}}(U)$. 
\end{prop}
\begin{proof} On va trouver $H$ de la forme $G_n$ pour $n \in \ZZ_{\geq 1}$.   Vue la pleine fidélité de $ \mathbf{Coh}_{H}(U) \rightarrow  \mathbf{Coh}_{\mathfrak{g}}(U)$ pour $H$ connexe, on peut supposer que $U = \Spa(A,A^+)$ est affinoide. On pose $M = \mathscr{F}(U)$. On fixe un sous $A^+$-module ouvert et borné $M^+$ de $M$.  
On affirme qu'il existe $n \in \ZZ_{\geq 0}$ tel que le morphisme $\mathfrak{g} \rightarrow \mathrm{End}(M)$ induise un morphisme $p^n \mathfrak{g}^+ \rightarrow \mathrm{End}_{\oscr_k}(M^+)$. En effet, on commence par trouver $n_0$ tel que $p^{n_0} \mathfrak{g}^+(A^+) \subseteq A^+$ (ce qui est possible car $A^+$ est topologiquement de type fini sur $\ocal_k$). Puis, on peut trouver $n \geq n_0$ tel que pour une famille finie génératrice $\mathcal{M}$ d'éléments de $M^+$ comme $A^+$-module, on ait $p^n\mathfrak{g}^+ ( \mathcal{M}) \subseteq M^+$. 
On possède alors un morphisme $M^+ \rightarrow M^+ \otimes_{\ocal_k} \mathrm{Hom}_{\ocal_k}( \mathcal{U}(p^n\mathfrak{g}^+), \ocal_k)$. On possède des morphismes $\mathrm{Hom}_{\ocal_k}( \mathcal{U}(p^n\mathfrak{g}^+), \ocal_k) \rightarrow  \lim_k \oscr^+_{G_{n+1}}/I_{G_{n+1}}^k \rightarrow \oscr^+_{G_{n+2}}$. Ceci induit donc  un morphisme $M \rightarrow M \otimes  \oscr_{G_{n+2}}$. 
\end{proof}

\subsubsection{Faisceaux $(\mathfrak{g},M)$-équivariants}\label{faisceauxGM}
Soit $M$ un sous-groupe localement profini de $G(\qq_p)$. Soit $U$ un ouvert stable sous $M$.  Soit $m_M : M \times U \rightarrow U$ la projection, et $p_M : M \times U \rightarrow U$ l'action.  Soit un $\mathscr{F}$ un  objet de $ \mathbf{Coh}_{\mathfrak{g}}(U)$. On suppose que $\mathscr{F}$  est aussi $M$-équivariant : on possède un isomorphisme $act_M : m_M^\star \mathscr{F} = p^\star_M \mathscr{F}$ qui vérifie la condition de cocycle du numéro \ref{sect-f-eq}. Considérons maintenant les  trois conditions de compatibilité suivantes entre les actions de $\mathfrak{g}$ et de $M$. Pour tout ouvert $V$ de $U$, on a (fonctoriellement en $V$): 
   \begin{enumerate} 
\item Pour tout $(f,a,m) \in \mathscr{F}(V) \times \oscr_{\mathcal{FL}}(V) \times M$, avec $mV \subseteq V$,  on a  $m(af)= m(a) mf$,
\item Pour tout $(f,g,m) \in \mathscr{F}(V) \times \mathfrak{g} \times M$, avec $mV \subseteq V$ on a $\mathrm{Ad}_m(g) f = m( g(m^{-1} f))$. 
\item Si $V$ est quasi-compact, il existe un sous-groupe ouvert $H \subseteq G$ tel que l'action de $\mathfrak{g}$  sur $\mathscr{F}\vert_V$ s'intègre en une action $act_H$ de $H$ et tel que les actions  $act_H$  et $act_M$ coincident sur $M \cap H$.
\end{enumerate}
\begin{defi}
On dit que le faisceau $\mathscr{F}$ est $(\mathfrak{g},M)$-équivariant, si les actions de $\mathfrak{g}$ et $M$ vérifient les conditions $(1)$ et $(2)$. On note $\mathbf{Coh}_{(\mathfrak{g}, M)}(U)$ la catégorie des faisceaux $(\mathfrak{g},M)$-équivariants. 

On dit que le faisceau $\mathscr{F}$ est $(\mathfrak{g},M)$-fortement équivariant, si les actions de $\mathfrak{g}$ et $M$ vérifient les conditions $(1)$,$(2)$ et $(3)$. On note $\mathbf{Coh}_{(\mathfrak{g}, M)^f}(U)$ la catégorie des faisceaux $(\mathfrak{g},M)$-fortement équivariants.
\end{defi}

\begin{rem} La condition $(3)$  signifie que  l'action de $M$ est localement analytique, et  que l'action infinitesimale de $M$ est induite par celle de  $\mathfrak{g}$.  Même si cette condition est   naturelle, nous allons rencontrer des exemples où elle n'est pas satisfaite. 
\end{rem}

\begin{rem} Soit $H$ un sous-groupe ouvert de $G$ et $M$ un sous-groupe localement profini de $H \cap G(\qq_p)$. On possède alors un foncteur  naturel : $\mathbf{Coh}_{H}(U) \rightarrow \mathbf{Coh}_{(\mathfrak{g}, M)^f}(U)$.  
\end{rem}

\subsection{Faisceaux inversibles $\mathfrak{g}$-équivariants} On se propose de déterminer les faisceaux inversibles $\mathfrak{g}$-équivariants sur $\mathcal{FL}$, $\mathcal{FL} \setminus \{\infty\}$, et au voisinage du point $\{\infty\}$. 

\subsubsection{Faisceaux inversibles $\mathfrak{g}$-équivariants sur $\mathcal{FL}$}
Soit $Z$ le centre de $G$ et $T^{der}$ le tore maximal diagonal du groupe dérivé de $G$. On a un morphisme surjectif $T^{der} \times Z \rightarrow T$ de noyau $\mu_2$.  Ce morphisme induit un isomorphisme
$\mathfrak{h} = \mathfrak{h}^{der} \oplus \mathfrak{g}^{ab}$. 
On a $X^\star(T)  \hookrightarrow X^\star(T^{der}) \times X^\star(Z)$. Le conoyau de ce morphisme est un $2$-groupe.

Pour tout caractère $\kappa \in X^\star(T)$, on possède un faisceau inversible  $\oscr_{\mathcal{FL}}^\kappa$ (voir l'exemple \ref{exemple-basique}). Pour tout $\mu \in X^\star(Z)_{k}$ vu comme un caractère $\mu : \mathfrak{g} \rightarrow k$, on note $k(\mu)$ le $k$-espace vectoriel de dimension $1$, avec  action de $\mathfrak{g}$ par $\mu$. 

Soit $\kappa \in X^\star(T^{der}) \times X^\star(Z)_{k} \hookrightarrow X^\star(T)_k$. Ecrivons $\kappa = \kappa' + \mu$ avec $\kappa \in X^\star(T)$ et $\mu \in X^\star(Z)_{k}$. On définit $\oscr_{\mathcal{FL}}^{\kappa} = \oscr_{\mathcal{FL}}^{\kappa'} \otimes_{k} k(\mu) $. 

\begin{lem} Le faisceau $\mathfrak{g}$-équivariant $\oscr_{\mathcal{FL}}^{\kappa} $ ne dépend que de $\kappa$ (comme la notation l'indique), et pas   de la décomposition $\kappa = \kappa' + \mu$. 
\end{lem}

\begin{proof}
Si $\oscr_{\mathcal{FL}}^{\kappa} \otimes_{k} k(\mu) \simeq  \oscr_{\mathcal{FL}}^{\kappa'} \otimes_{k} k(\mu')$,  alors les deux faisceaux ont même degré, ce qui signifie que $\kappa-\kappa' \in X^\star(Z)$. En tordant, on se ramène au cas où $\kappa =\mu =0$. L'action de  $\mathfrak{g}$ sur $\HH^0(\mathcal{FL}, \oscr^{\kappa'}_{\mathcal{FL}}) \simeq k$ se fait via $\kappa'$ et donc $\kappa'=-\mu'$. 
\end{proof}

  
\begin{lem} Tout faisceau $\mathfrak{g}$-équivariant inversible sur $\mathcal{FL}$ est isomorphe à un faisceau $\oscr_{\mathcal{FL}}^{\kappa}$ pour un unique $\kappa  \in X^\star(T_{der}) \times X^\star(Z)_{k}$.\end{lem}

\begin{proof} En tordant par un faisceau convenable $\oscr_{\mathcal{FL}}^{\kappa}$, $\kappa \in X^\star(T)$, on  se ramène au cas où le  faisceau sous-jacent est le faisceau structural $\oscr_{\mathcal{FL}}$.  L'action de $\mathfrak{g}$ sur $\HH^0(\mathcal{FL}, \oscr_{\mathcal{FL}})$ se fait via un caractère $\mu : \mathfrak{g} \rightarrow k$. 
\end{proof}

On détermine à présent les objets  inversibles de $\mathbf{Coh}_{(\mathfrak{g}, G(\qq_p))}(\mathcal{FL})$. Pour tout $\kappa \in X^\star(T)$, le faisceau $\oscr_{\mathcal{FL}}^\kappa$ est $G(\qq_p)$-equivariant et défini donc un objet de $\mathbf{Coh}_{(\mathfrak{g}, G(\qq_p))}(\mathcal{FL})$. 
Il en résulte que pour tout $\kappa  \in X^\star(T_{der}) \times X^\star(Z)_{k}$, $\oscr_{\mathcal{FL}}^\kappa$ est également  un objet de $\mathbf{Coh}_{(\mathfrak{g}, G(\qq_p))}(\mathcal{FL})$.


Pour tout caractère continu  $\chi : G(\qq_p) \rightarrow k^\times $, on note $k(\chi)$ le $k$-espace vectoriel de dimension $1$, avec  action de $G(\qq_p)$ par $\chi$. 


\begin{lem}  Tout faisceau inversible  $(\mathfrak{g},G(\qq_p))$-équivariant est isomorphe à un faisceau $\oscr_{\mathcal{FL}}^{\kappa} \otimes_k k(\chi) $ pour un unique $\kappa \in  X^\star(T_{der}) \times X^\star(Z)_{k}$ et un unique $\chi :  G(\qq_p) \rightarrow k^\times$.
\end{lem}
\begin{proof} En tordant par un faisceau $\oscr_{\mathcal{FL}}^{\kappa}$, on se ramène au cas où le faisceau $\mathfrak{g}$-equivariant sous-jacent est $\oscr_{\mathcal{FL}}$. L'action de $G(\qq_p)$ est donc donnée par un caractère $\chi$. 
\end{proof}



\subsubsection{Faisceaux inversibles $\mathfrak{g}$-équivariants sur $\mathcal{FL}\setminus \{\infty\}$}\label{sec-fsurcon}
On veut maintenant détérminer les faisceaux inversibles $\mathfrak{g}$-équivariants  sur l'ouvert $B\backslash Bw_0B = U_{w_0} = \hat{B} \backslash \hat{B} w_0 U $ 

Commençons par décrire deux constructions. 

\begin{enumerate}
\item Considérons le morphisme: 
$$ \pi :  \hat{U} \backslash \hat{B} w_0 U \rightarrow \hat{B} \backslash \hat{B} w_0 U.$$
Pour tout $\kappa \in X^\star(T)_k$, on pose :
$$ \oscr_{U_{w_0}}^\kappa(V) = \{ f : \pi^{-1}(V) \rightarrow \mathbb{A}^1, f(bg) = w_0\kappa(b) f(g)~\forall b \in \hat{B}\}.$$
L'action par translation à droite de $\hat{G}$ fournit une structure $\mathfrak{g}$-équivariante. 



\item Soit   $\xi \in \Omega^1_{U_{w_0}/\qq_p}$. On note $\oscr_{U_{w_0}}[\xi]$ le faisceau trivial muni de la connection $\nabla$ donnée par $\nabla(1) = \xi$. La connexion définit, via le morphisme $g \in \mathfrak{g} \mapsto \nabla_g$,  une structure $\mathfrak{g}$-équivariante.  
\end{enumerate}

On note  $\oscr_{U_{w_0}}^\kappa[\xi] = \oscr_{U_{w_0}}^\kappa \otimes_{\oscr_{U_{w_0}}} \oscr_{U_{w_0}}[\xi]$. 

\begin{lem} 
\begin{enumerate}
\item Le faisceau $\oscr_{U_{w_0}}^\kappa[\xi]$ est de poids $w_0\kappa$,
\item  Tout objet de rang $1$ de  $\mathbf{Coh}_{\mathfrak{g}}(U_{w_0})$  est isomorphe à un  faisceau $ \oscr_{U_{w_0}}^\kappa[\xi]$ pour un unique $\kappa$ et une forme différentielle $\xi$ unique à un scalaire près. 
\end{enumerate}
\end{lem}
\begin{proof} Le premier point est clair. Passons au second point. Rappelons pour commencer que tout faisceau inversible sur $U_{w_0}$ est trivial (\cite{MR4166998}, thm 4). 
On possède un morphisme $\mathfrak{g}$-équivariant $\mathfrak{b}^0  \rightarrow \underline{\mathrm{End}}_{\oscr_{U_{w_0}}}(\mathscr{F}) = \oscr_{U_{w_0}}$. Le morphisme $\mathfrak{n}^0 \rightarrow  \oscr_{U_{w_0}}$ est nul car le faisceau $\mathfrak{n}^0$ est de poids $-2 \rho$. Le morphisme $\mathfrak{b}^0 /\mathfrak{n}^0 \rightarrow  \oscr_{U_{w_0}}$ est $\mathfrak{g}$-équivariant et donc donné par un caractère $\kappa \in X^\star(T)_{k}$. En tordant par $\oscr_{U_{w_0}}^{-w_0\kappa}$ on peut supposer que $\kappa=0$. L'action de $\mathfrak{g}$ sur $\mathscr{F}$ induit donc un morphism :
$ \mathcal{T}_{U_{w_0}} \otimes \mathscr{F} \rightarrow  \mathscr{F}$  ou de façon équivalente, une connexion 
$\nabla :  \mathscr{F} \rightarrow \mathscr{F} \otimes_{\oscr_{U}} \Omega^1_{U_{w_0}/\qq_p}$.  Il en résulte que  le faisceau $\mathscr{F}$ est isomorphe à $\oscr_{U_{w_0}}^\kappa[\xi]$.  

\end{proof}

On s'intéresse à présent aux objets de $\mathbf{Coh}_{(\mathfrak{g}, B(\qq_p))}(U_{w_0})$.   Pour tout $\kappa \in  X^\star(T)_k$,    on peut définir une structure $(\mathfrak{g}, B(\qq_p))$-équivariante sur  le faisceau $\oscr^\kappa_{U_{w_0}}$ de la façon suivante. Pour toute section $f$ de $\oscr^\kappa_{U_{w_0}}$ (vue comme une fonction $f(bw_0 u)$ pour $b \in \hat{B}$ et $u \in U$), on pose :

\begin{enumerate}
\item     $u'f( b w_0 u) =  \kappa(b) f(bw_0  u u')$ pour tout $u' \in U(\qq_p)$,
\item $t f(bw_0 u) = f(bw_0 t^{-1} u t)$ pour tout $t \in T(\qq_p)$. 
\end{enumerate}


Pour tout caractère continu $\lambda : T(\qq_p) \rightarrow k^\times$, on note $k(\lambda)$ le $k$-espace vectoriel munit de l'action de $B(\qq_p)$ par $\lambda$. %


\begin{prop} Tout objet inversible de $\mathbf{Coh}_{(\mathfrak{g}, B(\qq_p))}(U_{w_0})$ est isomorphe à un faisceau $\oscr^\kappa_{U_{w_0}} \otimes_k k(\lambda)$ pour un unique $\kappa \in  X_\star(T)_k$ et un unique  $\lambda : T(\qq_p) \rightarrow  k^\times$.
\end{prop}

\begin{proof} Soit  $\mathscr{F}$ un objet de $\mathbf{Coh}_{(\mathfrak{g}, B(\qq_p))}(U_{w_0})$ de rang $1$.  Quitte à tordre, on peut le supposer de poids $0$. On peut aussi fixer un isomorphisme $\mathscr{F} \simeq \oscr_{U_{w_0}}$. L'action de $\hat{G}$ est donnée par une connection $\nabla$ et   l'action de $B(\qq_p)$  sur $\mathscr{F}$ commute avec la connection.  Pour tout $m \in B(\qq_p)$ on a $m. 1 = f_m . 1$ où $f_m$ est une fonction inversible sur $U_{w_0}$, donc une constante  (considérer le polygone de Newton de $f_m$). Il en résulte que l'action de $B(\qq_p)$ sur $1$ est donnée par un caractère $\lambda : B(\qq_p) \rightarrow k^\times$.  
Soit $z$ la coordonnée naturelle sur $U_{w_0}$.  Posons $\nabla(1) =  \xi(z) = f(z)dz $.  L'invariance par l'action de $B(\qq_p)$ signifie que $ \xi(z) = bf(z) d bz $ pour tout $b \in B(\qq_p)$. Si on l'applique à la matrice unipotente standard, on trouve que $f(z+1) = f(z)$, donc  $\xi =  rdz$ avec $r \in k$. Si on l'applique à  $t = \mathrm{diag}(t,1) \in B(\qq_p)$, 
on trouve que $rdz = trdz$. Il en résulte que   $r=0$.   Ainsi, les objets inversibles de poids zero de $\mathbf{Coh}_{(\mathfrak{g}, B(\qq_p))}(U_{w_0})$ sont bien de la forme $\oscr_{U_{w_0}} \otimes_k k(\lambda)$.
\end{proof}

\begin{rem}\label{rem-fortement} Soit $\kappa \in X^\star(T^{der}) \times X^\star(Z)_k \hookrightarrow X^\star(T)_k$. On possède un faisceau $(\mathfrak{g}, G(\qq_p))$-équivariant $\oscr^{\kappa}_{\mathcal{FL}}$ et on vérifie qu'on possède un isomorphisme canonique de faisceaux $(\mathfrak{g}, B(\qq_p))$-équivariants :
$$\oscr^{\kappa}_{\mathcal{FL}}\vert_{U_{w_0}} \simeq \oscr_{U_{w_0}}^{\kappa}\otimes_k k( \kappa).$$
\end{rem}
\subsubsection{Faisceaux inversibles $\mathfrak{g}$-équivariants au voisinage de l'infini}\label{sec-fsurcon2}
Notons $\mathbf{Coh}_{\mathfrak{g}} (\oscr_{\mathcal{FL}, \infty})$ la limite inductive des catégories  $\mathbf{Coh}_{\mathfrak{g}}(U)$ où $U$ parcourt les voisinages du point  $\infty$.    

Commençons par  une construction sur l'ouvert $B\backslash Bw_0Bw_0 = U_{1}$. 
Considérons le morphisme : 
$$ \pi :  \hat{U} \backslash \hat{B} w_0 Uw_0 \rightarrow \hat{B} \backslash \hat{B} w_0 Uw_0.$$
Pour tout $\kappa \in X^\star(T)_k$, on pose :
$$ \oscr_{U_{1}}^\kappa(V) = \{ f : \pi^{-1}(V) \rightarrow \mathbb{A}^1, f(bg) = w_0\kappa(b) f(g)~\forall b \in \hat{B}\}.$$

L'action par translation à droite de $\hat{G}$ fournit une structure $\mathfrak{g}$-équivariante  sur  $\oscr_{U_{1}}^\kappa$.

\begin{rem}   On a donc construit des faisceaux $\mathfrak{g}$-équivariants,  $\oscr_{U_1}^\kappa$ et $\oscr_{U_{w_0}}^\kappa$ pour tout $\kappa \in X^\star(T)_k$. Ces deux faisceaux se recollent en le faisceau $\oscr_{\mathcal{FL}}^\kappa$ si $\kappa \in X^\star(T^{der})\times X^\star(Z)_{k}$. 
\end{rem}

On pose $\oscr_{\infty}^{\kappa} = i_{\infty}^{-1}(\oscr_{U_1}^{\kappa}) = \oscr_{U_1,\infty}^{\kappa} $. C'est un objet de $\mathbf{Coh}_{\mathfrak{g}} (\oscr_{\mathcal{FL}, \infty})$.


\begin{lem} \begin{enumerate}
\item Le faisceau $\oscr_{\infty}^{\kappa} $ est de poids $\omega_0 \kappa$.
\item Tout objet de rang $1$ de $\mathbf{Coh}_{\mathfrak{g}} (\oscr_{\mathcal{FL}, \infty})$ est isomorphe à un unique faisceau $\oscr_{\infty}^\kappa$. 
\end{enumerate}
\end{lem}

\begin{proof} Le premier point est évident. Pour le second point, soit $\mathscr{F}$ un objet de rang $1$ de $\mathbf{Coh}_{\mathfrak{g}} (\oscr_{\mathcal{FL}, \infty})$. Le morphisme $\mathfrak{n}^0_{\infty} \rightarrow \mathrm{End}_{\oscr_{\mathcal{FL},\infty}}( \mathscr{F})$ est $\mathfrak{g}$-équivariant, donc nul.  On possède donc un morphisme 
$(\mathfrak{b}^0/\mathfrak{n}^0)_{\infty} \rightarrow \mathrm{End}_{\oscr_{\mathcal{FL},\infty}}( \mathscr{F})$ qui est $\hat{G}$-equivariant, et donc $\mathfrak{h}$ agit à travers un caractère $\chi \in X^\star(T)_k$. Quitte à tordre, on peut supposer que $\chi=0$. Ainsi, $\mathscr{F}$ est un fibré à connection de rang $1$ sur $\oscr_{\mathcal{FL},\infty}$. On vérifie sans peine que $\mathscr{F}$ possède un générateur horizontal et donc $\mathscr{F} \simeq \oscr_{\mathcal{FL}, \infty}$. 
\end{proof}

Le groupe $B(\qq_p)$ stabilise le point $\infty$ et  agit  sur l'ensemble des  voisinages $U$ de $\infty$.  Notons  $\mathbf{Coh}_{(\mathfrak{g}, B(\qq_p))}(\oscr_{\mathcal{FL}, \infty})$ la catégorie dont les objets sont des modules de types finis sur $\oscr_{\mathcal{FL}, \infty}$, munis d'une action de $\mathfrak{g}$ et d'une action de $B(\qq_p)$ compatibles au sens des points $(1)$ et $(2)$ de la section \ref{faisceauxGM}.
On définit une structure $B(\qq_p)$-équivariante sur $\oscr^\kappa_{\infty}$. Pour toute section $f$ de $\oscr^\kappa_{\infty}$ (vue comme une fonction $f(bw_0 u w_0)$ pour $b \in \hat{B}$ et $u \in U$ arbitrairement proche de $1$), on pose :
\begin{enumerate}
\item     $u'f( b w_0 u w_0) =   f(bw_0  u w_0 u')$ pour tout $u' \in U(\qq_p)$,
\item $t f(bw_0 u w_0) = f(b t^{-1} w_0  u w_0 t)$ pour tout $t \in T(\qq_p)$. 
\end{enumerate}
On vérifie immédiatement que $\oscr^\kappa_{\infty}$ est bien un objet de $\mathbf{Coh}_{(\mathfrak{g}, B(\qq_p))}(\oscr_{\mathcal{FL}, \infty})$. 

Rappelons que pour tout caractère continu $\lambda : T(\qq_p) \rightarrow k^\times$, on note $k(\lambda)$ le $k$-espace vectoriel munit de l'action de $B(\qq_p)$ par $\lambda$. %

\begin{lem}
Tout objet de rang $1$ de  $\mathbf{Coh}_{(\mathfrak{g}, B(\qq_p))} (\oscr_{\mathcal{FL}, \infty})$ est isomorphe à un faisceau 
$\oscr^\kappa_{\infty} \otimes_k k(\lambda)$ pour un unique $\kappa \in X_\star(T)_k$ et $\lambda : T(\qq_p) \rightarrow k^\times$.
\end{lem}
\begin{proof} Soit $\mathscr{F}$ un objet de rang $1$ de $\mathbf{Coh}_{(\mathfrak{g}, B(\qq_p))} (\oscr_{\mathcal{FL}, \infty})$.   Quitte à tordre, on peut supposer  que $\mathscr{F}$ est de poids nul donc isomorphe à $\oscr_{\mathcal{FL,\infty}}$ (et l'action de $\hat{G}$ correspond à la connection triviale). L'action de $b \in B(\qq_p)$ est donnée par la mutliplication par une fonction $f_b \in \oscr_{\mathcal{FL}, \infty}$ et la compatibilité entre l'action de $\hat{G}$ et de $B(\qq_p)$ entraine que $d( f_b) = b d(1) = 0$. Ainsi, l'action de $B(\qq_p)$ est bien donnée par un caractère $\lambda$ du tore $T(\qq_p)$. 
\end{proof}

\begin{rem} Soit $\kappa \in X^\star(T^{der}) \times X^\star(Z)_k \hookrightarrow X^\star(T)_k$. On possède un faisceau $(\mathfrak{g}, G(\qq_p))$-équivariant $\oscr^{\kappa}_{\mathcal{FL}}$ et on vérifie qu'on possède un isomorphisme canonique de faisceaux $(\mathfrak{g}, B(\qq_p))$-équivariants :
$$\oscr^{\kappa}_{\mathcal{FL},\infty} \simeq \oscr_{\infty}^{\kappa}\otimes_k k( w_0 \kappa).$$
\end{rem}

\section{Faisceaux automorphes}
\subsection{Catégories de faisceaux cohérents sur les courbes modulaires}
\subsubsection{Courbes modulaires perfectoides} Soit $K^p$ un sous-groupe ouvert compact de $G(\mathbb{A}_f^p)$. Soit $K_p$ un sous-groupe ouvert compact de $G(\qq_p)$. 
Soit $X_{K_pK^p}$ la courbe modulaire adique de niveau $K_pK^p$ sur $\Spa (\C_p, \ocal_{\C_p})$. Soit $X_{K^p} = \lim_{K_p} X_{K_p K^p}$ la courbe modulaire perfectoide (\cite{scholze-torsion}). Soit $\pi_{K_p} : X_{K^p} \rightarrow X_{K_pK^p}$ la projection vers la courbe modulaire de  niveau $K_pK^p$. Si $U$ est un ouvert quasi-compact de $X_{K^p}$ alors $U$ est de la forme $\pi^{-1}_{K_p} (U_{K_p})$ où $U_{K_p}$ est un ouvert quasi-compact de $X_{K_pK^p}$ pour tout  $K_p$ assez petit. 
On note $\oscr_{X_{K^p}}^{sm} = \colim_{K_p} \pi_{K_p}^{-1} \oscr_{X_{K_pK^p}}$. La complétion du faisceau $\oscr_{X_{K^p}}^{sm}$ est le  faisceau  structural $\oscr_{X_{K^p}}$ de $X_{K^p}$. 
On possède une morphisme d'espaces annelés $\pi^{sm}_{K_p} : (X_{K^p}, \oscr_{X_{K^p}}^{sm}) \rightarrow (X_{K_pK^p}, \oscr_{X_{K_pK^p}})$ (ne pas confondre avec $\pi_{K_p}$ !). On a donc un foncteur 
\begin{eqnarray*} (\pi_{K_p}^{sm})^\star :  \mathbf{Coh}(X_{K_pK^p}) &\rightarrow &\mathbf{Coh}(\oscr_{X_{K^p}}^{sm}) \\\mathscr{F} &\mapsto& \pi_{K_p}^{-1} \mathscr{F}\otimes_{\pi_{K_p}^{-1} \oscr_{X_{K_pK^p}}}\oscr_{X_{K^p}}^{sm}
\end{eqnarray*}

L'énoncé qui suit exprime entre autre que la catégorie $\mathbf{Coh}(\oscr_{X_{K^p}}^{sm})$ est la limite inductive des catégories 
$\mathbf{Coh}(X_{K_pK^p})$. 

\begin{prop}\label{prop-coherence} \begin{enumerate} 
\item Le faisceau $\oscr_{X_{K^p}}^{sm}$ est cohérent dans $\mathbf{Mod}(\oscr_{X_{K^p}}^{sm})$.
\item Soit $U = \lim_{K_p} U_{K_p}$ un ouvert  quasi-compact de $X_{K_p}$ et $\mathscr{F} \in \mathbf{Coh}( \oscr_{X_{K^p}}^{sm}\vert_U)$. Alors il existe un sous-groupe ouvert compact $K_p$ est  $\mathscr{F}_{K_p} \in \mathbf{Coh}( U_{K_p})$ tel que   $\mathscr{F} = (\pi_{K_p}^{sm})^\star \mathscr{F}_{K_p}$.

\item Soit $\phi : \mathscr{F} \rightarrow \mathscr{G}$ un morphisme dans  $\mathbf{Coh}( \oscr_{X_{K^p}}^{sm}\vert_U)$. Alors il existe un sous-groupe ouvert compact $K_p$ et un morphisme  $\phi_{K_p} : \mathscr{F}_{K_p} \rightarrow \mathscr{G}_{K_p}$ dans  $\mathbf{Coh}( U_{K_p})$ tel que   $\phi = (\pi_{K_p}^{sm})^{\star} \phi_{K_p}$.
\item Soit $\phi : \mathscr{F} \rightarrow \mathscr{G}$ un morphisme dans  $\mathbf{Coh}( \oscr_{X_{K^p}}^{sm}\vert_U)$. Alors $\phi$ est un isomorphisme si et seulement si, pour tout $\C_p$-point $x$ de $U$, $\phi_x : \mathscr{F}_x \rightarrow \mathscr{G}_x$ est un isomorphisme.   
\end{enumerate}
\end{prop}

\begin{proof} Montrons que le faisceau  $\oscr_{X_{K^p}}^{sm}$  est cohérent (voir le numéro \ref{sect-f-eq}). Le premier point de la définition est évident. Pour le second point, on peut supposer que l'ouvert $U \subseteq X_{K_p}$ est quasi-compact.   Soit $s_1, \cdots, s_i \in \oscr_{X_{K^p}}^{sm}(U)$.  Alors il existe un sous-groupe ouvert compact $K_p$ et un ouvert $U_{K_p}$ de $X_{K_pK^p}$ tel que $U = \pi_{K_p}^{-1}(U_{K_p})$ et  tel que $s_1, \cdots, s_i \in \oscr_{X_{K_pK^p}}(U_{K_p})$.  Le noyau du morphisme $\oscr_{U_{K_p}}^i \rightarrow \oscr_{U_{K_p}}$ est de type fini comme $\oscr_{U_{K_p}}$-module. Le noyau du morphisme $(\oscr^{sm}_{U})^i \rightarrow \oscr^{sm}_{U}$ est donc de type fini. Le reste  resulte  du lemme 01ZR de \cite{stacks-project} (énoncé pour des schémas). Le dernier point, se déduit du résultat analogue dans  $\mathbf{Coh}(\oscr_{X_{K_pK^p}})$.
\end{proof} 

\subsubsection{Application des périodes de Hodge-Tate}
Soit $\pi_{HT} : X_{K^p} \rightarrow \mathcal{FL}$ l'application des périodes de Hodge-Tate  (\cite{scholze-torsion}).  On note $\oscr_{K_p} =  (\pi_{HT})_\star \pi_{K_p}^{-1}( \oscr_{X_{K_p}})$, 
$\oscr^{sm} = (\pi_{HT})_\star \oscr_{X_{K^p}}^{sm} = \colim_{K_p} \oscr_{K_p}$ et $\hat{\oscr} = (\pi_{HT})_\star \oscr_{X_{K^p}}$. 

On considère le morphisme d'espaces annelés  $\pi^{sm}_{HT} : (X_{K^p}, \oscr_{X_{K^p}}^{sm}) \rightarrow (\mathcal{FL}, \oscr^{sm})$ (ne pas confondre avec $\pi_{HT}$ !). 
On possède donc un foncteur $(\pi^{sm}_{HT})_\star = (\pi_{HT})_\star : \mathbf{Coh}(\oscr_{X_{K^p}}^{sm}) \rightarrow \mathbf{Coh}(\oscr^{sm})$ et un foncteur 
\begin{eqnarray*}
  (\pi_{HT}^{sm})^\star : \mathbf{Coh}(\oscr^{sm}) &\rightarrow &\mathbf{Coh}(\oscr_{X_{K^p}}^{sm}) \\
\mathscr{F} &\mapsto& \pi_{HT}^{-1} \mathscr{F}\otimes_{\pi_{HT}^{-1} \oscr^{sm}}\oscr_{X_{K^p}}^{sm}
\end{eqnarray*}

\begin{prop}
\begin{enumerate} 
\item Le faisceau $\oscr^{sm}$ est cohérent dans $\mathbf{Mod}(\oscr^{sm})$,
\item Les foncteurs $(\pi_{HT})_\star $ et $(\pi_{HT}^{sm})^\star$ induisent des équivalences entre les catégories  $\mathbf{Coh}(\oscr_{X_{K^p}}^{sm})$ et  $\mathbf{Coh}(\oscr^{sm})$. 

\item Pour tout $\mathscr{F} \in \mathbf{Coh}(\oscr_{X_{K^p}}^{sm})$, on a $\mathrm{R}^i (\pi_{HT})_\star \mathscr{F} =0$ pour tout $i >0$.
\end{enumerate}
\end{prop}

\begin{proof} Le morphisme  $\pi_{HT}^{-1} \oscr^{sm} \rightarrow \oscr^{sm}_{X_{K^p}}$ est plat. Ainsi, le foncteur $(\pi^{sm}_{HT})^\star : \mathscr{F} \mapsto \pi_{HT}^{-1} \mathscr{F}\otimes_{\pi_{HT}^{-1} \oscr^{sm}}\oscr_{X_{K^p}}^{sm}$ est exact et le foncteur $(\pi_{HT})_\star$ est un inverse à gauche. 
Soit $\mathscr{F} \in \mathbf{Coh}(\oscr_{X_{K^p}}^{sm})$.  D'après \cite{scholze-torsion}, on possède une base d'ouverts quasi-compacts  $U$ de $\mathcal{FL}$ tel que $\pi_{HT}^{-1}(U) = \lim_{K_p} U_{K_p}$ où $U_{K_p}$ est un ouvert affinoide de $X_{K_pK^p}$ pour tout compact ouvert $K_p$ assez petit. De plus $\mathscr{F}\vert_U = \colim_{K_p} \mathscr{F}_{K_p}$ pour des  faisceaux cohérents $\mathscr{F}_{K_p}$ sur $U_{K_p}$ pour tout $K_p$ assez petit. Ainsi $\mathscr{F}_{K_p}$ est le faisceau associé à son module de sections globales  qui est de présentation  finie sur $\oscr_{X_{K_pK^p}}(U_{K_p})$ et  $\HH^i(U_{K_p}, \mathscr{F}_{K_p}) = 0$ pour tout $i >0$.
Il en résulte que $\mathscr{F}\vert_{{\pi_{HT}^{-1}}(U)}$ est le faisceau associé à son module de sections globales qui est de présentation finie sur $\oscr^{sm}(U)$ et donc $(\pi^{sm}_{HT})^\star (\pi_{HT})_{\star} \mathscr{F} = \mathscr{F}$. De plus,  $\HH^i(U, (\pi_{HT})_\star \mathscr{F}) = \colim_{K_p} \HH^i(U_{K_p}, \mathscr{F}_{K_p}) = 0$ pour tout $i >0$. 
Il en résulte que $\mathrm{R}^i(\pi_{HT})_\star\mathscr{F} =0$ pour tout $i >0$. 
Soit $s_1, \cdots, s_i \in \oscr^{sm}(U)$. Considérons le  morphisme $\phi : \oscr^{sm}\vert_U^i \rightarrow \oscr^{sm}\vert_U$.  On déduit que $(\pi^{sm}_{HT})^\star \mathrm{Ker}(\phi)$ est de type fini, et donc $\mathrm{Ker}(\phi)$ est de type fini. 
\end{proof}

\subsection{Le foncteur $VB$ sur la catégorie $\mathbf{Coh}_{\mathfrak{g}}(U)$}
Soit $U$ un ouvert de $\mathcal{FL}$. On définit un foncteur : 
\begin{eqnarray*} 
VB:   \mathbf{Coh}_{\mathfrak{g}}(U) &\rightarrow& \mathbf{Mod}(\oscr^{sm}\vert_U) \\
\mathscr{F} & \mapsto & \colim_{K_p}(\mathscr{F} {\otimes}_{\oscr_{U}} \hat{\oscr}\vert_U)^{K_p}
\end{eqnarray*}
Explicitons la définition. Pour tout ouvert  affinoide $V$  de $U$, on a $VB(\mathscr{F})(V) =  \colim_{K_p} \HH^0(K_p, \mathscr{F}(V) {\otimes}_{\oscr_{\mathcal{FL}(V)}} \hat{\oscr}(V))$. On observe que l'action de $K_p$ diagonale est bien défini pour tout $K_p$ suffisamment petit. En effet, d'après le lemme \ref{lem-extension}, l'action de $\hat{G}$ sur $\mathscr{F}(V)$ est induite par l'action d'un sous-groupe ouvert $H$ de $G$, qui  contient un sous-groupe ouvert de $G(\qq_p)$. 
\begin{thm}\label{thm}
\begin{enumerate} 
\item Pour tout $\mathscr{F} \in \mathbf{Coh}_{\mathfrak{g}}(U)$, $VB(\mathscr{F})$ est un faisceau localement libre de rang fini. 

\item  Si $U$ est quasi-compact,  et $\mathscr{F} \in \mathbf{Coh}_{\mathfrak{g}}(U)$, il existe un sous-groupe compact ouvert $K_p$ et un faisceau localement libre canonique de $(\oscr_{K_p})\vert_U$-modules, noté  $VB_{K_p}(\mathscr{F})$, et un isomorphisme  $VB(\mathscr{F}) = VB_{K_p}(\mathscr{F}) \otimes_{\oscr_{K_p}\vert_U} \oscr^{sm}\vert_U$. 

\item Pour tout $\mathscr{F} \in \mathbf{Coh}_{\mathfrak{g}}(U)^{\mathfrak{n^0}}$, on possède un isomorphisme canonique  $$VB(\mathscr{F}) {\otimes}_{\oscr^{sm}\vert_U} \hat{\oscr}\vert_U \rightarrow \mathscr{F} {\otimes}_{\oscr_{U}} \hat{\oscr}\vert_U$$   et le foncteur $VB$ restreint à la catégorie $\mathbf{Coh}_{\mathfrak{g}}(U)^{\mathfrak{n^0}}$ est exact.

\item Le foncteur $VB$ est dérivable, et pour tout  $\mathscr{F} \in \mathbf{Coh}_{\mathfrak{g}}(U)$,  on a $\mathrm{R}^iVB(\mathscr{F}) = VB(\HH^i(\mathfrak{n}^0, \mathscr{F}))$.

\end{enumerate}
\end{thm}

\begin{rem} Si $U$ est quasi-compact, et $K_p$ est suffisamment petit comme au $(2)$, l'isomorphisme du point $(3)$ précédent s'écrit  $$VB_{K_p}(\mathscr{F}) {\otimes}_{(\oscr_{K_p})\vert_U} \hat{\oscr}\vert_U \rightarrow \mathscr{F} {\otimes}_{\oscr_{U}} \hat{\oscr}\vert_U.$$
On peut de plus supposer que l'action de $\hat{G}$ sur $\mathscr{F}$ induit une action de $K_p$ (quitte à rapetisser $K_p$). On peut alors supposer que  l'isomorphisme est  $K_p$-équivariant pour l'action de $K_p$ sur  $\hat{\oscr}\vert_U$ sur le membre de gauche, et l'action de $K_p$ diagonale   sur le membre de droite.
\end{rem}

\begin{rem}\label{rem-lisse} Soit $M$ un sous-groupe localement profini qui stabilise $U$ et   $\mathscr{F} \in \mathbf{Coh}_{(\mathfrak{g}, M)}(U)$, alors $VB(\mathscr{F})$ est un faisceau $M$-équivariant. Il en résulte que les groupes de  cohomologie $\HH^i(U, \mathrm{R}^jVB(\mathscr{F}))$ sont des représentations de $M$. 
Si $\mathscr{F} \in \mathbf{Coh}_{(\mathfrak{g}, M)^f}(U)$ alors ces représentations sont lisses (cela résulte de la remarque précédente). 
\end{rem}

\subsection{Le foncteur $\tilde{VB}$}

Il est parfois plus simple de travailler sur $X_{K^p}$. On définit donc une variante du  foncteur précédent :
\begin{eqnarray*} 
\tilde{VB}:   \mathbf{Coh}_{\mathfrak{g}}(U) &\rightarrow& \mathbf{Mod}(\oscr_{X_{K^p}}^{sm}\vert_U) \\
\mathscr{F} & \mapsto & \colim_{K_p}\big( \pi_{HT}^{-1} \mathscr{F} {\otimes}_{\pi_{HT}^{-1}\oscr_{U}} \oscr_{X_{K^p}}\vert_{\pi_{HT}^{-1}(U)}\big)^{K_p}
\end{eqnarray*}

On va démontrer 

\begin{thm}\label{thm2}
\begin{enumerate} 
\item Pour tout $\mathscr{F} \in \mathbf{Coh}_{\mathfrak{g}}(U)$, $\tilde{VB}(\mathscr{F})$ est un faisceau localement libre de rang fini de $\oscr_{X_{K^p}}^{sm}\vert_U$-modules.

 \item Supposons que  $U$ est quasi-compact. Soit  $\mathscr{F} \in \mathbf{Coh}_{\mathfrak{g}}(U)$, il existe un sous-groupe compact ouvert $K_p$ et un ouvert $U_{K_p} \subseteq X_{K_pK^p}$ tel que $\pi_{K_p}^{-1}(U_{K_p}) = \pi_{HT}^{-1}(U)$, et un faisceau localement libre canonique de $\oscr_{U_{K_p}}$-modules, noté  $\tilde{VB}_{K_p}(\mathscr{F})$ muni d'un isomorphisme canonique  $$\tilde{VB}(\mathscr{F}) = \tilde{VB}_{K_p}(\mathscr{F}) \otimes_{\oscr_{U_{K_p}}} \oscr^{sm}_{X_{K^p}}\vert_{\pi_{HT}^{-1}(U)}.$$ 
\item Pour tout $\mathscr{F} \in \mathbf{Coh}_{\mathfrak{g}}(U)^{\mathfrak{n^0}}$, on a $$\tilde{VB}(\mathscr{F}) {\otimes}_{\oscr_{X_{K_p}}^{sm}\vert_U} \oscr_{X_{K^p}}\vert_U = \pi_{HT}^{-1}\mathscr{F} {\otimes}_{\pi_{HT}^{-1}\oscr_{U}} {\oscr_{X_{K^p}}}\vert_U$$ et le foncteur $\tilde{VB}$ restreint à la catégorie $\mathbf{Coh}_{\mathfrak{g}}(U)^{\mathfrak{n^0}}$ est exact.
\item Le foncteur $\tilde{VB}$ est dérivable, et pour tout  $\mathscr{F} \in \mathbf{Coh}_{\mathfrak{g}}(U)$,  on a $\mathrm{R}^i\tilde{VB}(\mathscr{F}) = \tilde{VB}(\HH^i(\mathfrak{n}^0, \mathscr{F}))$.
\end{enumerate}
\end{thm}

\begin{lem} Le théorème \ref{thm2} implique le théorème \ref{thm}.
\end{lem}

\begin{proof} 
On démontre que $VB(\mathscr{F}) = (\pi_{HT})_\star \tilde{VB}(\mathscr{F})$. En effet, on trouve que 
$$ (\pi_{HT})_\star  \colim_{K_p}\big( \pi_{HT}^{-1} \mathscr{F} {\otimes}_{\pi_{HT}^{-1}\oscr_{U}} \oscr_{X_{K^p}}\vert_{\pi_{HT}^{-1}(U)}\big)^{K_p} = $$$$  \colim_{K_p}\big((\pi_{HT})_\star \big( \pi_{HT}^{-1} \mathscr{F} {\otimes}_{\pi_{HT}^{-1}\oscr_{U}} \oscr_{X_{K^p}}\vert_{\pi_{HT}^{-1}(U)})\big)^{K_p}$$
et la formule de projection donne  un isomorphisme $$(\pi_{HT}) _\star (\pi_{HT}^{-1} \mathscr{F} {\otimes}_{\pi_{HT}^{-1}\oscr_{U}} \oscr_{X_{K^p}}\vert_{\pi_{HT}^{-1}(U)}) \rightarrow  \mathscr{F} {\otimes}_{\oscr_{U}} \hat{\oscr}\vert_U$$ qui induit un isomorphisme $(\pi_{HT})_\star \tilde{VB}(\mathscr{F}) = VB(\mathscr{F})$.
Soit $\mathscr{F} \in \mathbf{Coh}_{\mathfrak{g}}(U)^{\mathfrak{n^0}}$. Appliquons le foncteur $(\pi_{HT})_\star$ a l'identité $$\tilde{VB}(\mathscr{F}) {\otimes}_{\oscr_{X_{K_p}}^{sm}\vert_U} \oscr_{X_{K^p}}\vert_U = \pi_{HT}^{-1}\mathscr{F} {\otimes}_{\pi_{HT}^{-1}\oscr_{U}} {\oscr_{X_{K^p}}}\vert_U.$$
Comme $\tilde{VB}(\mathscr{F}) {\otimes}_{\oscr_{X_{K_p}}^{sm}\vert_U} \oscr_{X_{K^p}}\vert_U = \pi_{HT}^{-1} VB(\mathscr{F}) \otimes_{\pi_{HT}^{-1} \oscr^{sm}\vert_U} \oscr_{X_{K^p}}\vert_U$ la formule de projection implique que :
$$(\pi_{HT})_\star \tilde{VB}(\mathscr{F}) {\otimes}_{\oscr_{X_{K_p}}^{sm}\vert_U} \oscr_{X_{K^p}}\vert_U = VB(\mathscr{F}) {\otimes}_{\oscr^{sm}\vert_U} \hat{\oscr}\vert_U.$$
Il en résulte que $$VB(\mathscr{F}) {\otimes}_{\oscr^{sm}\vert_U} \hat{\oscr}\vert_U \rightarrow \mathscr{F} {\otimes}_{\oscr_{U}} \hat{\oscr}\vert_U$$ est un isomorphisme.

On voit similairement que $(\pi_{HT})_\star \tilde{VB}_{K_p}(\mathscr{F})= VB_{K_p}(\mathscr{F})$. 
Le dernier point suit de l'exactitude de $(\pi_{HT})_\star$. 
\end{proof}

\subsection{Exemples}

On va illustrer ce théorème dans quelques exemples simples. 

\subsubsection{Fibrés vectoriels automorphes}\label{subsec-ex1}

On note $\mathcal{T}_{dR,K_p} \rightarrow X_{K_pK^p}$ le $T$-torseur des trivialisations de $\omega_{E^t} \oplus Lie(E)$. 
On note $\mathcal{T}_{HT} = U\backslash G \rightarrow \mathcal{FL}$ le $T$-torseur canonique sur $\mathcal{FL}$. On possède par construction un isomorphisme $G(\qq_p)$-équivariant de $T$-torseurs sur $X_{K^p}$ :
$$ \mathcal{T}_{dR,K_p} \times_{X_{K_p K^p}} X_{K^p} = \mathcal{T}_{HT} \times_{\mathcal{FL}} X_{K^p}$$
Le $T$-torseur $\mathcal{T}_{dR,K_p}$ induit une application 
\begin{eqnarray*}
X^\star(T) & \rightarrow & \mathbf{Coh}(X_{K_pK^p})\\
\kappa & \mapsto & \omega^{\kappa}_{K_p}
\end{eqnarray*}
où $\omega^\kappa_{K_p}$ est "le faisceau des formes modulaires de poids $\kappa$" sur $X_{K_pK^p}$. 
On a donc un isomorphisme :
$$ \omega^\kappa_{K_p} \otimes_{\oscr_{X_{K_pK^p}}} \oscr_{X_{K^p}} = \oscr_{\mathcal{FL}}^\kappa \otimes_{\oscr_{\mathcal{FL}} }\oscr_{X_{K^p}}.$$
Notons $\omega^{\kappa,sm} =  \colim_{K_p} \omega_{K_p}^{\kappa}$.  On déduit:

\begin{coro} On a $\tilde{VB}(\oscr_{\mathcal{FL}}^\kappa) =  \omega^{\kappa,sm}$. 
\end{coro}

On note  aussi $VB(\oscr_{\mathcal{FL}}^\kappa) = \omega_{\mathcal{FL}}^{\kappa,sm} = (\pi_{HT})_\star \omega^{\kappa,sm}$.  On possède alors un isomorphisme $G(\qq_p)$-équivariant :  $$\HH^i(\mathcal{FL}, \omega_{\mathcal{FL}}^{\kappa,sm}) = \colim_{K_p} \HH^i(X_{K_pK^p}, \omega_{K_p}^\kappa) = \HH^i(X_{K^p}, \omega^{\kappa,sm}).$$ 
\subsubsection{Systèmes locaux et décomposition de Hodge-Tate}\label{subsec-ex2}

Soit $\kappa \in X^\star(T)^+$ un poids $B$-dominant. Soit $V_\kappa$ la représentation  de $G$ de plus haut poids $\kappa$. Notons $\mathcal{V}_{\kappa, K_p}$ le système local dans la topologie pro-Kummer étale induit sur $X_{K_pK^p}$. 
Notons $\mathcal{V}_\kappa$ le système local sur $X_{K^p}$. 
On a $$\colim_{K_p} \mathrm{R}\Gamma(K_p, \mathrm{R}\Gamma_{proket}( X_{K^p}, \mathcal{V}_{\kappa})) = \colim_{K_p} \mathrm{R}\Gamma_{proket}(X_{K_pK^p}, \mathcal{V}_{K_p,\kappa}).$$
On a
\begin{eqnarray*}
 (V_\kappa \otimes_{\qq_p} \oscr_{\mathcal{FL}})^{\mathfrak{n}^0}& =& \oscr^{\omega_0 \kappa} \\
 (V_\kappa \otimes_{\qq_p} \oscr_{\mathcal{FL}})_{\mathfrak{n}^0}& = &\oscr^\kappa \otimes (\mathfrak{n}^0)^\vee
 \end{eqnarray*}
 On rappelle l'isomorphisme de Kodaira-Spencer : 
 $$KS : \tilde{VB}((\mathfrak{n}^0)^\vee) = \colim_{K_p} \Omega^1_{X_{K_pK^p}}(\mathrm{log}(cusp)).$$ 
Donc $\tilde{VB}(\oscr^{\omega_0 \kappa}) = \omega^{\omega_0\kappa,sm}$, $\tilde{VB}(\oscr^{ \kappa}\otimes (\mathfrak{n}^0)^\vee) = \omega^{\kappa,sm} \otimes\tilde{VB}((\mathfrak{n}^0)^\vee)$.

 On observe que 
\begin{eqnarray*}
\colim_{K_p} (\mathrm{R}\Gamma_{proket}(X_{K_pK^p}, V_{\kappa,et})\otimes_{\qq_p} \C_p) &=& \colim_{K_p} \mathrm{R}\Gamma\big(K_p, \mathrm{R}\Gamma_{proket}( X_{K^p}, \mathcal{V}_{\kappa})\otimes_{\qq_p} \C_p\big) \\
&=& \colim_{K_p} \mathrm{R}\Gamma(K_p, \mathrm{R}\Gamma_{proket}( X_{K^p}, {V}_{\kappa} \otimes_{\qq_p} \oscr_{X_{K^p}})) \\
&=& \colim_{K_p} \mathrm{R}\Gamma(K_p, \mathrm{R}\Gamma_{an}( X_{K^p}, {V}_{\kappa} \otimes_{\qq_p} \oscr_{X_{K^p}})) 
\end{eqnarray*}
d'après le théorème de comparaison primitif de \cite{MR3090230}.  On a  donc:
$$\colim_{K_p} \mathrm{R}\Gamma(K_p, \mathrm{R}\Gamma_{an}( X_{K^p}, {V}_{\kappa} \otimes_{\qq_p} \oscr_{X_{K^p}})) = \mathrm{R}\Gamma(\mathcal{FL}, \mathrm{R}VB(V_\kappa \otimes_{\qq_p} \oscr_{\mathcal{FL}}))$$
et il résulte du théorème 4.4.1 de \cite{MR102537} qu'on possède une suite exacte (de Hodge-Tate):

$$0 \rightarrow  \HH^1(X_{K^p}, \omega^{w_0\kappa,sm})  \rightarrow \colim_{K_p} (\HH^1_{proket}(X_{K_pK^p}, \mathcal{V}_{\kappa,K_p})\otimes_{\qq_p} \C_p) \rightarrow \HH^0(X_{K^p},  \omega^{\kappa,sm}\otimes \tilde{VB}((\mathfrak{n}^0)^\vee))\rightarrow 0$$

\subsubsection{Formes modulaires surconvergentes}\label{subsec-ex3}  D'après la section \ref{sec-fsurcon}, pour tout caractère  $\kappa \in X^\star(T)_{\C_p}$, on possède un faisceau $\oscr^{\kappa}_{U_{w_0}} \in \mathbf{Coh}_{(\mathfrak{g}, B(\qq_p))}( U_{w_0})$ et un faisceau $\oscr^{\kappa}_{\infty} \in \mathbf{Coh}_{(\mathfrak{g}, B(\qq_p))}( \oscr_{\mathcal{FL},\infty})$. On note: 
\begin{eqnarray*}
\omega^{\kappa,sm}_{U_{w_0}} &=& VB(\oscr^{\kappa}_{U_{w_0}})\\
\omega^{\kappa,sm}_{\infty} &=& VB(\oscr^{\kappa}_{\infty})
\end{eqnarray*}

\begin{lem} Les groupes de cohomologie $\HH^1_{c}(U_{w_0}, \omega_{U_{w_0}}^{\kappa,sm})$ et $\HH^0( \{\infty\}, \omega^{\kappa,sm}_{\infty})$  sont des  représentations localement analytiques  de $B(\qq_p)$ de poids respectifs $-\kappa$ et $-w_0 \kappa$ (c'est à dire que l'action de $\mathfrak{b}$ se factorise à travers les caractères $-\kappa$ et $-w_0 \kappa$) . 
\end{lem}
\begin{proof}  Fixons un caratère $\chi : T(\qq_p) \rightarrow \C_p^\times$ de differentielle $\mathrm{d}\chi = \kappa \in X^\star(T)$. On voit d'après la remarque \ref{rem-fortement}  que $\oscr^{\kappa}_{U_{w_0}} \otimes_{\C_p} \C_p(\chi)$ est un objet fortement $(\mathfrak{g}, B(\qq_p))$-équivariant et donc sa cohomologie est une représentation lisse d'après la remarque \ref{rem-lisse}. Il en va de même pour $\oscr^{\kappa}_{\infty} \otimes_{\C_p} \C_p(w_0\chi)$
\end{proof}

Prenons $K_p$ le sous-groupe d'Iwahori. Soit $X_{K_pK^p}^{ord,m}$ le lieu ordinaire-multiplicatif de $X_{K_pK^p}$. Soit $\chi$ un caractère $T(\Z_p) \rightarrow \C_p^\times$. 
Sur le dagger space $X_{K_pK^p}^{ord,m, \dag}$, on possède un faisceau des formes surconvergentes de poids $\chi$ (\cite{MR3097946}, \cite{MR3265287}) : $\omega^{\chi, \dag}_{K_p}$.  On note $M^\dag_{\chi} = \ \HH^0( X_{K_pK^p}^{ord,m, \dag}, \omega^{\chi, \dag}_{K_p})$, l'espace des formes surconvergentes de poids $\kappa$. Cet espace est muni d'un Frobenius $\phi$ et d'un opérateur $U_p$. 

On peut étendre $\chi$ en un caractère de $T(\qq_p)$ en posant $\chi( \mathrm{diag}(p,1)) = \chi(\mathrm{diag}(1,p)) = 1$. Posons $\kappa = \mathrm{d}\chi \in X^\star(T)_{\C_p}$.

\begin{prop}   
On a 
$$\big(\HH^0( \{\infty\}, \omega^{\kappa,sm}_{\infty}) \otimes_{\C_p} \C_p(\omega_0 \chi)\big)^{B(\ZZ_p)} = M^\dag_{\kappa}$$ et $U_p = B(\ZZ_p) \mathrm{diag}(1,p^{-1}) B(\ZZ_p)$, $\phi = B(\ZZ_p) \mathrm{diag}(p^{-1},1) B(\ZZ_p)$.
\end{prop}

\begin{proof} Le faisceau $\omega^{\chi, \dag}_{K_p}$ provient d'un faisceau inversible $\omega^\chi_{K_p}$ sur un voisinage strict $U_{K_p}$ de  $X_{K_pK^p}^{ord,m}$. On rappelle  que $\pi_{K_p}(\pi_{HT}^{-1}(\{\infty\})) = \overline{X_{K_pK^p}^{ord,m}}  \subseteq U_{K_p}$ et que $M^\dag_{\chi} = \HH^0(\overline{X_{K_pK^p}^{ord,m}}, \omega^\chi_{K_p})$. On vérifie alors qu'on possède un isomorphisme $B(\qq_p)$-équivariant:
$$ \omega^{\chi}_{K_p} \otimes_{\oscr_{U_{K_p}}} \oscr^{sm}_{X_{K^p}} \vert_{\pi_{HT}^{-1}({\infty})} = \omega^{\kappa,sm}_{\infty} \otimes_{\C_p} \C_p(w_0\chi) \otimes_{\oscr^{sm}} \oscr^{sm}_{X_{K^p}}$$

\end{proof}

Sur le dagger space $X_{K_pK^p}^{ord,et, \dag}$, on possède un faisceau des formes surconvergentes de poids $\chi$ : $\omega^{\chi, \dag}_{K_p}$ (voir \cite{Boxer-Pilloni}).  On note $N^\dag_{\chi} =  \HH^1_c( X_{K_pK^p}^{ord,et, \dag}, \omega^{\chi,\dag}_{{K_p}})$. Cet espace est muni d'un Frobenius $\phi$ et d'un opérateur $U_p$ qui est compact. Utilisons l'exposant $fs$ pour désigner la partie de pente finie pour l'opérateur $U_p$.

\begin{prop}    
On possède un isomorphisme  : 
$$(\HH^1_c( U_{w_0}, \omega^{\kappa,sm}_{U_{w_0}}) \otimes_{\C_p} \C_p(\chi))^{B(\ZZ_p),fs}  \rightarrow  N^{\dag,fs}_{\chi}$$ 
compatible pour les opérateurs  $U_p = B(\ZZ_p) \mathrm{diag}(1,p^{-1}) B(\ZZ_p)$, $\phi = B(\ZZ_p) \mathrm{diag}(p^{-1},1) B(\ZZ_p)$.

\end{prop}

\begin{proof} On a $\pi_{K_p}(U_{w_0}) = X_{K_pK^p} \setminus \overline{X_{K_pK^p}^{ord,m}}$. On possède un voisinage strict $V_{K_p}$de $X^{ord,et}_{K_pK^p}$ et un faisceau $\omega^{\chi}_{K_p}$ sur $V_{K_p}$ tel que $$ \omega^{\chi}_{K_p} \otimes_{\oscr_{V_{K_p}}} \oscr^{sm}_{X_{K^p}} \vert_{\pi_{K_p}^{-1}(V_{K_p})} = \omega^{\kappa,sm}_{U_{w_0}} \otimes_{\C_p} \C_p(\chi) \otimes_{\oscr^{sm}} \oscr^{sm}_{X_{K^p}}\vert_{\pi_{K_p}^{-1}(V_{K_p})}.$$
On possède donc un morphisme $N^\dag_{\chi} \rightarrow \HH^1_c(V_{K_p}, \omega^{\chi}_{K_p}) \rightarrow \HH^1_c( U_{w_0}, \omega^{\kappa,sm}_{U_{w_0}})$ et il est clair que ce morphisme se factorise à travers $(\HH^1_c( U_{w_0}, \omega^{\kappa,sm}_{U_{w_0}}) \otimes_{\C_p} \C_p(\chi))^{B(\ZZ_p)}$.   Notons $K_{p,n}$ le sous-groupe de $K_p$ des éléments qui sont dans le Borel modulo $p^n$. Notons $V_{K_{p,n}}$ l'image inverse de $V_{K_p}$ dans $X_{K_{p,n}K^p}$.  On voit que le morphisme :  $\colim_n \HH^1_c( V_{K_p,n}, \omega^{\chi}_{K_{p,n}}) \rightarrow  (\HH^1_c( U_{w_0}, \omega^{\kappa,sm}_{U_{w_0}}) \otimes_{\C_p} \C_p(\chi))^{B(\ZZ_p)}$ est un isomorphisme sur la partie de pente finie en utilisant un argument de prolongement analytique. 
On voit que le morphisme $\HH^1_c(V_{K_p}, \omega^{\chi}_{K_p}) \rightarrow \colim_n \HH^1_c( V_{K_p,n}, \omega^{\chi}_{K_{p,n}})$ est un isomorphisme sur la partie de pente finie (en utilisant que $U_p$ décroit le niveau en $p$). On voit finalement que $N^\dag_{\chi} \rightarrow \HH^1_c(V_{K_p}, \omega^{\chi}_{K_p})$ induit un isomorphisme sur la partie de pente finie en utilisant un argument de prolongement analytique. 
\end{proof}

\subsection{Preuve du théorème}  Cette section est dédiée à la preuve du théorème \ref{thm2}. On fixe donc $\mathscr{F}$, un objet de $ \mathbf{Coh}_{\mathfrak{g}}(U)$. 

\subsubsection{Un lemme péliminaire} 

\begin{lem}\label{lem-local-triviality} Pour tout ouvert affinoide $V= \Spa(C,C^+)$ de $U$ tel que $\mathscr{F}\vert_V$ soit libre, il existe un sous $C^+$-module libre et ouvert  $\mathscr{F}(V)^+\subseteq \mathscr{F}(V)$, et un sous-groupe ouvert 
$H_V$ de $G$ tel que l'action de $\hat{G}$ s'étende en une action  $ \rho : \mathscr{F}(V) \rightarrow \mathscr{F}(V) {\otimes}_{\C_p} \oscr_{H_V}$, et  telle que $ \rho\vert_{\mathscr{F}(V)^+} : \mathscr{F}(V)^+ \rightarrow \mathscr{F}(V)^+ {\otimes}_{\C_p} \oscr^+_{H_V}$. De plus, il existe un sous-groupe compact ouvert $K_p \subseteq H_V(\qq_p)$ dépendant uniquement de $H_V$ tel que l'action de $K_p$ sur $\mathscr{F}(V)^+/p\mathscr{F}(V)^+$ soit triviale.

 \end{lem}
\begin{proof}  Soit $V= \Spa(C,C^+)$ un ouvert affinoide de $U$. 
Posons  $F = \mathscr{F}(V)$, c'est un $C$-module libre.   D'après la proposition \ref{prop-integration-action},  il existe un sous-groupe ouvert  affinoide $H$ de $G$ tel que $F$ est un $H$-comodule.  

Pour fixer les idées, on va supposer que $H = G_m$ pour $m \in \ZZ_{\geq 1}$ et donc $\oscr_{G_m} = \C_p \langle p^{-m}X_1,p^{-m}X_2,p^{-m}X_3, p^{-m}X_4 \rangle$ (où l'élément universel de $G_m$ est $1 +    \begin{pmatrix} 
      X_1 & X_2 \\
      X_3 & X_4 \\
   \end{pmatrix}$). 
   On possède donc pour tout $r \geq m$ des morphisms  $\rho_{F,r} : F \rightarrow F {\otimes}_{\C_p} \oscr_{G_r}$ qui sont compatibles avec le morphisme $\rho_{C,r}  : C \rightarrow C {\otimes}_{\C_p} \oscr_{G_r}$ (et tous induits par le morphisme $\rho_{F,m}$).   On observe que $\rho_{C,m}(C^+) \subseteq C^+ {\otimes} \oscr_{G_m}^+$
 Considérons  une base $v_1, \cdots, v_t$ du $C$-module $F$ et notons $F^+ = \oplus C^+ v_i$.   Pour tout $v \in F$, on a $\rho_{F,m}(v) =  \sum_{(i,j,k,l) \in \Z^4_{\geq 0}} v_{i,j,k,l} p^{-m} X_1^i X_2^j X_3^kX_4^l$ où $v_{0,0,0,0} = v$ et il existe $s_v \geq 0$ tel que $v_{i,,j,k,l} \in p^{-s_v} F^+$. 
On déduit donc  qu'il existe $r \geq m$ tel que $\rho_{F,r}(F^+) \subseteq  F^+ {\otimes}_{\ocal_{\C_p}} \oscr^+_{G_r}$ (par exemple $r= m+ \max\{s_{v_i}, i=1, \cdots, t\}$).
De plus, on voit que si on prend $K_p = G_{r+1} (\qq_p) $ alors l'action de $K_p$ sur $\mathscr{F}(V)^+/p\mathscr{F}(V)^+$ est triviale. 
\end{proof}

\subsubsection{Le méthode de Sen}\label{sect-TSEN}

Soit $K_p$ un sous-groupe ouvert compact. Soit $W_{K_p} = \Spa (B_{K_p}, B_{K_p}^+) \hookrightarrow X_{K_pK^p}$ un ouvert. Pour tout $K'_p \subseteq K_p$ on pose $W_{K'_p} = \Spa(B_{K'_p}, B_{K'_p}^+) = W_{K_p} \times_{X_{K_pK^p}} X_{K'_pK^p}$. 

On suppose  que  $B_{K_p}$ est \emph{petit}. Cela signifie:

\begin{enumerate}

\item Dans le cas où 
$W_{K_p}$ ne rencontre pas de pointes,  on possède un morphisme $ \C_p \langle T^{±1} \rangle \rightarrow B_{K_p}$ qui est un composé de morphismes finis \'etales et d'inclusion d'ouverts rationnels.
\item Dans le cas où $W_{K_p}$ rencontre les pointes, on possède   un morphisme  $\C_p \langle T \rangle \rightarrow B_{K_p}$ qui est un composé de morphismes finis \'etales et d'inclusion d'ouverts rationnels  tel que $T=0$ soit le diviseur des pointes dans $W_{K_p}$.
\end{enumerate}

On vérifie sans peine  que $X_{K_pK^p}$ possède une base d'ouverts affinoides  petits.  On rappelle maintenant la construction de la double tour $\{B_{K'_p,n}\}$ de \cite{pan2021locally}, sect. 3.2.

On note $R_n = \C_p \langle T^{±p^{-n}} \rangle$ dans le cas $1)$ et $R_n =  \C_p \langle T^{p^{-n}} \rangle$ dans le cas $2)$  et $R_\infty = \C_p \langle T^{±p^{-\infty}} \rangle$ dans le cas $1)$ et  $R_\infty = \C_p \langle T^{p^{-\infty}} \rangle$  dans le cas $2)$. On note $\Gamma \simeq \ZZ_p$. Le groupe $\Gamma$ agit naturellement sur $R_{\infty}$.

Pour tout $K'_p \subseteq K_p$, on construit   $B_{K'_p, n}$, le normalisé de  $B_{K'_p} \hat{\otimes}_R R_n$ (dans le cas $1)$ le produit tensoriel  est déjà normal) et $B_{K'_p, \infty}$ est la complétion de $\colim_n B_{K'_p,n}$ pour la norme spectrale. Cela signifie que $B_{K'_p, \infty} = \widehat{ \colim_n B_{K'_p,n}^+}^p[1/p]$. On pose $B_n$ la completion de $\colim_{K'_p} B_{K'_p, n}$ et $B_\infty$ la completion de $\colim_{n, K'_p} B_{K'_p,n}$ (pour la norme spectrale). D'après le lemme d'Abhyankar, $B_{K_p, \infty} \rightarrow B_{K'_p, \infty}$ est finie étale, et $B_n \rightarrow B_m$ (pour $m \geq n$) est aussi finie étale. 

On note  $B_\infty^+$  le sous-anneau de $B_\infty$ des éléments à puissance bornée. Les anneaux $B_\infty^+$ et $B_\infty$ sont munis d'une action de $K_p\times \Gamma$. 
  


\begin{prop}\label{prop-sen-meth} Soit $0 < c<1$.   Soit $M$ un $B^+_\infty$ module libre de rang fini, muni d'une action semi-linéaire de $K_p\times \Gamma$. On suppose qu'il existe un base $(f_1, \cdots, f_n)$ de $M$ telle que $\sigma f_i - f_i \in  p M$ pour tout $\sigma \in K_p \times \Gamma$.  Alors il existe une constante $n(K_p,c)$ (qui ne dépend pas de $M$ mais qui dépend de la double tour $\{ B_{K'_p,n}\}_{K'_p,n}$) tels que  pour tout $n \geq n(K_p,c)$, on possède un unique $B^+_{K_p,n}$-module libre $D^+_{K_p,n}(M)$, qui est un sous-module de $M$ et tel que 

\begin{enumerate}
\item $D^+_{K_p,n}(M)$ est invariant sous $K_p$, et stable sous l'action de $\Gamma$,
\item Le morphisme $B^+_\infty \otimes_{B^+_{K_p,n}} D^+_{K_p,n}(M) \rightarrow M$ est un isomorphisme $K_p \times \Gamma$-équivariant,
\item il existe une base $\mathcal{B}$ de $D^+_{K_p,n}(M)$ tel que $(\gamma-1)(\mathcal{B}) \subseteq p^c D^+_{K_p,n}(M)$ pour tout $\gamma \in \Gamma$.
\end{enumerate}
\end{prop}
\begin{proof} La tour $\{ B_{K'_p,n}\}_{K'_p,n}$ vérifie les conditions de Tate-Sen (\cite{MR2493221}, def. 3.1.3) d'après \cite{pan2021locally}, sect. 3.2. On peut donc appliquer \cite{MR2493221}, prop. 3.2.6.
\end{proof} 

Pour $K'_p \subseteq K_p$ et $m \geq n$, on a $D^+_{K'_p,m}(M) = D^+_{K_p,n}(M) \otimes_{B_{K_p,n}^+} B_{K'_p,m}^+$. On pose $D_{K'_p,m}(M) = D^+_{K'_p,m}(M)[1/p]$. 

L'action de $\Gamma$ est localement analytique sur $D_{K_p,n}(M)$. On fixe un isomorphisme $\Gamma \simeq \ZZ_p$ et on note $\Theta_M \in \mathrm{End}_{B_{K_p,n}}(D_{K_p,n}(M))$ l'opérateur de Sen donné par la formule $\Theta_M (f) = \lim_{\gamma \rightarrow 0} \frac{\gamma f- f}{\gamma}$. C'est un endomorphisme qui commute à l'action de $\Gamma$.

\subsubsection{Application de la théorie de Sen à $\mathscr{F}$}\label{section-construct-Sen}


Soit $V = \Spa(C, C^+) \subseteq U \subseteq \mathcal{FL}$ un ouvert quasi-compact qui est stable sous l'action d'un sous-groupe ouvert   $H$ de $G$. 

Posons  $F = \mathscr{F}(V)$, c'est un $C$-modules. Quitte à rapetisser $V$, on peut supposer $F$ libre. 

Il résulte du lemme \ref{lem-local-triviality} qu'on possède un  sous $C^+$-module libre $F^+$ de $F$ qui engendre $F$ et un  sous-groupe compact ouvert $K_p$ assez petit  qui preserve $F^+$ et qui agit trivialement sur $F^+/pF^+$. 

Considérons un  ouvert  quasi-compact $W = \Spa(B,B^+) \subseteq  \pi_{HT}^{-1}(V)$ qui est affinoide perfectoide. 

Quitte à rapetisser $K_p$, on peut supposer que $W = \pi_{K_p}^{-1}(W_{K_p})$ avec  $W_{K_p} = \Spa(B_{K_p}, B_{K_p}^+)$.  Pour tout $K'_p \subseteq K_p$ on pose $W_{K'_p} = \Spa(B_{K'_p}, B_{K'_p}^+) = \pi_{K'_p}(W)$. 
On peut supposer (en remplaçant $W_{K_p}$ par un ouvert rationnel) que  $B_{K_p}$ est \emph{petit}.  \bigskip
   
On peut appliquer la théorie de Sen  rappelée au numéro \ref{sect-TSEN} à la représentation semi-linéaire de $K_p \times \Gamma$:  $B^+_\infty {\otimes}_{C^+} F^+$. 

\begin{lem} Fixons une constante $0< c<1$. Pour tout $n \geq n(K_p,c)$, on possède un unique $B^+_{K_p,n}$-module libre $D^+_{K_p,n}(F)$, qui est un sous-module de $B_\infty \otimes_C F$ et tel que 

\begin{enumerate}
\item $D^+_{K_p,n}(F)$ est fixé par $K_p$ et stable sous $\Gamma$,
\item Le morphisme $B^+_\infty \otimes_{B^+_{K_p,n}} D^+_{K_p,n}(F) \rightarrow B^+_\infty \otimes_{C^+} F^+$ est un isomorphisme $K_p\times \Gamma$-équivariant,
\item il existe une base $\mathcal{B}$ de $D^+_{K_p,n}(F)$ tel que $(\gamma-1)(\mathcal{B}) \subseteq p^c D^+_{K_p,n}(F)$ pour tout $\gamma \in \Gamma$,
\end{enumerate}
\end{lem}
\begin{proof} On applique la proposition \ref{prop-sen-meth} à $M = B^+_\infty {\otimes}_{C^+} F^+$ et on pose $ D^+_{K_p,n}(F) :=  D^+_{K_p,n}(M)$ dans les notations de la proposition.
\end{proof}




\begin{lem}\label{lem-calculcoho} On a $\HH^i(K_p, B \otimes_C F) = \HH^i(\Gamma \times K_p, B_\infty \otimes_C F) = \HH^i(\Gamma, (B_\infty \otimes_C F)^{K_p}) = \HH^i(\Gamma, D_{K_p,n'}(F))$ pour tout $n'$ assez grand. 
\end{lem}

\begin{proof} L'isomorphisme $\HH^i(K_p, B \otimes_C F) = \HH^i(\Gamma \times K_p, B_\infty \otimes_C F)$ se déduite de $\HH^i(\Gamma, B_\infty) = 0$ si $i>0$ et $\HH^0(\Gamma, B_\infty) = B$ qui résulte de la presque pureté (\cite{MR3090258}) appliquée à $B \rightarrow B_\infty$. 
L'isomorphisme  $\HH^i(\Gamma \times K_p, B_\infty \otimes_C F) = \HH^i(\Gamma, (B_\infty \otimes_C F)^{K_p})$ résulte de $\HH^i(K_p, B_\infty \otimes_C F) = 0$ si $i>0$  qui résulte de la presque pureté pour l'extension $B_{K_p, \infty} \rightarrow B_\infty$. On montre l'isomorphisme $\HH^i(\Gamma, (B_\infty \otimes_C F)^{K_p}) = \HH^i(\Gamma, D_{K_p,n}(F))$ en suivant la méthode  du  lemme 3.6.6 de \cite{pan2021locally}. On a $(B_\infty \otimes_C F)^{K_p} = B_{K_p, \infty} \otimes_{B_{K_p,n}} D_{K_p,n}(F)$. Il existe $m \geq 1$ tel que $(\gamma-1)^m D^+_{K_p,n}(F) \subseteq p D^+_{K_p,n}(F)$. 
On possède des traces de Tate : $\mathrm{Tr}_{n'} : B_{K_p, \infty} \rightarrow B_{K_p,n'}$ pour tout $n'$ assez grand qui ont la propriété que $\gamma-1$ est inversible sur $\mathrm{Ker}(\mathrm{Tr}_{n'})$ et $\vert (\gamma-1)^{-1}\vert \leq p^{\frac{1}{2m}}$.
Il suffit de démontrer que $\HH^1(\Gamma, \mathrm{Ker}(\mathrm{Tr}_{n'}) \otimes_{B_{K_p,n}} D_{K_p,n}(F)) = 0$. Soit $a \in  \mathrm{Ker}(\mathrm{Tr}_{n'}) \cap B_{K_p, \infty}^+$. Soit $b \in D^+_{K_p,n}(F)$. Soit $c = (\gamma-1)^{-1} pa$. 
On vérifie que la série $\sum_{\ell \geq 0} (\gamma^{-1}-1)^{-\ell}(c) \otimes (\gamma-1)^\ell(b)$ converge vers un élément $x$ dans $\mathrm{Ker}(\mathrm{Tr}_{n'}) \otimes_{B_{K_p,n}} D_{K_p,n}(F)$ et que $(\gamma-1)x = pa \otimes b$. 
 
\end{proof}

Il reste à calculer $\HH^i(\Gamma, D_{K_p,n'}(F))$. L'action de $\Gamma$ est  localement analytique, et donc on possède un opérateur de Sen $\Theta_F \in \mathrm{End}_{B_{K_p,n}}(D_{K_p,n}(F))$ (et par extension dans $\mathrm{End}_{B_{K_p,n'}}(D_{K_p,n'}(F))$ pour tout $n' \geq n$) qui commute avec l'action de $ \Gamma$. De plus on a  $\HH^i(\Gamma, D_{K_p,n'}(F)) = \HH^0(\Gamma, \HH^i(\Theta_F, D_{K_p,n'}(F))$.

Etudions d'avantage l'opérateur $\Theta_F$. On a  $\Theta_F \in \mathrm{End}_{B_{K_p,n}}(D_{K_p,n}(M)) \rightarrow \mathrm{End}_{B_\infty}(B_{\infty} \otimes_C F)$. On peut ainsi écrire également :
$$\HH^i(K_p,  B \otimes_C F) = \HH^0(K_p \times \Gamma, \HH^i(\Theta_F, B_\infty \otimes_C F))$$

Comme $\Theta_F$ commute avec l'action de $\Gamma \times K_p$, on peut aussi voir $$\Theta_F \in  (B \otimes_C F \otimes_C F^\vee)^{K_p}.$$ 

\begin{coro}\label{coro-op-Sen} La méthode de Sen nous permet de construire un recouvrement ouvert $\pi_{HT}^{-1}(U) = \cup W_i$ et pour chaque $W_i$, un opérateur de Sen $\Theta_i \in \tilde{VB}( \mathscr{F} \otimes \mathscr{F}^\vee)(W_i)$.
\end{coro}

La fin de la preuve est  dédiée à la  détermination des opérateurs de Sen. Voir la proposition \ref{prop-calcul-Senopfinal}


\subsubsection{Fonctorialité du module de Sen}

Etudions la fonctorialité en l'ouvert $W$. Soit $W' = \Spa (B',(B')^+) \subseteq W$ un ouvert rationnel. 

 On a $W' = (\pi_{K_p'})^{-1}(W'_{K'_p})$ pour tout compact ouvert $K'_p \subseteq K_p$ assez petit.  On a  $W'_{K'_p} = \Spa (B'_{K'_p}, (B'_{K'_p})^+)$. On définit $B'_{K'_p,n}$, $B'_{K'_p, \infty}$, $B'_\infty$ comme précédemment. 
 On possède une  constante $n'(K'_p,c)$ et pour tout $n \geq n'(K'_p,c)$ un module $D'_{K'_p,n}(F) \subseteq B_{\infty}' \otimes_C F$.

\begin{lem} Pour  $n \geq \max\{n(K_p,c), n'(K'_p,c)\}$,  le morphisme naturel 
$$ B'_{K'_p,n} \otimes_{B_{K_p,n}}  D_{K_p,n}(F)\rightarrow D'_{K'_p,n}(F)$$ est un isomorphisme.
\end{lem}

\begin{proof} C'est clair.  

\end{proof}

\begin{coro} Les faisceaux $\tilde{VB}(\mathscr{F})$ et $\mathrm{R}^1\tilde{VB}(\mathscr{F})$ sont cohérents. 
\end{coro}
\begin{proof} Avec les notations précédentes, on a $\tilde{VB}(\mathscr{F})(V) = \colim_{K_p} \HH^0(\Gamma, D_{K_p,n}(F)^{\Theta_F})$ et $\mathrm{R}^1\tilde{VB}(\mathscr{F})(V) = \colim_{K_p} \HH^0(\Gamma, D_{K_p,n}(F)_{\Theta_F})$.
Le morphisme $B_{K_p,n} \rightarrow B'_{K'_p,n}$ est plat, donc l'isomorphisme $$ B'_{K'_p,n} \otimes_{B_{K_p,n}}  D_{K_p,n}(F)\rightarrow D_{K'_p,n}(F)$$ induit des isomorphismes
$$ B'_{K'_p,n} \otimes_{B_{K_p,n}}  \HH^i(\Theta_F, D_{K_p,n}(F)) \rightarrow  \HH^i(\Theta_{F}, D'_{K'_p,n}(F)).$$
Le corollaire en résulte facilement. 
\end{proof}

Etudions la fonctorialité en le faisceau $\mathscr{F}$.  On travaille toujours sur un ouvert $W$ de $X_{K_pK^p}$ comme dans la section \ref{section-construct-Sen}. On se donne deux faisceaux $\mathscr{F}$ et $\mathscr{G}$ dans $\mathbf{Coh}_{\mathfrak{g}}(U)$ et pour $K_p$ assez petit, et $n \geq n(K_p,c)$ on possède leurs modules de Sen respectifs (sur l'ouvert $W$), $D_{K_p,n}(F)$ et   $D_{K_p,n}(G)$.

\begin{lem}\label{lem-functor} \begin{enumerate} 
\item Supposons qu'on a un morphisme  $\mathscr{F} \rightarrow \mathscr{G}$ dans $\mathbf{Coh}_{\mathfrak{g}}(U)$. Alors, pour tout $K_p$ assez petit et $n \geq n(K_p,c)$, on possède un morphisme naturel
$D_{K_p,n}(F) \rightarrow D_{K_p,n}(G)$ qui est $K_p \times \Gamma$-équivariant.
\item  Supposons qu'on a deux objets   $\mathscr{F}$ et  $\mathscr{G}$ dans $\mathbf{Coh}_{\mathfrak{g}}(U)$. Alors, pour tout $K_p$ assez petit et $n \geq n(K_p,c)$, on possède un isomorphisme  naturel
$D_{K_p,n}(F) \otimes_{B_{K_p,n}} D_{K_p,n}(G) \rightarrow D_{K_p,n}(F \otimes G)$ qui est $K_p \times \Gamma$-équivariant. Si $\Theta_F$ et $\Theta_G$ sont les opérateurs de Sen sur $D_{K_p,n}(F) $ et $ D_{K_p,n}(G)$, alors $\Theta_F \otimes 1 + 1 \otimes \Theta_G$ est l'opérateur de Sen sur $D_{K_p,n}(F \otimes F')$.
\end{enumerate}
\end{lem}
\begin{proof} C'est clair.
\end{proof}
\subsubsection{Objets de rang $1$}  On étudie maintenant le foncteur $\tilde{VB}$ pour des objets de rang $1$. 
On note $\mathcal{T}_{dR,K_p} \rightarrow X_{K_pK^p}$ le $T$-torseur des trivialisations de $\omega_{E^t} \oplus Lie(E)$. 
On note $\mathcal{T}_{HT} = U\backslash G \rightarrow \mathcal{FL}$ le $T$-torseur tautologique. On possède un isomorphisme $G(\qq_p)$-équivariant:
$$ \mathcal{T}_{dR,K_p} \times_{X_{K_p K^p}} X_{K^p} = \mathcal{T}_{HT} \times_{\mathcal{FL}} X_{K^p}$$

Soit $x$ un point classique de $\mathcal{FL}$ (c'est à dire un point de corps résiduel $\C_p$).  Soit $n \geq 0$. On rappelle qu'on possède un sous-groupe $G_n$ de $G$ des élements qui se réduisent sur $1$ modulo $p^n$. On note $B_{x,n} = B_x \cap G_n$, $U_{x,n} = U_x \cap G_n$, $T_{x,n} = T_x \cap G_n$.  Soit $x B_{x,n} \backslash G_n = V_x$ un voisinage du point $x$ (c'est la $G_n$-orbite de $x$). On possède un $T_{x,n}$-torseur qui est $G_n$-équivariant : $x U_{x,n} \backslash G_n \rightarrow V_x$. Notons le $\mathcal{T}_{HT, x,n}$. C'est une réduction de structure du torseur $\mathcal{T}_{HT}\vert_{V_x}$. 

Il existe un ouvert $U_{K_p} \subseteq X_{K_pK^p}$ tel qu'on ait $\pi_{HT}^{-1}(V_x) = \pi_{K_p}^{-1}(U_{K_p})$ pour $K_p$ assez petit (il suffit $K_p \subseteq G_n(\qq_p)$).  

\begin{lem} Le torseur $\mathcal{T}_{dR,K_p}\vert_{U_{K_p}}$ possède une réduction de structure à un unique $T_{x,n}$-torseur $\mathcal{T}_{dR,K_p,n} \rightarrow U_{K_p}$ tel que l'isomorphisme $$ \mathcal{T}_{dR,K_p} \times_{X_{K_p K^p}} X_{K^p} = \mathcal{T}_{HT} \times_{\mathcal{FL}} X_{K^p}$$
induise un isomorphisme : $$ \mathcal{T}_{dR,K_p,n} \times_{U_{K_p}} \pi_{HT}^{-1}(V_x) = \mathcal{T}_{HT,x,n} \times_{V_x} \pi_{HT}^{-1}(V_x)$$
\end{lem}

\begin{proof} On possède un ouvert (topologique) $\vert \mathcal{T}_{HT,x,n} \times_{V_x} \pi_{HT}^{-1}(V_x)\vert \hookrightarrow  \vert \mathcal{T}_{dR,K_p} \times_{U_{K_p}} \pi_{HT}^{-1}(V_x)\vert$. Comme cet ouvert est $K_p$-invariant, il provient d'un ouvert $\vert \mathcal{T}_{dR,K_p,n} \vert \hookrightarrow \vert \mathcal{T}_{dR,K_p}\vert_{U_{K_p}}\vert$. On note $\mathcal{T}_{dR,K_p,n} $ l'espace adique associé à cette ouvert. On vérifie alors que cet espace est stable sous l'action de $T_{x,n}$ et que le morphisme $T_{x,n} \times \mathcal{T}_{dR,K_p,n}  \rightarrow \mathcal{T}_{dR,K_p,n} \times T_{dR,K_p,n}$ est un isomorphisme. Pour le vérifier on peut changer de base à $X_{K^p}$ et c'est évident.
\end{proof}

\begin{coro}\label{coro-Sen-rang1} Soit $\mathscr{F} \in \mathbf{Coh}_{G_n}(V_x)$ un objet de rang $1$. Alors $\tilde{VB}(\mathscr{F})$ est un faisceau inversible de $\oscr^{sm}_{X_{K^p}}\vert_{\pi_{HT}^{-1}(V_x)}$-modules. De plus, les opérateurs de Sen associés à $\mathscr{F}$ sont triviaux.  
\end{coro}
\begin{proof} Le faisceau  inversible $\mathscr{F}$ est associé, par la proposition \ref{prop-rep-torseurs}, à une représentation de $B_{x,n}$ de rang $1$ qui se factorise en un caractère  $\chi : T_{x,n} \rightarrow \mathbb{G}_m^{an}$. En utilisant le torseur $\mathcal{T}_{dR,K_p,n}$, on peut associer à tout character $\chi$ de $T_{x,n}$ un faisceau inversible de $\oscr_{U_{K_p}}$-modules $\mathscr{F}_{\chi, K_p}$ et le lemme précédent implique qu'on possède un isomorphisme  $$\mathscr{F}_{\chi,K_p} \otimes_{\oscr_{U_{K_p}}} \oscr_{X_{K^p}}\vert_{\pi_{HT}^{-1}(V_x)} = \pi_{HT}^{-1} \mathscr{F} \otimes_{\pi_{HT}^{-1}\oscr_{V_x}} \oscr_{X_{K_p}}\vert_{\pi_{HT}^{-1}(V_x)}$$ qui est $K_p$-équivariant. Il en résulte que $\mathscr{F}_{\chi,K_p}= \tilde{VB}_{K_p}(\mathscr{F})$. Il est aussi évident que l'action de $\Gamma$ sur les modules de Sen associé à $\mathscr{F}$ est triviale.
\end{proof}


\subsubsection{L'extension de Faltings}

L'extension de Faltings (\cite{MR909221},  thm 2, (ii))  est une suite exacte  $G(\qq_p)$-équivariante:
$$ 0 \rightarrow \oscr_{X_{K_p}} \rightarrow FE \rightarrow \oscr_{X_{K^p}} \otimes \Omega^1_{X_{K_pK^p}}(\mathrm{log}(cusp)) \rightarrow 0 $$
En prenant la cohomologie de $K_p$, on récupère un isomorphisme:
$$ \omega^1_{X_{K_pK^p}} (\mathrm{log}(cusp)) \rightarrow \HH^1(K_p, (\pi_{K_p})_\star \oscr_{X_{K^p}}). $$
Notons $St$ la représentation standard de dimension $2$ de ${G}$. On possède également la suite exacte:
$$ 0 \rightarrow \oscr^{(0,1)}_{\mathcal{FL}} \rightarrow \oscr_{\mathcal{FL}} \otimes_{\qq_p} St \rightarrow  \oscr^{(1,0)}_{\mathcal{FL}} \rightarrow 0$$
On remarque que  \begin{eqnarray*}
\oscr^{(0,1)}_{\mathcal{FL}}&=& (\oscr_{\mathcal{FL}} \otimes_{\qq_p} St)^{\mathfrak{n}^0}, \\
\oscr^{(1,0)} \otimes \mathfrak{n_0}^\vee &=&(\oscr_{\mathcal{FL}} \otimes_{\qq_p} St)_{\mathfrak{n}^0}.
\end{eqnarray*}
En tirant en arrière via $\pi_{HT}$, on obtient une suite exacte :
$$  0 \rightarrow \omega_{K_p}^{(0,1)} \otimes \oscr_{X_{K^p}} \rightarrow \oscr_{X_{K^p}} \otimes_{\qq_p} St \rightarrow 
 \omega_{K_p}^{(1,0)} \rightarrow 0$$
 et en tordant par $\omega_{K_p}^{(0,-1)}$, on trouve une extension : 
$$  0 \rightarrow  \oscr_{X_{K^p}} \rightarrow \oscr_{X_{K^p}} \otimes_{\qq_p} St \otimes \omega_{K_p}^{(0,-1)}  \rightarrow 
 \omega_{K_p}^{(1,0)} \otimes \omega_{K_p}^{(0,-1)} \rightarrow 0.$$

 \begin{thm}[\cite{MR909221}, thm. 5; \cite{pan2021locally}, thm. 4.4.2] L''inverse de l'isomorphisme de Kodaira-Spencer $-KS : \omega_{K_p}^{(1,0)} \otimes \omega_{K_p}^{(0,-1)} \rightarrow  \Omega^1_{X_{K_pK^p}}(\mathrm{log}(cusp)) $ induit un isomorphisme de suite exactes :
 $$ \oscr_{X_{K^p}} \otimes_{\qq_p} St \otimes \omega_{K_p}^{(0,-1)} \rightarrow FE$$
 \end{thm}

 \begin{coro}\label{coro-calculdeFalt} On a $\tilde{VB}(\oscr_{\mathcal{FL}} \otimes_{\qq_p} St) = \omega^{(0,1),sm}$ et $\mathrm{R}^1\tilde{VB}((\oscr_{\mathcal{FL}} \otimes_{\qq_p} St) =  \omega^{(1,0),sm} \otimes \tilde{VB}((\mathfrak{n}^0)^\vee)$.
\end{coro}
\begin{proof} Cela résulte du calcul de la cohomologie de l'extension de Faltings rappelé plus haut.
\end{proof}
 
On rappelle qu'on a $\mathfrak{g} = \mathfrak{g}^{der} \oplus \mathfrak{g}^{ab}$.  On possède un morphisme $\mathfrak{g}$-équivariant  $\mathfrak{n}^0 \rightarrow \oscr_{\mathcal{FL}} \otimes_{\qq_p} \mathfrak{g}^{der}$. Il induit un morphisme $\tilde{VB}(\mathfrak{n}^0)   \rightarrow \tilde{VB}( \oscr_{\mathcal{FL}} \otimes_{\qq_p} \mathfrak{g}^{der})$. 
 
 \begin{coro} Le morphisme $\tilde{VB}(\mathfrak{n}^0)   \rightarrow \tilde{VB}( \oscr_{\mathcal{FL}} \otimes_{\qq_p} \mathfrak{g}^{der})$ est un isomorphisme.
 \end{coro}
 \begin{proof}    On possède une filtration $ \mathfrak{n}^0 \subseteq (\mathfrak{b}^{der})^0 \subseteq \oscr_{\mathcal{FL}} \otimes_{\qq_p} \mathfrak{g}^{der}$ et les extensions entre les gradués successifs de cette filtration s'identifient à l'extension de Faltings. 
On déduit donc que $\tilde{VB}(\mathfrak{n}^0) \rightarrow \tilde{VB}( \oscr_{\mathcal{FL}} \otimes_{\qq_p} \mathfrak{g}^{der})$ est un isomorphisme.
\end{proof}

 Pour tout ouvert $W$ assez petit de $X_{K_p}$,  on a construit  un opérateur de Sen $\Theta \in \tilde{VB}( \oscr_{\mathcal{FL}} \otimes_{\qq_p} \mathfrak{g})(W)$ associé au faisceau $\oscr_{\mathcal{FL}}\otimes {St}$ (voir le corollaire \ref{coro-op-Sen}).
 
 \begin{lem} L'opérateur de Sen $\Theta$ appartient à $\tilde{VB}( \oscr_{\mathcal{FL}} \otimes_{\qq_p} \mathfrak{g}^{der})(W)$.
 \end{lem}
 
 \begin{proof} On considère le morphisme $\det : {St} \otimes {St} \rightarrow \det$ de $G$-représentations. Soit $\Theta$ un opérateur de Sen pour $\oscr_{\mathcal{FL}} \otimes_{\qq_p} {St} $.  Alors $\Theta \otimes 1 + 1 \otimes \Theta$ est un opérateur de Sen pour $\oscr_{\mathcal{FL}} \otimes_{\qq_p} {St} \otimes \mathrm{St}$  et $\mathrm{Tr}(\Theta)$ est   un opérateur de Sen pour $\oscr_{\mathcal{FL}} \otimes \det$,  donc $\mathrm{Tr}(\Theta)=0$ (par exemple d'après le corollaire \ref{coro-Sen-rang1}). Il en résulte que  $\Theta$ appartient à $\tilde{VB}( \oscr_{\mathcal{FL}} \otimes_{\qq_p} \mathfrak{g}^{der})(W)$.
 \end{proof}

\begin{prop}\label{prop-calcul-operaSenST} L'opérateur de Sen $\Theta \in \tilde{VB}( \oscr_{\mathcal{FL}} \otimes_{\qq_p} \mathfrak{g}^{der})(W) = \tilde{VB}(\mathfrak{n}^0)(W)$ engendre  $\tilde{VB}(\mathfrak{n}^0)(W)$.
\end{prop}

\begin{proof} On a  $\mathrm{R}^1\tilde{VB}((\oscr_{\mathcal{FL}} \otimes_{\qq_p} St) = \colim_{K_p} \omega^{(1,0)}_{K_p} \otimes \Omega^1_{X_{K_pK^p}}(\mathrm{log})$ par le corollaire \ref{coro-calculdeFalt}. 
D'après la théorie de Sen, on a  $$\mathrm{R}^1\tilde{VB}(\oscr_{\mathcal{FL}} \otimes_{\qq_p} St)(W) = \HH^0( \Gamma,  \colim_{K_p,n}\HH^1(\Theta,   D_{K_p,n}(St))),$$ et par descente étale, pour $K_p$ suffisament petit,

$$ B_{K_p,n} \otimes_{B_{K_p}} (\HH^1(\Theta,  D_{K_p,n}(St)))^{\Gamma} = \HH^1(\Theta,  D_{K_p,n}(St))$$

Ceci montre que $\HH^1(\Theta,  D_{K_p,n}(St))$ est un faisceau localement libre de rang $1$, et donc que l'opérateur $\Theta$ est non nul en tout point de $W$ (puisque $D_{K_p,n}(St)$ est de rang $2$). 


Il en résulte  que $\Theta$ est un générateur du module projectif de rang $1$, $\tilde{VB}(\mathfrak{n}^0)(W)$.
\end{proof}


\subsubsection{Fin de la démonstration}

On possède un morphisme $\mathfrak{n}^0 \rightarrow \mathscr{F} \otimes_{\oscr_{\mathcal{FL}}} \mathscr{F}^\vee$ qui induit un morphisme $\tilde{VB}(\mathfrak{n}^0) \rightarrow \tilde{VB}(\mathscr{F} \otimes_{\oscr_{\mathcal{FL}}} \mathscr{F}^\vee)$. 

\begin{prop}\label{prop-calcul-Senopfinal} Le conoyau du morphisme $\tilde{VB}(\mathfrak{n}^0) \rightarrow \tilde{VB}(\mathscr{F} \otimes_{\oscr_{\mathcal{FL}}} \mathscr{F}^\vee)$ est un faisceau localement libre et donc l'image du faisceau $\tilde{VB}(\mathfrak{n}^0) \rightarrow \tilde{VB}(\mathscr{F} \otimes_{\oscr_{\mathcal{FL}}} \mathscr{F}^\vee)$ est un faisceau localement libre de rang $0$ ou $1$ qui est localement engendré par un opérateur de Sen de $\mathscr{F}$. 
\end{prop}

\begin{proof}  Il suffit de démontrer tous les énoncés au voisinage des points classiques d'après la proposition \ref{prop-coherence}.  On considère un point classique $x \in \mathcal{FL}$ et un ouvert  de la forme  $V_x = x B_{n,x} \backslash G_{n,x}$. Il en résulte que $\mathscr{F}\vert_{V_x}$ est décrit par une représentation de dimension finie $\rho$ de $B_{x,n}$ (voir la proposition \ref{prop-rep-torseurs}). On peut la supposer irréductible. On a donc $\rho = \chi \otimes \mathrm{Sym}^k {St}$ pour un caractère $\chi$ du tore $T_{x,n}$ de $B_{x,n}$. 
On a donc $\mathscr{F} = \mathscr{F}_{\chi} \otimes \mathscr{F}_{\mathrm{Sym}^k}$ 
L' opérateur de Sen pour $\mathscr{F}$ est de la forme $1 \otimes \Theta$ où $\Theta$ est un opérateur de Sen pour $\mathrm{Sym}^k {St}$ lequel est déterminé par une combinaison du lemme \ref{lem-functor} et du calcul pour ${St}$ (donné dans la proposition \ref{prop-calcul-operaSenST}). 
Il en résulte bien que le morphisme 
$\tilde{VB}(\mathfrak{n}^0)(\pi_{HT}^{-1}(V_x)) \rightarrow \tilde{VB}(\mathscr{F} \otimes_{\oscr_{\mathcal{FL}}} \mathscr{F}^\vee)(\pi_{HT}^{-1}(V_x))$ a une image  localement libre de rang $0$ (si $k=0$) ou $1$ (si $k\neq 0$) et qu'il est engendré par un opérateur de Sen. 
\end{proof}

On déduit maintenant le théorème en combinant la proposition \ref{prop-calcul-Senopfinal} et le lemme \ref{lem-calculcoho}.  

\section{Vecteurs localement analytiques}

Dans cette section on explique  certains résultats de la section 4 de  \cite{pan2021locally} d'après le point de vue de cet article. 

Notons $\oscr_G^{alg}$ l'algèbre des fonctions sur le schéma $GL_2^{alg}$. Elle admet la description (comme $G$-bimodule) $\oscr_G^{alg}= \oplus_{V \in \mathbf{IrrRep}(G)} V \otimes V^\vee$ où $\mathbf{IrrRep}(G)$ désigne les classes d'isomorphisme de représentations irréductibles de dimension finie de $G$.

Pour tout $n \geq 0$, on note $\oscr_{G_n}$ l'espace des fonctions analytiques sur $G_n$, le sous-groupe des éléments $G$ qui se réduisent sur $1$ modulo $p^n$.  On possède un morphisme $\oscr_G^{alg} \rightarrow \oscr_{G_n}$ et $\oscr_{G_n}$ est la complétion de $\oscr_{G}^{alg}$ pour une norme $\vert - \vert_n$ (la norme spectrale sur $\oscr_{G_n}$). On remarque que $\oscr_{G_n}$ est  un $G_n$-bi-module. Pour tout $f \in \oscr_{G_n}$ et $g \in G_n$, on pose 
\begin{eqnarray*}
g \star_1 f(-) & =& f(g^{-1}-)\\
g\star_2 f(-) &=& f(-g)
\end{eqnarray*}


On considère le faisceau $\oscr_{G_n} \hat{\otimes} \oscr_{\mathcal{FL}}$. On possède une action $\star_3$ de $G$ sur $\oscr_{\mathcal{FL}}$. Munis de l'action $\star_{1,3}$ (composée des $\star_{1}$ et $\star_{3}$), le faisceau $\oscr_{G_n} \hat{\otimes} \oscr_{\mathcal{FL}}$ est un faisceau $G_n$-équivariant.  Il possède une seconde action $\star_2$ de $G_n$ qui est $\oscr_{\mathcal{FL}}$-linéaire.  
Pour tout $f \in \oscr_{G_n} \hat{\otimes} \oscr_{\mathcal{FL}}$ (vu comme un fonction $G_n \rightarrow \oscr_{\mathcal{FL}}$), on a
 \begin{eqnarray*}
g \star_{1,3} f(-) & =&  g f(g^{-1}-), ~\forall g \in G_n\\
g\star_2 f(-) &=& f(-g),~\forall g \in G_n\\
\end{eqnarray*}

L'action $\star_{1,3}$ se dérive et induit une action de $\mathfrak{g}$ et donc une action ($\oscr_{\mathcal{FL}}$-linéaire)
$$ \mathfrak{n}^0 \rightarrow \mathrm{End}_{\oscr_{\mathcal{FL}}}(\oscr_{G_n} \hat{\otimes} \oscr_{\mathcal{FL}})$$
On définit   $\mathcal{C}^{n-an}  = \HH^0(\mathfrak{n}^0, \oscr_{G_n} \hat{\otimes} \oscr_{\mathcal{FL}})$.  Ce faisceau est stable sous les actions $\star_{1,3}$, $\star_{2}$. 
L'action $\star_{1,3}$ induit  aussi une action  $\star_{hor}$ du   Cartan horizontal  $\mathfrak{h} \hookrightarrow \mathfrak{b}^0/\mathfrak{n}^0$.

En passant à la limite, on définit l'espace des germes de fonctions analytiques $ \oscr_{G,1} = \colim_n \oscr_{G_n}$ et on définit le  faisceau $$\mathcal{C}^{la} = \colim_n \mathcal{C}^{n-an} = \HH^0(\mathfrak{n}_0, \oscr_{G,1}  \hat{\otimes}_{\C_p} \oscr_{\mathcal{FL}}).$$ La formule $h \star_{1,2,3} f(-)  = hf(h^{-1}- h)$ définit une action de $G(\qq_p)$.

 La formule définissant le foncteur $VB$ s'étend naturellement à tous les faisceaux $G_n$-équivariants  et par passage à la limite à $\mathcal{C}^{la}$.

Par définition, $\oscr^{la} =  \colim_{K_p} (\oscr_{G,1} \hat{\otimes}_{\C_p} \hat{\oscr})^{K_p}= VB ( \oscr_{G,1}  \hat{\otimes}_{\C_p} \oscr_{\mathcal{FL}})$.  Le théorème qui suit est une formulation agréable de certains résultats de la section 4 de \cite{pan2021locally}.

\begin{thm}\label{thm3}
On a $\oscr^{la} = VB( \mathcal{C}^{la})$ et on  possède un isomorphisme canonique, $\hat{\oscr}$-linéaire, $\mathfrak{g}$ et $G(\qq_p)$-équivariant:
$$ \Psi :  \oscr^{la} \hat{\otimes}_{\oscr^{sm}} \hat{\oscr} = \mathcal{C}^{la} \hat{\otimes}_{\oscr_{\mathcal{FL}}} \hat{\oscr}$$
\end{thm}

\subsubsection{Démonstration du théorème, d'après \cite{pan2021locally}}
\begin{prop} Il existe $K_p(n)$ assez petit et un faisceau $\tilde{VB}_{K_p(n)}(\oscr_{G_n} \otimes \oscr_{\mathcal{FL}})$ sur $X_{K_p(n)K^p}$  tel que :
$$\tilde{VB}_{K_p(n)}(\oscr_{G_n} \otimes \oscr_{\mathcal{FL}}) \hat{\otimes}_{\oscr_{X_{K_p(n)K^p}}} \oscr_{X_{K^p}} = \mathcal{C}^{n-an} \hat{\otimes}_{\oscr_{\mathcal{FL}}} \oscr_{X_{K^p}}$$
\end{prop}
\begin{proof}
On applique   la théorie de Sen simultanément aux faisceaux  $G$-équivariants $V \otimes V^\vee \otimes_{\C_p} \oscr_{\mathcal{FL}}$ pour $V \in \mathbf{IrrRep}(G)$. Plus précisément, on équipe chaque représentation $V\otimes V^\vee$ de la norme $\vert-\vert_n$. La boule unité est une représentation $V^+ \otimes V^{\vee,+}$ et on voit que $G_{n+1}(\qq_p)$, agit trivialement sur $V^+ \otimes V^{\vee,+}/p$ (l'action est celle sur le second facteur $V^{\vee,+}$). On a alors pour $K_p(n) = G_{n+1}(\qq_p)$ que:

$$\tilde{VB}_{K_p(n)}(\oscr_{G_n} \otimes \oscr_{\mathcal{FL}}) = \widehat{ \bigoplus_{V \in \mathbf{IrrRep}(G)} \tilde{VB}_{K_p(n)}( V \otimes V^{\vee} \otimes \oscr_{\mathcal{FL}})}$$
où chaque $V\otimes V^\vee$ a été muni de la norme $\vert - \vert_n$ ce qui fournit une norme sur $\tilde{VB}_{K_p(n)}( V \otimes V^{\vee} \otimes \oscr_{\mathcal{FL}})$ par rapport à laquelle on complète. 
On déduit par passage à la limite dans le théorème \ref{thm2},   $$\widehat{\big( \bigoplus_{V \in \mathbf{IrrRep}(G)} \tilde{VB}_{K_p(n)}( V \otimes V^{\vee} \otimes \oscr_{\mathcal{FL}})\big)} \hat{\otimes}_{\oscr_{X_{K_p(n)K^p}}} \oscr_{X_{K^p}} = \mathcal{C}^{n-an} \hat{\otimes}_{\oscr_{\mathcal{FL}}} \oscr_{X_{K^p}}.$$
\end{proof}

On décrit  à présent sur chaque ouvert standard affinoide $U  \hookrightarrow \mathcal{FL}$, le faisceau $\tilde{VB}_{K_p(n)}(\oscr_{G_n} \otimes \oscr_{\mathcal{FL}}) \hat{\otimes}_{\oscr_{X_{K_p(n)K^p}}} \oscr_{X_{K^p}}$. 

Fixons une section de la projection $G \rightarrow \mathcal{FL}$ au dessus de  $U$. On note $x$ l'élément universel de $G$ au-dessus de $U$, déterminé par la section. 

\begin{lem} On possède un isomorphisme  $\mathcal{C}^{n-an}\vert_{U} = \oscr_U \langle x_1,x_2,x_3\rangle$. 
\end{lem}
\begin{proof} Voyons une section de $\mathcal{C}^{n-an}$ comme une fonction analytique $f(g,x)$ où $g \in G_n$ et $f(n_x g , x) = f(g,x)$ pour tout $n_x \in U_x = x^{-1}U_nx$. On construit alors un isomorphisme : $\mathcal{C}^{n-an}\vert_{U} \rightarrow \oscr_{U_n \backslash G_n} \hat{\otimes} \oscr_U$ en envoyant $f(g,x)$ sur $f( x^{-1} gx, x)$. On a alors $\oscr_{U_n \backslash G_n} = \C_p \langle x_1,x_2,x_3\rangle$. 
\end{proof}

Notons $\pi_{HT}^{-1}(U) = \pi_{K_p}^{-1}(U_{K_p})$ pour tout $K_p$ assez petit. 

\begin{lem} Il existe un sous-groupe ouvert compact $K_p(n)' \subseteq K_p(n)$ tel qu'on possède un isomorphisme :
$$\tilde{VB}_{K_p(n)}(\oscr_{G_n} \otimes \oscr_{\mathcal{FL}}) \hat{\otimes}_{\oscr_{X_{K_p(n)K^p}}} \oscr_{U_{K_p(n)}}=  \oscr_{U_{K_p(n)}}\langle x'_1,x'_2,x'_3\rangle$$
\end{lem}
\begin{proof} On possède un isomorphisme   $$\phi : \oscr_{\pi_{HT}^{-1}(U)}\langle x_1,x_2,x_3\rangle \rightarrow \tilde{VB}_{K_p(n)}( \oscr_{G_n} \otimes \oscr_{\mathcal{FL}})\hat{\otimes}_{\oscr_{X_{K_p(n)K^p}}} \oscr_{\pi_{HT}^{-1}(U)}$$
qui est $\oscr_{\pi_{HT}^{-1}(U)}$-linéaire et  $K_p(n)$-équivariant. L'idée est d'approximer les variables $x_1,x_2,x_3$ par des variables qui sont stables pour un sous-groupe ouvert compact $K_p(n)'$. 

On a des inclusions  des faisceaux $ \oscr_{X_{K^p}}^{sm} \otimes_{\oscr_{X_{K_p(n)K^p}}} \oscr_{X_{K^p}}^{sm} \subseteq \tilde{VB}_{K_p(n)}( \oscr_{G_n} \otimes \oscr_{\mathcal{FL}})\hat{\otimes}_{\oscr_{X_{K_p(n)K^p}}} \oscr_{X_{K^p}} \subseteq  \oscr_{X_{K^p}} \otimes_{\oscr_{X_{K_p(n)K^p}}} \oscr_{X_{K^p}}$. On évalue sur $\pi_{HT}^{-1}(U)$ (qui est affinoide pefectoide et satisfait la propriété d'approximation que  $\colim_{K_p} \HH^0(U_{K_p}, \oscr_{X_{K_pK^p}})$ est dense dans $\hat{\oscr}(U)$). Il en résulte qu'il existe un suite d'éléments $x_{i,n} \in  \oscr_{X_{K^p}}^{sm} \otimes_{\oscr_{X_{K_p(n)K^p}}} \oscr_{X_{K^p}}^{sm} (\pi_{HT}^{-1}(U))$ tels que $x_{i,n} \rightarrow x_i$ quand $n \rightarrow \infty$ pour $i=1,2,3$. 

On peut alors définir des applications $$\phi_n : \oscr_{\pi_{HT}^{-1}(U)}\langle x_{1,n},x_{2,n},x_{3,n}\rangle \rightarrow \tilde{VB}_{K_p(n)}( \oscr_{G_n} \otimes \oscr_{\mathcal{FL}})\hat{\otimes}_{\oscr_{X_{K_p(n)K^p}}} \oscr_{\pi_{HT}^{-1}(U)}$$ et $\phi_n \rightarrow \phi$ quand $n\rightarrow \infty$. Il en résulte que $\phi_n$ est un isomorphisme pour $n$ assez grand. Mais il existe un compact $K'_p(n)$ qui fixe $x_{1,n}, x_{2,n}, x_{3,n}$. On conclut en prenant les invariants par $K'_p(n)$. 
\end{proof}

Quitte à rapetisser $K_p(n)$ on peut supposer que $K_p(n) = K_p(n)'$. Posons alors $(\pi_{HT})_\star \tilde{VB}_{K_p(n)}(\oscr_{G_n} \otimes \oscr_{\mathcal{FL}})  =   \oscr^{n-an}_{K_p(n)}$. 
La formule de projection montre  qu'on possède un isomorphisme : 
$$\Psi_n : {\oscr}^{n-an}_{K_p(n)} \hat{\otimes}_{\oscr_{K_{p}(n)}} \hat{\oscr} = \mathcal{C}^{n-an} \hat{\otimes}_{\oscr_{\mathcal{FL}}} \hat{\oscr}.$$
Le morphisme $\Psi$ du théorème \ref{thm3} s'obtient par passage à la limite sur $n$. 

Vérifions à présent l'équivariance pour $G(\qq_p)$ et $\mathfrak{g}$.   Pour tout $h \in G(\qq_p)$, $f$  section locale de $\mathcal{C}^{n-an}$  et $k \in K_p$ assez petit,  on a $$h\star_{1,2,3} . k \star_{1,3} . f = (h k h^{-1}) \star_{1,3} . h \star_{1,2,3}.f$$
Il en résulte qu'on possède une action de $G(\qq_p)$ sur $VB(\mathcal{C}^{la}) = \oscr^{la}$ et le morphisme : $\oscr^{la} \rightarrow  \mathcal{C}^{la} \hat{\otimes}_{\oscr_{\mathcal{FL}}} \hat{\oscr}$ est $G(\qq_p)$-équivariant. 
On possède aussi un morphisme d'évaluation en $1$ (déduit de l'augmentation $\oscr_{G} \rightarrow \C_p$): 
$$ ev_{1} :  \mathcal{C}^{la} \hat{\otimes}_{\oscr_{\mathcal{FL}}} \hat{\oscr} \rightarrow \oscr_{G,1} \hat{\otimes}_{\C_p} \hat{\oscr}
\rightarrow \hat{\oscr}$$
qui est $G(\qq_p)$-équivariant (pour l'action $\star_{1,2,3}$ sur la source). Comme $ev_{1} \circ \Psi : \oscr^{la} \rightarrow \hat{\oscr}$ est l'inclusion naturelle,  on déduit  que l'action de $G(\qq_p)$ construite sur $\oscr^{la}$ via le foncteur $VB$ est bien celle déduite de l'inclusion $\oscr^{la} \subseteq \hat{\oscr}$. 

\begin{lem} Le morphisme $\oscr^{la} \rightarrow \mathcal{C}^{la} \hat{\otimes}_{\oscr_{\mathcal{FL}}} \hat{\oscr}$ induit un morphisme $\mathfrak{g}$-equivariant où $\mathfrak{g}$ agit sur 
$\oscr^{la}$ en dérivant l'action de $G(\qq_p)$ et sur $\mathcal{C}^{la} \hat{\otimes}_{\oscr_{\mathcal{FL}}} \hat{\oscr}$ via l'action $\star_2$. 
\end{lem}
\begin{proof}Sur $\mathcal{C}^{n-an}$, l'action $\star_{1,2,3}$ de $K_p   \subseteq G_n$ se décompose en le  composé  des actions $\star_{1,3}$ et $\star_2$. Il en résulte que l'action $\star_{1,2,3}$ de $K_p$ sur $VB(\mathcal{C}^{n-an})$ se décompose en le composé d'une action lisse  $\star_{1,3}$ (donc de dérivée nule !) et de l'action $\star_2$. \end{proof}

On possède une action $\star_{hor}$ du Cartan horizontal $\mathfrak{h}$ sur $\mathcal{C}^{la}$. Celle-ci induit une action $\star_{hor}$ de $\mathfrak{h}$ sur $VB(\mathcal{C}^{la}) = \oscr^{la}$.

\begin{lem} L'action de $\mathfrak{g}$ sur $\oscr^{la}$ obtenue en dérivant l'action de $G(\qq_p)$ induit une action de $\mathfrak{g}^0$ qui se factorise en une action de $\mathfrak{g}^0/\mathfrak{n}^0$. L'action induite de $\mathfrak{h} \hookrightarrow \mathfrak{g}^0/\mathfrak{n}^0$ est égale à $-\star_{hor}$. 
\end{lem}
\begin{proof} Considérons  le morphisme  d'évaluation en $1$, $ev_1 : \oscr_{G,1} \hat{\otimes}_{\C_p} \hat{\oscr}
\rightarrow \hat{\oscr}$. Pour tout section $f \in  \oscr_{G,1} \hat{\otimes}_{\C_p} \hat{\oscr}
$ et tout  $g \in \mathfrak{g}$,   on a $ev_1 (g \star_1 f) = - ev_1 (g \star_2 f)$.  On linéarise  pour en déduire que pour tout  $g \in \mathfrak{g}^0$,   on a $ev_1 (g \star_1 f) = - ev_1 (g \star_2 f)$.
Soit $f \in \oscr^{la}$,  on obtient $ev_1 (g \star_1 \Psi( f)) = - ev_1 (g \star_2 \Psi( f)) = -ev_1  \circ \Psi( g. f)$. On déduit donc que si $g \in \mathfrak{n}^0$, $g.f = 0$, puis que l'action induite de $\mathfrak{h}$ est $-\star_{hor}$. 
\end{proof}

\section{Cohomologie des faisceaux $\mathcal{C}^{la}$ et $\oscr^{la}$}\label{sectiononcohomology}
Dans cette section, on reprend les calculs de la section 5 de \cite{pan2021locally}. On commence par faire des calculs de cohomologie sur le faisceau $\mathcal{C}^{la}$, puis on transfert ces calculs au faisceau $\oscr^{la}$ à l'aide du théorème \ref{thm3}. 

On adopte la notation suivante: Soit $\chi$ un caractère de $\mathfrak{h}$. Si $M$ est un $\C_p$-espace vectoriel (voir même un faisceau) muni d'une action $\star_{hor} $ du Cartan horizontal $\mathfrak{h}$, on note $M^{\chi} = \mathrm{Hom}_{\mathfrak{h}, \star_{hor}}( \chi, M)$. Si $M$ est un $\C_p$-espace vectoriel (voir même un faisceau) muni d'une action $\star_2 $ de $\mathfrak{h}$ on note $M[\chi] = \mathrm{Hom}_{\mathfrak{h}, \star_2}( \chi, M)$.

\subsection{Cohomologie du faisceau $\mathcal{C}^{la}$}
Pour tout $\chi : \mathfrak{h} \rightarrow \C_p$, on note $\mathcal{C}^{la,\chi}$ le sous-faisceau de $\mathcal{C}^{la}$ donné  par définition $\mathcal{C}^{la, \chi} = \mathrm{Hom}_{\mathfrak{h}, \star_{hor}} ( \chi, \mathcal{C}^{la})$. 
Les faisceaux $\mathcal{C}^{la}$ et $\mathcal{C}^{la,\chi}$ sont $(\mathfrak{g}, G(\qq_p))$-equivariants, et possèdent une action supplémentaire $\star_2$ de $\mathfrak{g}$ (qui commute à l'action de $\mathfrak{g}$ donnée par la structure équivariante). 
On va s'intéresser aux faisceaux de  $\mathfrak{n}$-cohomologie (pour l'action $\star_2$), qui se définissent ainsi: 

$$ 0 \rightarrow \mathcal{C}^{la, \mathfrak{n}} \rightarrow \mathcal{C}^{la} \stackrel{\mathfrak{n}} \rightarrow \mathcal{C}^{la} \otimes \mathfrak{n}^\vee \rightarrow \mathcal{C}^{la}_{\mathfrak{n}} \rightarrow 0$$ et 
$$ 0 \rightarrow \mathcal{C}^{la, \chi, \mathfrak{n}} \rightarrow \mathcal{C}^{la, \chi} \stackrel{\mathfrak{n}} \rightarrow \mathcal{C}^{la,\chi} \otimes \mathfrak{n}^\vee \rightarrow \mathcal{C}^{la,\chi}_{\mathfrak{n}} \rightarrow 0$$ 

Les faisceaux de cohomologie  $\mathcal{C}^{la, \mathfrak{n}}, \mathcal{C}^{la, \chi, \mathfrak{n}}, \mathcal{C}^{la}_{ \mathfrak{n}}$ et $\mathcal{C}^{la,\chi}_{ \mathfrak{n}}$ sont $(\mathfrak{g}, B(\qq_p))$-equivariants et  possèdent une action $\star_2$ induite de $\mathfrak{h} = \mathfrak{b}/\mathfrak{n}$. 

Notons $j = \mathcal{FL} \setminus \{\infty\} = U_{w_0} \hookrightarrow \mathcal{FL}$ l'inclusion et $i_{\infty} : \{ \infty\} \hookrightarrow \mathcal{FL}$. On possède des suites exactes: $$0 \rightarrow  j_! j^\star \mathcal{C}^{la} \rightarrow \mathcal{C}^{la} \rightarrow (i_{\infty})_\star i_{\infty}^{-1} \mathcal{C}^{la} \rightarrow 0 $$
et 
$$0 \rightarrow  j_! j^\star \mathcal{C}^{la,\chi} \rightarrow \mathcal{C}^{la,\chi} \rightarrow (i_{\infty})_\star i_{\infty}^{-1} \mathcal{C}^{la,\chi} \rightarrow 0 .$$

On rappelle aussi la définition suivante :

\begin{defi} Un nombre complexe p-adique $\alpha \in \C_p$ est un nombre de Liouville si :
$$\mathrm{liminf}_{n \rightarrow \infty} \vert \alpha + n \vert^{\frac{1}{n}} = 0$$
\end{defi}

On note $\mathrm{Liouv}$ l'ensemble des nombres de Liouville, et $-\mathrm{Liouv}$ son opposé. La proposition suivante décrit essentiellement la $\mathfrak{n}$-cohomologie des faisceaux $\mathcal{C}^{la}$ et $\mathcal{C}^{la,\chi}$.  
 \begin{prop}\label{prop-computingcoho} \begin{enumerate}
 \item On possède deux morphismes   de faisceaux $(\mathfrak{g}, B(\qq_p))$-équivariants  : $ES_{w_0} :  \mathcal{C}^{la}_{w_0} \rightarrow \mathcal{C}^{la}$ et $ES_{\infty} : \mathcal{C}^{la}  \rightarrow \mathcal{C}^{la}_{\infty}$.
 \item Pour tout caractère $\chi : \mathfrak{h} \rightarrow \C_p$, on a $\mathcal{C}^{la, \chi}_{w_0} =  j_! \oscr_{U_{w_0}}^{w_0\chi}$,  $\mathcal{C}^{la, \chi}_{\infty} = (i_{\infty})_\star \oscr_{\infty}^{\omega_0\chi} $ (voir les section \ref{sec-fsurcon} et \ref{sec-fsurcon2}). 
 \item On a $\mathrm{Ext}_{\mathfrak{h}, \star_{hor}}^i ( \chi, \mathcal{C}^{la}) =  \mathrm{Ext}_{\mathfrak{h}, \star_{hor}}^i ( \chi, \mathcal{C}^{la}_\infty) = \mathrm{Ext}_{\mathfrak{h}, \star_{hor}}^i ( \chi, \mathcal{C}^{la}_{w_0})  = 0$ si $i>0$.
 
 \item L'application $ES_{w_0}$ se factorise en un morphisme 
 $ES_{w_0} : \mathcal{C}^{la}_{w_0} \rightarrow  \mathcal{C}^{la, \mathfrak{n}}$
 On possède une action du Cartan horizontal $\star_{hor}$ sur $\mathcal{C}^{la}_{w_0}$ et le morphisme $ES_{w_0}$   est équivariant pour les actions de $-w_0 \star_{hor}$  et  $\star_2$ de $\mathfrak{h}$. 
 \item L'application $ES_{\infty} $ se factorise en un morphisme 
 $ES_{\infty} : \mathcal{C}^{la}_{\mathfrak{n}} \rightarrow  \mathcal{C}^{la}_{\infty} \otimes \mathfrak{n}^\vee$.  On possède une action du Cartan horizontal $\star_{hor}$ sur $\mathcal{C}^{la}_{\infty}$ et le morphisme $ES_{\infty}$ est équivariant pour les actions  $\star_2$ et  $- \star_{hor} - 2\rho$  de $\mathfrak{h}$.
  \item Si $\chi( (1,-1)) \notin \Z_{\geq 0}$, le morphisme $ES_{w_0}$ induit   un isomorphisme de faisceaux $(\mathfrak{g},B(\qq_p))$-équivariants $$ \mathcal{C}^{la, \chi}_{w_0} \rightarrow \mathcal{C}^{la,\chi, \mathfrak{n}} $$
  l'action de $\star_2$ de $\mathfrak{h}$ se fait à travers le caractère $-w_0\chi$.

  \item Si $\chi( (1,-1)) \in \Z_{\geq 0}$,  on possède un isomorphisme de faisceaux  $(\mathfrak{g},B(\qq_p))$-equivariants $\mathcal{C}^{la,\chi, \mathfrak{n}} =  \oscr_{\mathcal{FL}}^{w_0\chi} \otimes_{\C_p} \C_p(-w_0\chi)$ (le produit tensoriel par $\C_p(-w_0\chi)$ signifie que l'action de $B(\qq_p)$ est tordue),
  et  l'action de $\star_2$ de $\mathfrak{h}$ se fait à travers le caractère $-w_0\chi$. Le morphisme $ES_{w_0}$ se factorise à travers le  morphisme naturel $ j_! \oscr_{U_{w_0}}^{w_0\chi} \rightarrow \oscr_{\mathcal{FL}}^{w_0\chi} \otimes_{\C_p} \C_p(-w_0\chi)$.
  
  \item  Si $\chi( (1,-1)) \notin \Z_{\geq 0} \cup -\mathrm{Liouv}$,  l'application $ES_\infty$ induit un isomorphisme de faisceaux $(\mathfrak{g},B(\qq_p))$-équivariants 
  $$ES_\infty : \mathcal{C}^{la,\chi}_{ \mathfrak{n}}  \rightarrow \mathcal{C}^{la,\chi}_\infty \otimes \mathfrak{n}^\vee$$
  l'action de $\star_2$ de $\mathfrak{h}$ se fait à travers le caractère $-\chi-2\rho$.
  \item  Si $\chi( (1,-1)) \in \Z_{\geq 0} $,  on possède une suite exacte   $$0 \rightarrow (i_{\infty})_\star \oscr_{\infty}^{\omega_0\chi} \otimes_{\C_p} \C_p(\chi-w_0\chi) \rightarrow  \mathcal{C}^{la,\chi}_{ \mathfrak{n}}  \stackrel{ES_{\infty}}\rightarrow \mathcal{C}^{la,\chi}_\infty \otimes \mathfrak{n}^\vee.$$ L'action   $\star_2$ de $\mathfrak{h}$ se fait à travers le caractère $-w_0\chi$ sur le premier facteur et  $-\chi-2\rho$  sur  le second facteur.
  \end{enumerate}
  \end{prop}

La preuve de ce résultat est donnée dans les sections suivantes. 

\subsubsection{Calcul de la  $\mathfrak{n}$-cohomologie sur l'ouvert $U_{w_0}$}
Soit $B^0 \hookrightarrow G \times \mathcal{FL}$ le sous-groupe de Borel universel. Notons $U^0 \subseteq B^0$ le radical unipotent universel et $T^0 = U^0\backslash B^0$.  Soit $ e : \mathcal{FL} \rightarrow B^0$ la section neutre. On note $\oscr_{B^0,1}= e^{-1} \oscr_{B^0}$. C'est le faisceau des germes de fonctions analytiques sur $B^0$ au voisinage de la section neutre. On définit de même $\oscr_{U^0,1}$ et  $\oscr_{T^0,1}$. On note aussi $\oscr_{U,1}$ le $\C_p$-espace vectoriel des germes de fonctions analytiques en l'identité sur $U$. 


\begin{lem} Le morphism :
$$ B^0 \times U \rightarrow G\times \mathcal{FL}$$ induit sur l'ouvert $U_{w_0}$ un isomorphisme :
$$ \oscr_{G,1} \otimes \oscr_{U_{w_0}}= (\oscr_{B^0,1}\hat{ \otimes} \oscr_{U,1})\vert_{U_{w_0}}$$
\end{lem}

\begin{proof} En tout point, $x \in U_{w_0}$, le morphisme $\mathfrak{b}_x \oplus \mathfrak{n} \rightarrow \mathfrak{g}$ est un isomorphisme.
\end{proof}  

\begin{coro}\label{coro-calculcoho}  On a :
\begin{enumerate} 
\item $j^\star \mathcal{C}^{la, \mathfrak{n}} \simeq j^\star \oscr_{T^0,1}$ de plus, pour tout $h \in \mathfrak{h}$ on a $h \star_{hor}( -)  = -w_0h \star_2( -)$,
\item  $j^\star \mathcal{C}^{la}_{\mathfrak{n}} = 0$,
\item  On possède un isomorphisme de faisceaux $(\mathfrak{g}, B(\qq_p))$-équivariants 
$$j^\star \mathcal{C}^{la, \chi, \mathfrak{n}}  = \oscr_{U_{w_0}}^{w_0\chi}$$
et l'action $\star_2$ de $\mathfrak{h}$ est donnée par le caractère $-w_0\chi$. 
\item  $j^\star \mathcal{C}^{la, \chi}_{\mathfrak{n}} =0$.
\end{enumerate}
\end{coro}
\begin{proof} Il est évident que $j^\star \mathcal{C}^{la, \mathfrak{n}} = j^\star \oscr_{T^0,1}$ et $j^\star \mathcal{C}^{la}_{\mathfrak{n}} = 0$ vu la description du lemme. Posons $x = w_0    \begin{pmatrix} 
      1 & u \\
      0 & 1 \\
   \end{pmatrix}$. 

Un  calcul direct montre que pour tout  $h \in \mathfrak{h}$ on a $ h - x^{-1}  h x \in \mathfrak{n}$. Cela montre immédiatement que $h \star_{hor}- = -w_0h \star_2 -$.

 Il est clair que $\mathcal{C}^{la, \chi, \mathfrak{n}}$ est un fibré inversible 
Comme il est de poids $\chi$, il en résulte que $\mathcal{C}^{la, \chi, \mathfrak{n}} \simeq \oscr_{U_{w_0}}^{w_0\chi}$ comme faisceau $\mathfrak{g}$-équivariant. Vérifions que c'est un isomorphisme de faisceaux $(\mathfrak{g}, B(\qq_p))$-équivariant. Il suffit de calculer l'action de $T(\qq_p)$ en la fibre en $w_0$ de $\mathcal{C}^{la, \chi, \mathfrak{n}}$ et de voir qu'elle est triviale. Une section de $\mathcal{C}^{la, \chi, \mathfrak{n}}$ est une fonction $f(g,x)$ tel que $f(b^{-1}_xg,x) = \chi(h_x)f(g,x)$ et $f(gn,x) = 0$  pour $b_x = h_x n_x \in \mathfrak{b}_x$, $n\in \mathfrak{n}$, $g \in G$ arbitrairement proche de $1$ et $x \in U_{w_0}$.  Soit $t \in T(\qq_p)$. On a $t. f(g,w_0) = f(t^{-1} g t, w_0)$. Posons $g = u_{w_0}t_{w_0} u$ avec $b_{w_0} = u_{w_0} t_{w_0} \in B_{w_0}$ et $u \in U$. On voit alors que $f(t^{-1} g t, w_0) = f( t^{-1} t_{w_0} t , w_0) = f(g,w_0)$.

\end{proof}

On pose donc $ ES_{w_0} : \mathcal{C}^{la}_{w_0} = j_! j^\star \mathcal{C}^{la,\mathfrak{n}} \rightarrow \mathcal{C}^{la}$. 
\subsubsection{Un isomorphisme utile}
 Notons $\oscr_{U \backslash G,1}$ les germes de fonctions sur $U \backslash G$ au voisinage de $U.1$.  Notons $\oscr_{U \backslash G,1}^{\chi}$ le sous-espace des fonctions $f$ qui vérifient $h. f = f( h^{-1}-) = \chi(h) f$. On observe que $\oscr_{U \backslash G,1}^{\chi=0} = \oscr_{\mathcal{FL}, \infty}$ . Les espaces $\oscr_{U \backslash G,1}$ et $\oscr_{U \backslash G,1}^{\chi}$ possèdent une action  $\star_d$ de $\mathfrak{g}$ obtenue en dérivant la translation à droite sur $G$ et une action $\star_g$  de $\mathfrak{h}$ obtenue en dérivant l'action par translation à gauche de $T$ sur $G$.

 On note aussi qu'on possède une action $\star_{gd}$ de $B(\qq_p)$, donnée par $b.f(-) = f(b^{-1}-b)$. 

\begin{lem} On possède des isomorphismes naturelles:  

\begin{enumerate}
\item $i^{-1}_{\infty} \mathcal{C}^{la} \hat{\otimes}_{\oscr_{\mathcal{FL}, \infty}} \oscr_{U\backslash G,1} \rightarrow   \oscr_{U \backslash G,1} \hat{\otimes}_{\C_p} \oscr_{U \backslash G,1}$.
\item $i^{-1}_{\infty} \mathcal{C}^{la,\chi} \hat{\otimes}_{\oscr_{\mathcal{FL}, \infty}} \oscr_{U\backslash G,1} \rightarrow \oscr_{U \backslash G,1}^{\chi} \hat{\otimes}_{\C_p} \oscr_{U \backslash G, 1}$
\end{enumerate}
Ces isomorphismes entrelacent les actions $\star_{1,3} \otimes \star_d$ et $Id \otimes  \star_d$ de $\mathfrak{g}$,  et les actions $\star_2$ et   $\star_d \otimes Id$. 
Ils entrelacent les actions $\star_{1,2,3} \otimes \star_{gd}$ et $\star_{gd} \otimes \star_{gd}$ de $B(\qq_p)$. L'action de $\mathfrak{h}$ donnée par $\star_{hor} \otimes Id$ devient $\star_g \otimes Id$. 
Enfin les actions $Id \otimes \star_g$ et $\star_g \otimes \star_g$ de $\mathfrak{h}$ se correspondent aussi. 
\end{lem}

\begin{proof} Voyons une section de $\mathcal{C}^{la}$ comme une fonction $f(g,x)$ avec $g \in G$ et $x \in \mathcal{FL}$ et voyons donc une section de $i^{-1}_{\infty} \mathcal{C}^{la} \otimes_{\oscr_{\mathcal{FL}, \infty}} \oscr_{U\backslash G,1}$ comme une fonction $f(g,x)$ avec $g \in G$, $x \in U\backslash G$  arbitrairement proches de $1$. Au voisinage de $\infty$, on peut relever $x$ en un élément de $G$ tel que $x(\infty) = 1$ (on vérifiera que l'application que nous allons définir ne dépend pas du choix d'un tel relèvement). 
On définit alors l'application $f(g,x) \mapsto f(xg,x)$. On voit  que cette application réalise un isomorphisme
$i^{-1}_{\infty} \mathcal{C}^{la} \otimes_{\oscr_{\mathcal{FL}, \infty}} \oscr_{U\backslash G,1} \rightarrow  \oscr_{U \backslash G,1} \otimes_{\C_p} \oscr_{U \backslash G,1}.$
et la compatibilité aux différentes actions est un simple exercice. 

\end{proof}
\subsubsection{Le morphisme $ES_\infty$}
On possède un morphisme de restriction naturel $\oscr_{U \backslash G, 1} \rightarrow \oscr_{T,1}$, déduit de l'inclusion $T \hookrightarrow U\backslash G$.. On définit  un morphisme surjectif $$i^{-1}_{\infty} \mathcal{C}^{la} \hat{\otimes}_{\oscr_{\mathcal{FL}, \infty}} \oscr_{U\backslash G,1} \rightarrow   \oscr_{U \backslash G,1} \hat{\otimes}_{\C_p} \oscr_{U \backslash G,1} \rightarrow  \oscr_{T,1} \hat{\otimes}_{\C_p} \oscr_{U \backslash G,1}$$

Ce morphisme entrelace les actions $\star_{1,3} \otimes \star_d$ et $Id \otimes  \star_d$ de $\mathfrak{g}$,  et les actions $\star_2 \otimes Id$ et   $\star_d \otimes Id$ de $\mathfrak{h}$. 
Il entrelace les actions $\star_{1,2,3} \otimes \star_{gd}$ et $\star_{gd} \otimes \star_{gd}$ de $B(\qq_p)$. L'action de $\mathfrak{h}$ donnée par $\star_{hor} \otimes Id$ devient $\star_g \otimes Id$. 
Enfin les actions $Id \otimes \star_g$ et $\star_g \otimes \star_g$ de $\mathfrak{h}$ se correspondent aussi. 

On en déduit un morphisme surjectif de faisceaux $(\mathfrak{g}, B(\qq_p))$-équivariants :

$$i^{-1}_{\infty} \mathcal{C}^{la} \rightarrow  (\oscr_{T,1} \hat{\otimes}_{\C_p} \oscr_{U \backslash G,1})^{\mathfrak{h}\star_g \otimes {\star_g} = 0}$$

On note $ \mathcal{C}^{la}_{\infty} = (i_{\infty})_\star (\oscr_{T,1} \hat{\otimes}_{\C_p} \oscr_{U \backslash G,1})^{\mathfrak{h}\star_g \otimes {\star_g} = 0}$. On observe que ce quotient est encore muni d'une action $\star_2$ de $\mathfrak{h}$ et que par construction $\star_2 = - \star_{hor}$. 

On possède une application $ES_\infty :   \mathcal{C}^{la} \rightarrow \mathcal{C}^{la}_{\infty}$. 

\subsubsection{La $\mathfrak{n}$-cohomologie d'un module de Verma dual}

Un élément de $U\backslash G$ au voisinage de $1$ peut être uniquement représenté par un produit de matrices:
 $$  \begin{pmatrix} 
      a & 0 \\
      0 & d \\
   \end{pmatrix} \begin{pmatrix} 
      1& 0 \\
      x & 1\\
   \end{pmatrix}$$
   
 et on voit un élément de $\oscr_{U \backslash G,1}$ comme une serie formelle en les variables $a,d,x$ dont le rayon de convergence est strictement positif.  Dans ces coordonnées, l'action par translation à droite de $U$ est donnée par :  
 
 $$\begin{pmatrix} 
 {a} & 0 \\
      0 & d \\
   \end{pmatrix} \begin{pmatrix} 
      1& 0 \\
      x & 1\\
   \end{pmatrix}. \begin{pmatrix} 
      1& u \\
      0 & 1\\
   \end{pmatrix}   = \begin{pmatrix} 
      \frac{a}{xu+1} & 0 \\
      0 & d(xu+1) \\
   \end{pmatrix} \begin{pmatrix} 
      1& 0 \\
      \frac{x}{xu+1} & 1\\
   \end{pmatrix} $$

 On fixe le générateur standard $n \in \mathfrak{n}$. On vérifie alors 
 $$n . f(a,d,x) =( -ax \partial_a f + dx \partial_d - x^2 \partial_x).f$$ Il en résulte que $n.a^{k_1} d^{k_2} x^{k_3}  = ((k_2-k_1)-k_3) a^{k_1} d^{k_2} x^{k_3+1}.$
 
 Si $f$ est une section de $\oscr_{U\backslash G,1}^{\chi}$ elle est bien determinée par la fonction $g(x) = f(1,1,x)$ et on trouve $$n.g(x) = \chi((1,-1)) x g(x) - x^2 g'(x)$$
  Il en resulte que $n.x^r =(\chi((1,-1))-r)x^{r+1}$. 
  \begin{lem}\label{lem-calculcoho}
  \begin{enumerate}
  \item Si $\chi( (1,-1)) \notin \Z_{\geq 0}$, $\oscr_{U\backslash G,1}^{\chi, \mathfrak{n}} = 0$,
  \item Si $\chi( (1,-1)) \in \Z_{\geq 0}$, $\mathrm{dim}_{\C_p} \oscr_{U\backslash G,1}^{\chi, \mathfrak{n}} =1 $ et l'action $\star_d$ de $\mathfrak{h}$ se fait à travers $-w_0\chi$.
  \item La restriction à $T$ induit  un morphisme $\oscr_{U \backslash G,1} \rightarrow \oscr_{T,1}$ et celui-ci passe au quotient en un morphisme $ (\oscr_{U \backslash G,1})_{\mathfrak{n}} \rightarrow \oscr_{T,1} \otimes \mathfrak{n}^\vee$. 
 \item Si $\chi((1,-1)) \notin \Z_{\geq 0} \cup \mathrm{Liouv}$, $\mathrm{dim}_{\C_p} (\oscr_{U\backslash G,1}^{\chi})_ {\mathfrak{n}} =1$, l'action $\star_d$ de $\mathfrak{h}$ se fait à travers le caractère $-\chi-2\rho$ et l'action $\star_{1,2}$ de $B(\qq_p)$ se fait à travers le caractère $-2\rho$.
 \item Si $\chi( (1,-1)) \in \Z_{\geq 0}$, $\mathrm{dim}_{\C_p} (\oscr_{U\backslash G,1}^{\chi})_{ \mathfrak{n}} =2 $ et l'action  $\star_d$ de $\mathfrak{h}$ se fait à travers les caractères $-w_0 \chi$ et $-\chi-2\rho$, et l'action $\star_{gd}$ de $B(\qq_p)$ se fait à travers les caractères $\chi-\omega_0 \chi$ et  $-2\rho$. 
  \end{enumerate}
  \end{lem}
 
  \begin{proof} Dans la base topologique $1,x, x^2, \cdots$, la matrice de $n$ est sous-diagonale avec coefficients 
  $\chi(1,-1), \chi(1,-1)-1, \chi(1,-1)-2, \cdots$. 
  Le point $(1)$ est clair. Dans le cas $(2)$, on voit que $\oscr_{U\backslash G,1}^{\chi, \mathfrak{n}} = \C_p x^{\chi( (1,-1))}$. Le point $(3)$ signifie que l'image de $\mathfrak{n}$ est inclue dans l'idéal engendré par le monôme $x$. Dans le cas $(4)$, on voit que  $(\oscr_{U\backslash G,1}^{\chi})_ {\mathfrak{n}}$ est engendré par la fonction constante. Dans le cas $(5)$, $(\oscr_{U\backslash G,1}^{\chi})_ {\mathfrak{n}}$ est engendré par la fonction constante et la fonction $x^{\chi( (1,-1))+1}$. Le reste suit facilement. 
  \end{proof}

  \subsubsection{Fin de la démonstration de la proposition \ref{prop-computingcoho}} 
 Il s'agit de  mettre bout à bout les différents résultats qu'on vient d'obtenir. Il reste ensuite simplement le point $(7)$ à vérifier (il faut identifier l'extension $0 \rightarrow  j_! j^\star \mathcal{C}^{la,\chi,\mathfrak{n}} \rightarrow \mathcal{C}^{la,\chi, \mathfrak{n}} \rightarrow (i_{\infty})_\star i_{\infty}^{-1} \mathcal{C}^{la,\chi,\mathfrak{n}} \rightarrow 0$).  On possède naturellement un morphisme injectif $V_\kappa \otimes V_{\kappa}^\vee \rightarrow \oscr_G$.  Il induit un morphisme $V_\kappa \otimes V_{\kappa}^\vee \otimes \oscr_{\mathcal{FL}} \rightarrow \oscr_{G,1} \otimes \oscr_{\mathcal{FL}}$. Prenons les invariants par $\mathfrak{n}^0$ (agissant sur $V_\kappa\otimes \oscr_{\mathcal{FL}}$)  et par $\mathfrak{n}$ (agissant sur $V_\kappa^\vee$). On déduit une application injective  $\oscr_{\mathcal{FL}}^{\omega_0\kappa} \otimes_{\C_p} \C_p(-w_0\kappa)  \rightarrow  \mathcal{C}^{la ,\kappa, \mathfrak{n}}$. C'est clairement un isomorphisme générique, comme c'est un morphisme de faisceau $\mathfrak{g}$-équivariants, c'est un isomorphisme.


\subsection{Poids singuliers}\label{sect-poids-singulier}   On dit qu'un poids $\chi$ est singulier si  $w_0 \chi = \chi + 2\rho$. Autrement dit, $\chi = (k_1,k_2)$ avec $k_2 = k_1 + 1$.  On rappelle que $\mathcal{C}^{la, \chi}_{\mathfrak{n}} \simeq \oscr_{\mathcal{FL}, \infty}$ et que $i_\infty^{-1} \mathcal{C}^{la, \chi}_{\mathfrak{n}} \simeq i_{\infty}^{-1} \oscr_{U_{w_0}}$. 
On possède un isomorphisme  $$\oscr_{\mathcal{FL}, \infty} = \{ f = \sum_{n\geq 0} a_n z^n,~\exists \rho(f) >0,~\lim_{n =\infty}\vert a_n \vert \rho(f)^n = 0\}$$ et un isomorphisme  
$$i_{\infty}^{-1} j^\star \oscr_{U_{w_0}} = \{ f = \sum_{n\geq -\infty}^\infty a_n z^n,~\exists \rho(f) >0,~\lim_{n=\infty}\vert a_n \vert \rho(f)^n= 0,~ \forall r>0, \lim_{n=-\infty}\vert a_n \vert r^{-n}= 0 \}$$ 

On  construit un morphisme canonique de faisceaux $\mathfrak{g}$-équivariants. 
$$d : \mathcal{C}^{la, \chi}_{\mathfrak{n}} \rightarrow i_\infty^{-1}j_\star j^\star\mathcal{C}^{la, \chi, \mathfrak{n}}$$ de la façon suivante. On fixe un générateur $n$ de $\mathfrak{n}$ et on note $n^\vee$ le générateur dual de $\mathfrak{n}^\vee$. Soit $h_0=  \begin{pmatrix} 
      1 & 0 \\
      0 & 0 \\
   \end{pmatrix}$.

Soit $f  \in  \mathcal{C}^{la, \chi}_{\mathfrak{n}}$.  On va vérifier que le morphisme :
$i^{-1}_\infty \mathcal{C}^{la, \chi, \bar{\mathfrak{n}}} \otimes \mathfrak{n}^\vee \rightarrow \mathcal{C}^{la, \chi}_{\mathfrak{n}}$ est un isomorphisme, équivariant pour l'action de $\mathfrak{h}$. 

On possède donc  $\tilde{f} \otimes n^\vee$, un relèvement de $f$ dans $i_\infty^{-1} \mathcal{C}^{la,\chi}$ qui vérifie $ (h_0 + k_1+1)\tilde{f} \otimes n^\vee = 0$. Soit $\mathrm{Res}(\tilde{f} \otimes n^\vee)$ l'image de $\tilde{f} \otimes n^\vee$ par le morphisme $i_\infty^{-1} \mathcal{C}^{la,\chi} \rightarrow j^{-1} i_\infty^{-1} \mathcal{C}^{la,\chi}$.  Le morphisme de faisceaux $$\mathfrak{n} : j^{-1} i_\infty^{-1} \mathcal{C}^{la,\chi} \rightarrow j^{-1} i_\infty^{-1} \mathcal{C}^{la,\chi} \otimes \mathfrak{n}^\vee $$ est surjectif. Localement, on possède donc un élément  $g$  qui vérifie $n g = \mathrm{Res}(\tilde{f} \otimes n^\vee)$.  On pose alors $d(f) = (h_0+k_1+1)  g$. Le lemme qui suit montre que l'élement $d(f)$ est bien défini.

   \begin{lem} Le morphisme $d : \mathcal{C}^{la, \chi}_{\mathfrak{n}}\otimes \mathfrak{n}^\vee \rightarrow i_\infty^{-1}j_\star j^\star \mathcal{C}^{la, \chi, \mathfrak{n}}$  est bien défini. 
   \end{lem}
   
   \begin{proof}   L'élément  $g$ est unique à un élément de $j^{-1} i_\infty^{-1} \mathcal{C}^{la,\chi, \mathfrak{n}}   $ près.  Comme ${h}_0 + k_1+1$ tue  $j^{-1} i_\infty^{-1} \mathcal{C}^{la,\chi, \mathfrak{n}}   $, on trouve donc que $d(f) = (h_0+k_1+1)  g$ est bien défini. C'est un élément de $i_\infty^{-1} \mathfrak{C}^{la, \chi, \mathfrak{n}}$ car $n(h_0+k_1+1)g = (h_0+k_1+1) f \otimes n^\vee= 0$. 
   \end{proof}

\begin{prop}\label{prop-bordmaps} Dans des bases convenables, le morphisme $d :  \mathcal{C}^{la, \chi}_{\mathfrak{n}}\otimes \mathfrak{n}^\vee \rightarrow i_\infty^{-1}\mathcal{C}^{la, \chi, \mathfrak{n}}$ s'identifie à l'inclusion canonique $\oscr_{\mathcal{FL}, \infty}  \rightarrow i_{\infty}^{-1} \oscr_{U_{w_0}}$. 
\end{prop}

\begin{proof}  On considère l'élément, 
$x=  \begin{pmatrix} 
      1 & 0 \\
      u & 1 \\
   \end{pmatrix}$. On y pense comme à la coordonnée au voisinage de $\infty$. 
On a $$U_x = x^{-1} U x =  \{\begin{pmatrix} 
      1+cu & c \\
      -cu^2 & 1-cu \\
   \end{pmatrix}\}$$ et $$T_x = x^{-1} T  x = \{\begin{pmatrix} 
      a & 0 \\
      (b-a)u & b \\
   \end{pmatrix}\}$$

Une section de $i_\infty^{-1}\mathcal{C}^{la}$ s'écrit de la forme $f(g, x) = f(a,b, z, x)$ où 

$$g = \begin{pmatrix} 
      a & 0 \\
      (b-a)u & b \\
   \end{pmatrix} \begin{pmatrix} 
      1 & 0 \\
      z & 1 \\
   \end{pmatrix}$$
   
   La fonction $f$ est bien déterminée par ces valeurs. 
   
   On calcule ensuite que 
   
   $$ \begin{pmatrix} 
      1 & 0 \\
      z & 1 \\
   \end{pmatrix} \begin{pmatrix} 
      1 & t \\
      0& 1 \\
   \end{pmatrix} = \begin{pmatrix} 
      \frac{1}{(z+u)t+1} & 0 \\
      z + \frac{ut(z+u)}{(z+u)t+1}& (z+u)t+1 \\
   \end{pmatrix}  $$

Il en résulte que $$\begin{pmatrix} 
      1 & t \\
      0& 1 \\
   \end{pmatrix}. f(a,b, z, x) = f( \frac{a}{(z+u)t+1}, b((z+u)t+1) , - u + \frac{(z+u)}{t(z+u)+1})$$ puis que 
    $$n f(a,b,z,x) = -a(z+u) \partial_a + b(z+u) \partial_b - (z+u)^2 \partial_z.$$ 
   
  On se restreint maintenant à des éléments de $i_\infty^{-1}\mathcal{C}^{la, \chi}$. On pose donc $g(z,x) = f(1,1, z,x)$. 
  On a $n g(z,x) = -(z+u) g(z,x) - (z+u)^2 \partial_z g(z,x)$ (car $k_1-k_2=-1$). 
  
  On vérifie que $j^{-1}i_\infty^{-1}\mathcal{C}^{la} = \frac{1}{z+u} i_{\infty}^{-1} \oscr_{U_{w_0}}$. 
  
  On calcule maintenant l'action de 
 $h_0$. Un calcul simple montre que 
 
 $$\begin{pmatrix} 
      t_1 & 0 \\
      0& t_2 \\
   \end{pmatrix}. f(a,b, z, x) = f( t_1 a, t_2 b , t_1 t_2^{-1} z +\frac{t_1-t_2}{t_2} u, x)$$ puis que 
   $h_0 f(a,b,z,x) = a \partial_a  +  (z+u) \partial_z$. Finalement, $h_0 g(z,x) = -k_1 g(z,x)  +  (z+u) \partial_z g(z,x)$.
   
   On prend pour $f$ la classe de $1 \otimes n^\vee$ et on a donc   $\tilde{f} = 1 \otimes n^\vee$. Il en résulte que $g$ vérifie l'équation différentielle $n g(z,x) = -(z+u) g(z,x) - (z+u)^2 \partial_z g(z,x) = 1$. On  a la relation  $-(z+u)(h_0 + k_1 + 1) = n$, si bien que $(h_0 + k_1+1) g = - \frac{1}{z+u}$. 

  \end{proof}
  
  \subsection{Cohomologie de $\oscr^{la}$}

  La proposition qui suit est une reformulation de certains résultats de la section 5 de \cite{pan2021locally}. 
  
  \begin{prop}\label{prop-coho-compute-ola} \begin{enumerate}
 
 \item On possède deux morphismes (d'Eichler-Shimura)  de faisceaux   : $ES_{w_0} :  \oscr^{la}_{w_0} \rightarrow \oscr^{la}$ et $ES_{\infty} : \oscr^{la}  \rightarrow \oscr^{la}_{\infty}$.
 \item Pour tout caractère $\chi : \mathfrak{h} \rightarrow \C_p$, on a $\oscr^{la, \chi}_{w_0} =  j_! \omega_{U_{w_0}}^{-w_0\chi, sm}$,  $\oscr^{la, \chi}_{\infty} = (i_{\infty})_\star \omega_{\infty}^{-\omega_0\chi,sm} $. 
 \item On a $\mathrm{Ext}_{\mathfrak{h}, \star_{hor}}^i ( \chi, \oscr^{la}) =  \mathrm{Ext}_{\mathfrak{h}, \star_{hor}}^i ( \chi, \oscr^{la}_\infty) = \mathrm{Ext}_{\mathfrak{h}, \star_{hor}}^i ( \chi, \oscr^{la}_{w_0})  = 0$ si $i>0$.
 
 \item L'application $ES_{w_0}$ se factorise en un morphisme injectif : 
 $ES_{w_0} : \oscr^{la}_{w_0} \rightarrow  \oscr^{la, \mathfrak{n}}$.
 Sur le sous-faisceau $\oscr^{la}_{w_0}$ on a l'identité entre  actions de $\mathfrak{h}$ : pour tout $h \in \mathfrak{h}$,  $w_0 h \star_{hor}(-) = h \star_2(-)$. 
 \item L'application $ES_{\infty} $ se factorise en un morphisme  surjectif : 
 $ES_{\infty} : \oscr^{la}_{\mathfrak{n}} \rightarrow  \oscr^{la}_{\infty} \otimes \mathfrak{n}^\vee$. Sur le quotient $\oscr^{la}_{\infty} \otimes \mathfrak{n}^\vee$  on a l'identité entre  actions de $\mathfrak{h}$ : pour tout $h \in \mathfrak{h}$,  $ h \star_{hor}(-) - 2\rho(h) (-) = h \star_2(-)$.
  \item Si $\chi( (1,-1)) \notin \Z_{\leq 0}$, le morphisme $ES_{w_0}$ induit   un isomorphisme  $$ \oscr^{la, \chi}_{w_0} \rightarrow \oscr^{la,\chi, \mathfrak{n}} $$
  l'action de $\star_2$ de $\mathfrak{h}$ se fait à travers le caractère $w_0\chi$.

  \item Si $\chi( (1,-1)) \in \Z_{\leq 0}$,  on possède un isomorphisme $\oscr^{la,\chi, \mathfrak{n}} =  \omega_{\mathcal{FL}}^{-w_0\chi, sm} \otimes_{\C_p} \C_p(w_0\chi)$,
   l'action de $\star_2$ de $\mathfrak{h}$ se fait à travers le caractère $w_0\chi$. Le morphisme $ES_{w_0}$ induit le morphisme $ES_{w_0} : j_!  \omega_{U_{w_0}}^{-w_0\chi, sm} \rightarrow \omega_{\mathcal{FL}}^{-w_0\chi, sm} \otimes_{\C_p} \C_p(w_0\chi)$.
  
  \item  Si $\chi( (1,-1)) \notin \Z_{\leq 0} \cup \mathrm{Liouv}$,  l'application $ES_\infty$ induit un isomorphisme de faisceaux 
  $$ES_\infty : \mathscr{O}^{la,\chi}_{ \mathfrak{n}}  \rightarrow \mathscr{O}^{la,\chi}_\infty \otimes \mathfrak{n}^\vee$$
  l'action de $\star_2$ de $\mathfrak{h}$ se fait à travers le caractère $\chi-2\rho$.
  \item  Si $\chi( (1,-1)) \in \Z_{\leq 0} $,  on possède une suite exacte   $$0 \rightarrow (i_{\infty})_\star \omega_{\infty}^{-\omega_0\chi,sm} \otimes_{\C_p} \C_p(\chi-w_0\chi) \rightarrow  \oscr^{la,\chi}_{ \mathfrak{n}}  \stackrel{ES_{w_0}}\rightarrow \mathcal{C}^{la,\chi}_\infty \otimes \mathfrak{n}^\vee \rightarrow 0.$$ L'action   $\star_2$ de $\mathfrak{h}$ se fait à travers le caractère $w_0\chi$ sur le premier facteur et  $\chi-2\rho$  sur  le second facteur.
  \end{enumerate}
  \end{prop}

\begin{proof} C'est la traduction de  la proposition \ref{prop-computingcoho} grace à une  application répétée du lemme ci-dessous.
\end{proof}
  
  \begin{lem} Soit $\mathscr{F}_1 \rightarrow \mathscr{F}_2 \rightarrow \mathscr{F}_3$ un complexe de faisceaux cohérent   de $\oscr^{sm}$-modules. Supposons que la suite :
 $$0 \rightarrow  \hat{\oscr}\otimes_{\oscr^{sm}}\mathscr{F}_1 \rightarrow \hat{\oscr} \otimes_{\oscr^{sm}}\mathscr{F}_2 \rightarrow \hat{\oscr} \otimes_{\oscr^{sm}}\mathscr{F}_3 \rightarrow 0$$ soit exact. Alors la suite   $0 \rightarrow \mathscr{F}_1 \rightarrow \mathscr{F}_2 \rightarrow \mathscr{F}_3 \rightarrow 0$ est exacte. 
  \end{lem}
   \begin{proof} On prend les vecteurs lisses dans $$0 \rightarrow  \hat{\oscr}\otimes_{\oscr^{sm}}\mathscr{F}_1 \rightarrow \hat{\oscr} \otimes_{\oscr^{sm}}\mathscr{F}_2 \rightarrow \hat{\oscr} \otimes_{\oscr^{sm}}\mathscr{F}_3 \rightarrow 0.$$ On déduit que $0 \rightarrow \mathscr{F}_1 \rightarrow \mathscr{F}_2 \rightarrow \mathscr{F}_3 $ est exacte. Comme   $\colim_{K_p} \HH^1(K_p, \hat{\oscr}) = VB((\mathfrak{n}^0)^\vee)$  est un faisceau inversible de $\oscr^{sm}$-modules, on déduit que 
  $ VB((\mathfrak{n}^0)^\vee) \otimes_{\oscr^{sm}} \mathscr{F}_2 \rightarrow  VB((\mathfrak{n}^0)^\vee) \otimes_{\oscr^{sm}}\mathscr{F}_3 $ est surjective puis que $\mathscr{F}_2 \rightarrow \mathscr{F}_3$ est surjective. 
  \end{proof}

\subsection{La $\mathfrak{n}$-cohomologie de $\oscr^{la, \chi}$}
  On rappelle que si $M$ est un $\C_p$-espace vectoriel muni d'une action $\star_{hor} $ du Cartan horizontal $\mathfrak{h}$ on note $M^{\chi} = \mathrm{Hom}_{\mathfrak{h}, \star_{hor}}( \chi, M)$. Si $M$ est un $\C_p$-espace vectoriel muni d'une action $\star_2 $ de $\mathfrak{h}$ on note $M[\chi] = \mathrm{Hom}_{\mathfrak{h}, \star_2}( \chi, M)$.

  \begin{coro}[\cite{pan2021locally}, thm. 5.3.11, 5.3.12, 5.3.16, 5.3.18] \label{coropan}
  
  \begin{enumerate}
 
  \item Si $\chi( (1,-1)) \notin \Z_{\leq 0} \cup \mathrm{Liouv}$, on possède une suite exacte :

 $$ 0 \rightarrow \HH^1_c( U_{w_0}, \omega_{U_{w_0}}^{-w_0 \chi, sm}) \rightarrow \mathrm{H}^1(\mathcal{FL}, \oscr^{la, \chi})^{\mathfrak{n}} \rightarrow \HH^0(\{\infty\},  \omega_{\infty}^{-w_0 \chi, sm})\otimes \mathfrak{n}^\vee \rightarrow 0$$
 
\item  Si $\chi( (1,-1)) \notin \Z_{\leq 0} \cup \mathrm{Liouv} \cup \{1\}$, l'action de $\star_2$ de $\mathfrak{h}$ est semi-simple et: 
\begin{eqnarray*}
 \mathrm{H}^1(\mathcal{FL}, \oscr^{la, \chi})^{\mathfrak{n}}[w_0\chi] & =& \HH^1_c( U_{w_0}, \omega_{U_{w_0}}^{-w_0 \chi, sm}) \\ 
\mathrm{H}^1(\mathcal{FL}, \oscr^{la, \chi})^{\mathfrak{n}}[\chi-2\rho] & =& \HH^0(\{\infty\},  \omega_{\infty}^{-w_0 \chi, sm})\otimes \mathfrak{n}^\vee 
\end{eqnarray*}

  \item  Si  $\chi( (1,-1)) = 1$ alors $w_0\chi = \chi -2\rho$ et l'action $\star_2$ de $\mathfrak{h}$ n'est pas semi-simple.  On possède un diagramme commutatif de suites exactes courtes :  
  
  \begin{eqnarray*}
  \xymatrix{ 0 \ar[r] & \HH^1_c( U_{w_0}, \omega_{U_{w_0}}^{-w_0 \chi, sm}) \ar[r] & \mathrm{H}^1(\mathcal{FL}, \oscr^{la, \chi})^{\mathfrak{n}} \ar[r] & \HH^0(\{\infty\},  \omega_{\infty}^{-w_0 \chi, sm})\otimes \mathfrak{n}^\vee \ar[r] & 0\\
  0 \ar[r] & \HH^1_c( U_{w_0}, \omega_{U_{w_0}}^{-w_0 \chi, sm}) \ar[r] \ar[u]  &  \mathrm{H}^1(\mathcal{FL}, \oscr^{la, \chi})^{\mathfrak{n}}[w_0\chi] \ar[r]\ar[u] & \HH^0(\mathcal{FL},  \omega_{\mathcal{FL}}^{-w_0 \chi, sm})\otimes \C_p(w_0\chi) \ar[r] \ar[u] & 0}
  \end{eqnarray*}

  \item Si $\chi( (1,-1)) \in \Z_{\leq -1} $,  on possède une suite exacte :

 $$ 0 \rightarrow \HH^1( \mathcal{FL}, \omega_{\mathcal{FL}}^{-w_0 \chi, sm})\otimes \C_p(w_0 \chi) \rightarrow \mathrm{H}^1(\mathcal{FL}, \oscr^{la, \chi})^{\mathfrak{n}} \rightarrow $$ $$ \HH^0(\{\infty\},  \omega_{\infty}^{-w_0 \chi, sm})\otimes \C_p(w_0\chi-\chi) \oplus  \HH^0(\{\infty\},  \omega_{\infty}^{-w_0 \chi, sm})\otimes \mathfrak{n}^\vee \rightarrow 0$$

 l'action de $\mathfrak{h}$ est semi-simple et : 
 
 $$ 0 \rightarrow \HH^1( \mathcal{FL}, \omega_{\mathcal{FL}}^{-w_0 \chi, sm})\otimes \C_p(w_0 \chi) \rightarrow \mathrm{H}^1(\mathcal{FL}, \oscr^{la, \chi})^{\mathfrak{n}}[w_0\chi] \rightarrow $$ $$ \HH^0(\{\infty\},  \omega_{\infty}^{-w_0 \chi, sm})\otimes \C_p(w_0\chi-\chi) \rightarrow 0$$
 
 $$ \mathrm{H}^1(\mathcal{FL}, \oscr^{la, \chi})^{\mathfrak{n}} [\chi - 2\rho] = \HH^0(\{\infty\},  \omega_{\infty}^{-w_0 \chi, sm})\otimes \mathfrak{n}^\vee.$$
 \item Si $\chi( (1,-1)) =0 $,  on possède une suite exacte :

 $$ 0 \rightarrow \HH^1( \mathcal{FL}, \omega_{\mathcal{FL}}^{-w_0 \chi, sm})\otimes \C_p(w_0 \chi) \rightarrow \mathrm{H}^1(\mathcal{FL}, \oscr^{la, \chi})^{\mathfrak{n}} \rightarrow $$ $$ \HH^0(\{\infty\},  \omega_{\infty}^{-w_0 \chi, sm})\otimes \C_p(w_0\chi-\chi) \oplus  \big(\HH^0(\{\infty\},  \omega_{\infty}^{-w_0 \chi, sm})/ \HH^0(\mathcal{FL}, \omega_{\mathcal{FL}}^{-w_0 \chi, sm}) \otimes \C_p(\chi)\big)\otimes \mathfrak{n}^\vee \rightarrow 0$$

 l'action de $\mathfrak{h}$ est semi-simple et : 
 
 $$ 0 \rightarrow \HH^1( \mathcal{FL}, \omega_{\mathcal{FL}}^{-w_0 \chi, sm})\otimes \C_p(w_0 \chi) \rightarrow \mathrm{H}^1(\mathcal{FL}, \oscr^{la, \chi})^{\mathfrak{n}}[w_0\chi] \rightarrow $$ $$ \HH^0(\{\infty\},  \omega_{\infty}^{-w_0 \chi, sm})\otimes \C_p(w_0\chi-\chi) \rightarrow 0$$
 
 $$ \mathrm{H}^1(\mathcal{FL}, \oscr^{la, \chi})^{\mathfrak{n}} [\chi - 2\rho] = \big(\HH^0(\{\infty\},  \omega_{\infty}^{-w_0 \chi, sm})/ \HH^0(\mathcal{FL}, \omega_{\mathcal{FL}}^{-w_0 \chi, sm})\otimes \C_p(\chi)\big) \otimes \mathfrak{n}^\vee.$$

\end{enumerate}
\end{coro}

\begin{proof} 

On considère la suite exacte :
 
 $$ 0 \rightarrow j_! j^\star \oscr^{la, \chi} \rightarrow \oscr^{la,\chi} \rightarrow (i_\infty)_\star i^{-1} \oscr^{la,\chi} \rightarrow 0$$
 
 En prenant la cohomologie, on obtient  une suite exacte : 
 $$0 \rightarrow \mathrm{H}^0(\mathcal{FL}, \oscr^{la,\chi}) \rightarrow \mathrm{H}^0( \{\infty\}, i_{\infty}^{-1} \oscr^{la,\chi}) \rightarrow \mathrm{H}^1_c( U_{w_0}, \oscr^{la,\chi})  \rightarrow \HH^1(  \mathcal{FL}, \oscr^{la,\chi}) \rightarrow 0$$
 Prenant la cohomologie de $\mathfrak{n}$, on déduit la suite exacte suivante :
 $$ 0 \rightarrow \mathrm{H}^0(\mathcal{FL}, \oscr^{la,\chi}) \rightarrow \mathrm{H}^0( \{\infty\}, i_{\infty}^{-1} \oscr^{la, \chi, \mathfrak{n}}) \rightarrow \mathrm{H}^1_c( U_{w_0}, \oscr^{la, \chi, \mathfrak{n}}) $$ $$ \rightarrow \HH^1(  \mathcal{FL}, \oscr^{la,\chi})^{\mathfrak{n}} \rightarrow \mathrm{H}^0( \{\infty\}, i_{\infty}^{-1} \oscr^{la,\chi}_{ \mathfrak{n}})/\mathrm{H}^0(\mathcal{FL}, \oscr^{la,\chi})\otimes \mathfrak{n}^\vee \rightarrow 0$$
  
Le corollaire est donc une conséquence directe de la proposition  \ref{prop-coho-compute-ola} à l'exception du point $3)$ et de la semi-simplicité de $\mathfrak{h}$ dans les points $4)$ et $5)$. 

Vérifions ces points.  Soit $h_0 = \mathrm{diag} (1,0)$. On pose  $\chi(h_0) = k_1$.   En appliquant le foncteur $\mathrm{Ext}_{ \C_ph^0, \star_2}(w_0\chi, -)$ à la suite exacte $$ 0 \rightarrow \HH^1_c( U_{w_0}, \omega_{U_{w_0}}^{-w_0 \chi, sm}) \rightarrow \mathrm{H}^1(\mathcal{FL}, \oscr^{la, \chi})^{\mathfrak{n}} \rightarrow \HH^0(\{\infty\},  \omega_{\infty}^{-w_0 \chi, sm})\otimes \mathfrak{n}^\vee \rightarrow 0$$ on obtient une application de bord : 
$d_0  : \HH^0(\{\infty\},  \omega_{\infty}^{-w_0 \chi, sm})\otimes \mathfrak{n}^\vee  \rightarrow \HH^1_c( U_{w_0}, \omega_{U_{w_0}}^{-w_0 \chi, sm})$ et on prétend que cette application est à un scalaire près l'application de bord $ d_1 : \HH^0(\{\infty\}, i_\infty^{-1} \omega_{\mathcal{FL}}^{-w_0 \chi, sm})  \rightarrow \HH^1_c( U_{w_0}, \omega_{\mathcal{FL}}^{-w_0 \chi, sm}).$ 
dans la suite exacte longue associée à la suite exacte courte:  
$$  0 \rightarrow j_! j^\star \omega_{\mathcal{FL}}^{-w_0 \chi, sm} \rightarrow \omega_{\mathcal{FL}}^{-w_0 \chi, sm} \rightarrow (i_\infty)_\star i_\infty^{-1} \omega_{\mathcal{FL}}^{-w_0 \chi, sm} \rightarrow 0$$

On peut calculer la cohomologie $\mathrm{H}^1(\mathcal{FL}, \oscr^{la, \chi})$ à la Chech en utilisant un recouvrement  $\mathcal{FL} = U_1 \cup U_2$, où $U_2$ est un voisinage ouvert affinoide arbitrairement petit de $\infty$. On pose $U_{1,2} = U_1 \cap U_2$. 

L'application de bord $d_0$  se décrit alors ainsi. Soit $f \in \HH^0(\{\infty\},  \omega_{\infty}^{-w_0 \chi, sm})\otimes \mathfrak{n}^\vee$, qu'on peut supposer définie sur $U_2$. On peut relever $f$ en une section $\tilde{f}\otimes n^\vee$ de  $\oscr^{la, \chi} \otimes \mathfrak{n}^\vee$ sur $U_2$.  On note $\mathrm{Res}(\tilde{f} \otimes n^\vee)$ sa restriction à l'ouvert $U_{1,2}$. On peut alors localement sur $U_{1,2}$ trouver une section $g$ qui vérifie $n.g = \mathrm{Res}(\tilde{f} \otimes n^\vee)$. Puis on considère $(h_0-k_1+1) g$ qui définit un élément de $\HH^1_c( U_{w_0}, \omega_{U_{w_0}}^{-w_0 \chi, sm})$. 

D'autre part, dans la section \ref{sect-poids-singulier}, on a défini un morphisme $d : \mathcal{C}^{la, -\chi}_{\mathfrak{n}}\otimes \mathfrak{n}^\vee \rightarrow i_\infty^{-1}j_\star j^\star \mathcal{C}^{la, -\chi, \mathfrak{n}}$. En tensorisant par $\hat{\oscr}$ on obtient un morphisme : 
$$\hat{\oscr} \hat{\otimes}_{\oscr_{\mathcal{FL}}} \mathcal{C}^{la, -\chi}_{\mathfrak{n}}\otimes \mathfrak{n}^\vee \rightarrow \hat{\oscr} \hat{\otimes}_{\oscr_{\mathcal{FL}}} i_\infty^{-1}j_\star j^\star \mathcal{C}^{la, -\chi, \mathfrak{n}}$$
puis en prenant les vecteurs lisses on déduit un morphisme: 
$VB(d) :  \omega_{\infty}^{-w_0 \chi, sm} \rightarrow i_{\infty}^{-1} \omega_{U_{w_0}}^{-w_0 \chi, sm}$. D'une part il est clair par construction que le morphisme $VB(d)$ représente le morphisme de bord $d_0$. D'autre part, d'après la proposition \ref{prop-bordmaps}, il représente également le morphisme $d_1$ à un scalaire près.  

Les points $4)$ et $5)$ se démontrent de la même façon. Il s'agit à nouveau  de calculer un morphisme de bord   $$d_0  : \HH^0(\{\infty\},  \omega_{\infty}^{-w_0 \chi, sm})\otimes \C_p(w_0\chi-\chi)  \rightarrow \HH^1( \mathcal{FL}, \omega_{\mathcal{FL}}^{-w_0 \chi, sm})\otimes \C_p(w_0 \chi) $$  associé à la suite exacte courte: 
$$ 0 \rightarrow \HH^1( \mathcal{FL}, \omega_{\mathcal{FL}}^{-w_0 \chi, sm})\otimes \C_p(w_0 \chi) \rightarrow \mathrm{H}^1(\mathcal{FL}, \oscr^{la, \chi})^{\mathfrak{n}}[w_0\chi] \rightarrow $$ $$ \HH^0(\{\infty\},  \omega_{\infty}^{-w_0 \chi, sm})\otimes \C_p(w_0\chi-\chi) \rightarrow 0.$$ en applicant  $\mathrm{Ext}_{ \C_ph^0, \star_2}(w_0\chi, -)$.
La différence par rapport au cas précédent  est  que l'on peut choisir un relêvement de $f\otimes n^\vee $, noté $\tilde{f} \otimes n^\vee$, qui est intégrable pour l'action de $\mathfrak{n}$ sur l'ouvert $U_2$ (et non seulement sur $U_{1,2}$). En effet, ceci se traduit en un calcul sur le faisceau $\mathcal{C}^{la,-\chi}$: comme on a vu dans la preuve du lemme \ref{lem-calculcoho}, la fonction $x^{-\chi((1,-1))+1}$ s'intègre en la fonction $x^{-\chi((1,-1))}$. Il en résulte que $g$ est un cobord et donc l'application $d_0$ est nulle. 
 \end{proof}

 \subsection{Théorie de Harish-Chandra}\label{sect-theorie-Harish-Chandra}
 
 Soit $\mathcal{Z}(\mathfrak{g})$ le centre de $\mathcal{U}(\mathfrak{g})$. On possède un morphisme 
 $ HC : \mathcal{Z}(\mathfrak{g}) \rightarrow \mathcal{U}(\mathfrak{h})$ dont voici la définition. Dans le module de Verma universel  $\mathcal{U}(\mathfrak{g}) \otimes_{\mathcal{U}(\mathfrak{b})} \mathcal{U}(\mathfrak{h})$ on a $$ z \otimes 1 = 1 \otimes HC(z).$$
 Le morphisme $HC$ induit un isomorphisme $\mathcal{Z}(\mathfrak{g}) \rightarrow \mathcal{U}(\mathfrak{h})^{W,.}$, où $W,.$ est l'action tordue du groupe de Weyl (qui induit l'action  $w. \chi = w(\chi + \rho)-\rho$ sur $X^\star(T)_{\C_p} = \Spec~\mathcal{U}(\mathfrak{h})$).

On note $\mathcal{D} = \oscr_{\mathcal{FL}} \otimes \mathcal{U}(\mathfrak{g})/ \mathfrak{n}^0( \oscr_{\mathcal{FL}} \otimes \mathcal{U}(\mathfrak{g}))$ le faisceau des opérateurs différentiels étendus sur $\mathcal{FL}$.

Le morphisme $\mathfrak{h} \rightarrow \mathfrak{g}^0/\mathfrak{n}^0$ induit un morphisme $\mathcal{U}(\mathfrak{h}) \rightarrow \mathcal{D}$. On possède de plus un morphisme $\mathcal{Z}(\mathfrak{g}) \rightarrow \mathcal{U}(\mathfrak{g}) \rightarrow \mathcal{D}$. On montre que le morphisme $\mathcal{Z}(\mathfrak{g}) \rightarrow \mathcal{D}$ se factorise à travers $\mathcal{U}(\mathfrak{h})$. On possède donc un second morphisme $HC' : \mathcal{Z}(\mathfrak{g}) \rightarrow \mathcal{U}(\mathfrak{h})$. 

\begin{lem} Le diagramme suivant est commutatif  (où $w_0 : \mathcal{U}(\mathfrak{h}) \rightarrow \mathcal{U}(\mathfrak{h})$ est l'action non tordue):
\begin{eqnarray*}
\xymatrix{ \mathcal{Z}(\mathfrak{g}) \ar[r]^{HC} \ar[rd]^{HC'} & \mathcal{U}(\mathfrak{h})\ar[d]^{w_0} \\
 & \mathcal{U}(\mathfrak{h})}
 \end{eqnarray*}
 \end{lem}
 \begin{proof} On a $i_{\infty}^\star \mathcal{D} = \mathcal{U}(\mathfrak{h}) \otimes_{\mathcal{U}(\mathfrak{b})} \mathcal{U}(\mathfrak{g})$ et on a par specialization la formule $1 \otimes z = HC'(z) \otimes 1$. 
 La transposition suivie de la conjugaison par $w_0$ induit un isomorphisme :
  $ \mathcal{U}(\mathfrak{g}) \otimes_{\mathcal{U}(\mathfrak{b})} \mathcal{U}(\mathfrak{h}) \rightarrow \mathcal{U}(\mathfrak{h}) \otimes_{\mathcal{U}(\mathfrak{b})} \mathcal{U}(\mathfrak{g})$.
  Cette opération  envoie $z \otimes 1 = 1 \otimes HC(z)$ sur $1 \otimes z = w_0 HC(z) \otimes 1$. 
 \end{proof}
 \subsection{Etude de la cohomologie de $\mathfrak{b}$ et décomposition d'Eichler-Shimura}
 
 On considère la suite exacte :
 
 $$ 0 \rightarrow j_! j^\star \oscr^{la} \rightarrow \oscr^{la} \rightarrow (i_\infty)_\star i^{-1} \oscr^{la} \rightarrow 0$$
 
 En prenant la cohomologie, on obtient  une suite exacte : 
 $$0 \rightarrow \mathrm{H}^0(\mathcal{FL}, \oscr^{la}) \rightarrow \mathrm{H}^0( \{\infty\}, i_{\infty}^{-1} \oscr^{la}) \rightarrow \mathrm{H}^1_c( U_{w_0}, \oscr^{la})  \rightarrow \HH^1(  \mathcal{FL}, \oscr^{la}) \rightarrow 0$$
 Prenant la cohomologie de $\mathfrak{n}$, on déduit la suite exacte suivante :
 
 $$ 0 \rightarrow \mathrm{H}^0(\mathcal{FL}, \oscr^{la}) \rightarrow \mathrm{H}^0( \{\infty\}, i_{\infty}^{-1} \oscr^{la, \mathfrak{n}}) \rightarrow \mathrm{H}^1_c( U_{w_0}, \oscr^{la, \mathfrak{n}}) $$ $$ \rightarrow \HH^1(  \mathcal{FL}, \oscr^{la})^{\mathfrak{n}} \rightarrow \mathrm{H}^0( \{\infty\}, i_{\infty}^{-1} \oscr^{la}_{ \mathfrak{n}})/\mathrm{H}^0(\mathcal{FL}, \oscr^{la}) \rightarrow 0$$
On peut en fait  considérer le complexe d'Eichler Shimura suivant : 
$$ES :  \mathrm{H}^1_c( U_{w_0}, \oscr^{la}_{w_0})  \stackrel{ES_{w_0}}\rightarrow \HH^1(  \mathcal{FL}, \oscr^{la})^{\mathfrak{n}} \stackrel{ES_{\infty}}\rightarrow \big(\mathrm{H}^0( \{\infty\},  \oscr^{la}_{\infty}) )/\mathrm{H}^0(\mathcal{FL}, \oscr^{la}))\otimes \mathfrak{n}^\vee $$
où on rappelle que (comme prédit par la théorie de Harish-Chandra, voir la section \ref{sect-theorie-Harish-Chandra}),  pour tout $h \in \mathfrak{h}$ : 
\begin{eqnarray*}
 h \star_2(-) &=& w_0 h \star_{hor}(-)~\textrm{sur $\mathrm{H}^1_c( U_{w_0}, \oscr^{la}_{w_0})$},\\
 h \star_2(-) &=& h \star_{hor}(-)-2\rho(h)~\textrm{sur~$\big(\mathrm{H}^0( \{\infty\},  \oscr^{la}_{\infty}) )/\mathrm{H}^0(\mathcal{FL}, \oscr^{la}))\otimes \mathfrak{n}^\vee$}. 
\end{eqnarray*}

Le complexe $ES$ est  "génériquement"  une suite exacte et donne une  décomposition de Sen-Eichler-Shimura de $\HH^1(  \mathcal{FL}, \oscr^{la})^{\mathfrak{n}}$ comme le montre le théorème suivant qui découle assez facilement du corollaire \ref{coropan} (on renvoie à \cite{pan2021locally} pour la démonstration et des compléments).

 
 \begin{thm}[\cite{pan2021locally}, thm. 5.4.2, 5.4.6] \begin{enumerate} 
 \item Si $\chi((1,-1)) \notin \ZZ_{\geq 0} \cup -\mathrm{Liouv} \cup \{-2\}$, le complexe $ES[\chi]$ induit une suite exacte : 
 $$ 0 \rightarrow \mathrm{H}^1_c( U_{w_0}, \omega^{-\chi, sm}_{U_{w_0}}) \rightarrow \HH^1(  \mathcal{FL}, \oscr^{la})^{\mathfrak{n}}[\chi] \rightarrow \mathrm{H}^0(\{\infty\}, \omega^{-w_0\chi + 2\rho,sm}_{\infty}) \otimes \mathfrak{n}^\vee \rightarrow 0$$
 
 \item Si $\omega_0 \chi = \chi +2\rho$ (i.e $\chi((1,-1)) = -1$), on a 
 
 $$ 0 \rightarrow \mathrm{H}^1_c( U_{w_0}, \omega^{-\chi, sm}_{U_{w_0}}) \rightarrow \HH^1(  \mathcal{FL}, \oscr^{la})^{w_0\chi, \mathfrak{n}}[\chi] \rightarrow \mathrm{H}^0(\mathcal{FL}, \omega^{-w_0\chi +2\rho ,sm}_{\mathcal{FL}})\otimes_{\C_p} \C_p(\chi)\rightarrow 0$$
 
  \item Si $\chi((1,-1)) \in \ZZ_{\geq 1}$,   le complexe $ES[\chi]$ induit un complexe (en degré $0,1,2$) : 
 
  $$ 0 \rightarrow \mathrm{H}^1_c( U_{w_0}, \omega^{-\chi, sm}_{U_{w_0}}) \rightarrow \HH^1(  \mathcal{FL}, \oscr^{la})^{\mathfrak{n}}[\chi] \rightarrow \mathrm{H}^0(\{\infty\}, \omega^{-w_0\chi+2\rho,sm}_{\infty})\otimes \mathfrak{n}^\vee \rightarrow 0$$
 et 
 \begin{eqnarray*}
 \HH^0(ES[\chi])& =& \mathrm{H}^0( \{\infty\}, \omega^{-\chi, sm}_{\infty}) \otimes_{\C_p} \C_p( \chi-w_0\chi) \\ 
 \HH^1(ES[\chi])& = &\mathrm{H}^0( \{\infty \}, \omega^{-\chi, sm}_{\infty})\otimes_{\C_p} \C_p( \chi-w_0\chi) \\ \HH^2(ES[\chi]) &=& 0.
 \end{eqnarray*}

  \end{enumerate}
 \end{thm}

\bibliographystyle{amsalpha}
\bibliography{eqsheaves}

\renewcommand{\MR}[1]{}
\providecommand{\bysame}{\leavevmode\hbox to3em{\hrulefill}\thinspace}
\providecommand{\MR}{\relax\ifhmode\unskip\space\fi MR }
\providecommand{\MRhref}[2]{%
  \href{http://www.ams.org/mathscinet-getitem?mr=#1}{#2}
}
\providecommand{\href}[2]{#2}
\begin{thebibliography}{{Sta}13}

\bibitem[AIS14]{MR3265287}
Fabrizio Andreatta, Adrian Iovita, and Glenn Stevens, \emph{Overconvergent
  modular sheaves and modular forms for {${\bf GL}_{2/F}$}}, Israel J. Math.
  \textbf{201} (2014), no.~1, 299--359. \MR{3265287}

\bibitem[BB83]{MR733805}
Alexander Be\u{\i}linson and Joseph Bernstein, \emph{A generalization of
  {C}asselman's submodule theorem}, Representation theory of reductive groups
  ({P}ark {C}ity, {U}tah, 1982), Progr. Math., vol.~40, Birkh\"{a}user Boston,
  Boston, MA, 1983, pp.~35--52. \MR{733805}

\bibitem[BC08]{MR2493221}
Laurent Berger and Pierre Colmez, \emph{Familles de repr\'esentations de de
  {R}ham et monodromie {$p$}-adique}, Ast\'erisque (2008), no.~319, 303--337,
  Repr{\'e}sentations $p$-adiques de groupes $p$-adiques. I.
  Repr{\'e}sentations galoisiennes et $(\phi,\Gamma)$-modules. \MR{2493221
  (2010g:11091)}

\bibitem[BC16]{MR3552018}
\bysame, \emph{Th\'{e}orie de {S}en et vecteurs localement analytiques}, Ann.
  Sci. \'{E}c. Norm. Sup\'{e}r. (4) \textbf{49} (2016), no.~4, 947--970.
  \MR{3552018}

\bibitem[BP20]{Boxer-Pilloni}
George Boxer and Vincent Pilloni, \emph{Higher hida and coleman theories on the
  modular curves}, preprint, 2020.

\bibitem[CHJ17]{MR3609197}
Przemys\l~aw Chojecki, D.~Hansen, and C.~Johansson, \emph{Overconvergent
  modular forms and perfectoid {S}himura curves}, Doc. Math. \textbf{22}
  (2017), 191--262. \MR{3609197}

\bibitem[Fal87]{MR909221}
Gerd Faltings, \emph{Hodge-{T}ate structures and modular forms}, Math. Ann.
  \textbf{278} (1987), no.~1-4, 133--149. \MR{909221}

\bibitem[Gro57]{MR102537}
Alexander Grothendieck, \emph{Sur quelques points d'alg\`ebre homologique},
  Tohoku Math. J. (2) \textbf{9} (1957), 119--221. \MR{102537}

\bibitem[Hub94]{MR1306024}
R.~Huber, \emph{A generalization of formal schemes and rigid analytic
  varieties}, Math. Z. \textbf{217} (1994), no.~4, 513--551. \MR{1306024}

\bibitem[Jan03]{MR2015057}
Jens~Carsten Jantzen, \emph{Representations of algebraic groups}, second ed.,
  Mathematical Surveys and Monographs, vol. 107, American Mathematical Society,
  Providence, RI, 2003. \MR{MR2015057 (2004h:20061)}

\bibitem[KST20]{MR4166998}
Moritz Kerz, Shuji Saito, and Georg Tamme, \emph{Towards a non-archimedean
  analytic analog of the {B}ass-{Q}uillen conjecture}, J. Inst. Math. Jussieu
  \textbf{19} (2020), no.~6, 1931--1946. \MR{4166998}

\bibitem[Pan21]{pan2021locally}
Lue Pan, \emph{On locally analytic vectors of the completed cohomology of
  modular curves}, 2021.

\bibitem[Pil13]{MR3097946}
Vincent Pilloni, \emph{Overconvergent modular forms}, Ann. Inst. Fourier
  (Grenoble) \textbf{63} (2013), no.~1, 219--239. \MR{3097946}

\bibitem[Sch12]{MR3090258}
Peter Scholze, \emph{Perfectoid spaces}, Publ. Math. Inst. Hautes \'{E}tudes
  Sci. \textbf{116} (2012), 245--313. \MR{3090258}

\bibitem[Sch13]{MR3090230}
\bysame, \emph{{$p$}-adic {H}odge theory for rigid-analytic varieties}, Forum
  Math. Pi \textbf{1} (2013), e1, 77. \MR{3090230}

\bibitem[Sch15]{scholze-torsion}
\bysame, \emph{On torsion in the cohomology of locally symmetric varieties},
  Ann. of Math. (2) \textbf{182} (2015), no.~3, 945--1066. \MR{3418533}

\bibitem[{Sta}13]{stacks-project}
The {Stacks Project Authors}, \emph{\itshape {S}tacks {P}roject},
  \url{http://stacks.math.columbia.edu}, 2013.

\end{thebibliography}

\end{document}